\documentclass[12pt, reqno]{amsart}
\usepackage{amsfonts, amssymb}
\usepackage{amsmath, amscd}
\usepackage{bbm}
\usepackage{array, longtable}
\newcolumntype{C}{>{$}c<{$}}  
\newcolumntype{R}{>{$}r<{$}}  
\newcolumntype{L}{>{$}l<{$}}  
\usepackage{color}
\usepackage[colorlinks=true, pdfstartview=FitV, linkcolor=blue, citecolor=blue, urlcolor=blue, backref]{hyperref}
\usepackage{tikz}
\usepackage{tikz-cd}
\usepackage{graphicx}
\usepackage{mathtools}
\usepackage{enumitem}
\usetikzlibrary{arrows, decorations.markings}
\usetikzlibrary{decorations.markings, arrows.meta}
 \usepackage[all]{xy}
\newlength{\bibitemsep}
\newlength{\bibparskip}
\let\oldthebibliography\thebibliography
\renewcommand\thebibliography[1]{%
  \oldthebibliography{#1}%
  \setlength{\parskip}{\bibitemsep}%
  \setlength{\itemsep}{\bibparskip}%
}

\usepackage{mathrsfs}
\usepackage[scr=boondox]{mathalfa}

\theoremstyle{plain}
\newtheorem{theorem}{Theorem}[section]
\newtheorem{lemma}[theorem]{Lemma}
\newtheorem{proposition}[theorem]{Proposition}

\theoremstyle{definition}
\newtheorem{remark}[theorem]{Remark}
\newtheorem{definition}[theorem]{Definition}

\newcommand{\cc}{\mathsf{k}}

\newcommand{\nc}{\newcommand}
\nc{\Y}{\rotatebox[origin=c]{180}{$Y$}}
\nc{\D}{\reflectbox{$D$}}
\usepackage[X2,T1]{fontenc}
\DeclareSymbolFont{cyrillic}{X2}{cmr}{m}{n}
\SetSymbolFont{cyrillic}{bold}{X2}{cmr}{bx}{n}
\DeclareMathSymbol{\R}{\mathord}{cyrillic}{223}
\nc{\rnc}{\renewcommand}

\nc{\on}[1]{\operatorname{#1}}
\nc{\delete}[1]{}
\nc{\onbf}[1]{\on{\bf #1}}
\nc{\onsf}[1]{\on{\sf #1}}
\nc{\red}{\color{red}}
\nc{\blue}{\color{blue}}
\nc{\bb}[1]{\mathbb{#1}}
\nc{\mbf}[1]{\mathbf{#1}}
\nc{\mc}[1]{\mathcal{#1}}
\nc{\mf}[1]{\mathfrak{#1}}
\nc{\msf}[1]{\mathsf{#1}}

\font\cyr=wncyr10
\nc{\sha}{{\mbox{\cyr X}}}

\nc{\udim}{\underline{\on{dim}}}
\nc{\guvw}{g_{U,V}^{W}}

\nc{\vac}{\mathbbm{1}}

\nc{\htt}{\on{ht}}
\nc{\gr}{\on{\bold {gr}}}
\nc{\cf}{\on{cf}}
\nc{\ad}{\on{ad }}
\nc{\Ad}{\on{Ad}}
\nc{\id}{\on{Id}}
\nc{\Fr}{\on{Fr}}
\nc{\Der}{\on{Der}}
\nc{\End}{\on{End}}
\nc{\tor}{\on{Tor}}
\nc{\Ext}{\on{Ext}}
\nc{\ext}{\on{ext}}
\nc{\Fil}{\on{Fil}}
\nc{\Hom}{\on{Hom}}
\nc{\grhom}{\on{grHom}}
\nc{\ch}{\on{ch}}
\nc{\Ch}{\on{\bf Ch}}
\nc{\ind}{\on{Ind}}
\nc{\coind}{\on{Coind}}

\nc{\Mod}{\on{-Mod}}
\nc{\biMod}{\on{-biMod}}
\nc{\Poiss}{\on{-Poiss}}
\nc{\grmod}{\on{-grMod}}
\nc{\vamod}{\on{-VAMod}}
\nc{\res}{\on{Res}}
\nc{\soc}{\on{Soc}}
\nc{\rad}{\on{Rad}}
\nc{\Aut}{\on{Aut}}
\nc{\Dist}{\on{Dist}}
\nc{\Lie}{\on{Lie}}
\nc{\Ker}{\on{Ker}}
\nc{\im}{\on{Im}}
\nc{\wt}{\on{wt}}
\nc{\st}{\on{St}}
\nc{\diag}{\on{Diag}}
\nc{\rep}{\on{rep}}
\nc{\Set}{\on{\bf Set}}
\nc{\sSet}{\on{\bf sSet}}
\nc{\smCat}{\on{\bf smCat}}
\nc{\spec}{\on{spec}}
\nc{\Sym}{\on{Sym}}
\nc{\Vecs}{\on{\bf Vec}^{s}}
\nc{\colim}{\operatornamewithlimits{\underset{\longrightarrow}{lim}}}

\nc{\graphdot}{\onwithlimits{\bullet}}
\nc{\lrarrows}{\onwithlimits{\leftrightarrows}}

\nc{\interval}[1]{\mathinner{#1}}
\nc{\blist}{\begin{list}{\rom{(\roman{enumi})}}{\setlength{\leftmargin}{0em}
\setlength{\itemindent}{7ex}
\setlength{\labelsep}{2ex}\setlength{\listparindent}{\parindent}
\usecounter{enumi}}}
\nc{\elist}{\end{list}}
\nc{\subsub}[1]{\noindent{\bf #1}}

\usepackage{etoolbox}
\apptocmd{\thebibliography}{\setlength{\itemsep}{5pt}}{}{}

\allowdisplaybreaks
\usepackage{empheq}

\makeatletter \def\l@subsection{\@tocline{2}{0pt}{1pc}{5pc}{}} \def\l@subsection{\@tocline{2}{0pt}{2pc}{6pc}{}} \makeatother

\begin{document}

\title[Equivariant vertex coalgebras]{Equivariant vertex coalgebras, $C_2$-coalgebras and duality for diagonalisable group schemes}
\author[A. Caradot]{Antoine Caradot$^1$}
\address{$^1$Hubert Curien Laboratory \\ Jean Monnet University \\ Saint-\'Etienne \\ FRANCE}
\email{antoine.caradot@univ-st-etienne.fr}
\author[Z. Lin]{Zongzhu Lin$^2$}
\address{$^2$Department of Mathematics\\
Kansas State University \\
Manhattan, KS 66506, USA}
\email{zlin@ksu.edu}


\date{\today}


\maketitle

\begin{abstract}
In this paper, we define vertex algebras and vertex coalgebras in the category of rational $G_\Gamma$-modules, where $G_\Gamma$ is the group scheme defined by the group algebra $\msf k \Gamma$ for an abelian group $\Gamma$. In this context, we introduce the notion of $C_2$-coalgebra for a vertex coalgebra. We prove that there exists a duality between vertex algebras and vertex coalgebras in the category of $G_\Gamma$-modules, and this duality establishes a connection between $C_2$-algebras and $C_2$-coalgebras. Moreover, we also investigate the relationship between their respective modules/comodules.
\end{abstract}

\tableofcontents
\addtocontents{toc}


 \section{Introduction}
 The purpose of this paper is fourfold. The first objective is to serve as a model for defining a vertex algebra over an algebraic stack. 
 In this case, we have a diagonalisable algebraic $\msf k$-group scheme $G_\Gamma$, defined by an abelian group $ \Gamma$, acting trivially on the point $\spec(\msf k)$.
 Here $ \msf k$ is any field of characteristic zero. In this case, $\Gamma$ is the group of characters of $ G_\Gamma$.
 The category of quasi-coherent sheaves is nothing but the category of (possibly infinite dimensional) comodules of the group algebra $\msf k \Gamma$. 
 This category is a symmetric monoidal category and $\Gamma$ is the Picard group of this category. 
 The second objective is to deform a vertex algebra by a map $\beta: \Gamma\times \Gamma\to \msf k $ in order to define a $(G_\Gamma, \beta)$-vertex algebra. In the case where $\beta$ is a 2-cocycle (with values in $\msf k^\times$), $\beta$ can be thought of as a deformation of the braiding of the symmetric monoidal category. This approach has been explored by Borcherds in his effort to define quantum vertex algebras in \cite{Bo3}. 
 But here the map $ \beta$ has values in $ \msf k$ which are not necessarily invertible.  
 For example, when $\beta$ takes constant value zero,  the Jacobi identity degenerates to the weak associative case and the $(G_\Gamma, \beta)$-vertex algebras are closely related to the field algebras studied by Bakalov and Kac \cite{Bakalov-Kac} and the open string vertex algebras studied by Huang \cite{Huang}. 
 The third objective is to study the duality between vertex algebras and vertex coalgebras as well as the associated $C_2$-(co)algebras. 
 Vertex operator coalgebras and the duality were studied by Hubbard \cite{Hub1, Hub2}.
 It is well-known \cite{Bulacu-Caenepeel-Panaite-VanOystaeyen} that in a symmetric rigid tensor category, the category of algebras and the category of coalgebras are isomorphic via the duality functor.
 But the category of $G_\Gamma$-modules is not rigid. It contains a full closed rigid tensor subcategory of finite dimensional modules (compact objects). 
 The last goal is to serve as a model to formulate vertex algebras and vertex coalgebras, as well as their representation theory, 
 in a closed symmetric monoidal category with a goal to formulate derived vertex algebras over a derived algebraic stack. 
 
 There has been a long list of literature using both categorical approaches \cite{Bo3} and geometric  approach \cite{BD, BF} to understand vertex algebras, in addition to the classical approach of using operator product expansions formulated in \cite{FLM}.  
 A good way to understand vertex operators or fields   is in terms of operator valued distributions outlined in \cite{Kac}. 
 There are two issues from this point of view. One is the source of the distribution (the source analytic spaces) and the other is the linear spaces on which the operators act. The approaches in \cite{BD, BF, Bo2} deal with the source spaces. The work of Borcherds in \cite{Bo3} actually deals with both the source and the target linear spaces in a very general categorical setting.  The focus of our work is mainly on the target linear spaces, which are closely related to derived algebraic varieties. The work in \cite{Caradot-Jiang-Lin-4, Caradot-Jiang-Lin-5} is on the category of differential complexes on which the operators act. In this paper, 
 we focus on the category of rational modules over a diagonalisable algebraic group scheme $G$. One can easily extend to the category of differential complexes over this category following the setting of \cite{Caradot-Jiang-Lin-4, Caradot-Jiang-Lin-5}.

A group $G$ acting on the linear space on which the operators act can  be regarded as the group automorphisms of a vertex algebra.
Borcherds extended the source space to higher dimensions in \cite{Bo2} in terms of vertex groups, so the classical 1-dimensional case corresponds to the vertex group $G_1$ which is the additive group $G_a$ with a singularity at $0$. The vertex group $G_1$ roughly consists of the Hopf algebra of distributions (algebra of invariant differential operators \cite[Ch. 7]{Jantzen:book}), the completion of the coordinate algebra at $0$, and the algebra of rational functions with singularities at $0$. Li in \cite{Li3} has given an axiomatic definition of  vertex algebras in this setting.  One can also make the group $G$ act on the vertex group so that the state-field correspondences are compatible with the $G$-action. This is the main motivation of the setting in this paper. In \cite{JKLT1,JKLT3, Li1}, the equivariant vertex algebras (and the generalization to quantum vertex algebras) deals with an abstract group $G$ as well as a character $\chi$ of $G$ which determines the action of $G$ on the vertex group $G_1$. When the group $G$ is a cyclic of finite order, the corresponding representations in this context are exactly twisted representations of a vertex algebra. It should be mentioned that there is another generalization of the action on the vertex group $G_1$ expressed as a generalization of the locality in \cite{GKK, Li2}, which replaces the polynomials $(x_1-x_2)^N$ by $\prod_{i=1}^N(x_1-\alpha_i x_2)$ with $\alpha_i$ being distinct invertible elements. 

In addition to the notion of vertex coalgebra by Hubbard \cite{Hub1, Hub2, Hub3}, 
there is also another different version of vertex coalgebras defined in \cite{HLX}. 
The vertex coalgebra here is actually a coalgebra object in the category of vertex algebras, which is a symmetric monoidal category \cite{JL}. 
Therefore naturally the notion of vertex bialgebras in this context have been used to study the deformation theory of vertex algebras in \cite{JKLT2}. 
But we follow Borcherds' approach to vertex algebras to define vertex coalgebras in term of distributions on the vertex group $ G_1$ with values in $\Hom_{\msf k}(V, V\otimes V)$, in comparison to vertex algebras as distributions with values in $ \Hom_{\msf k}(V\otimes V, V)=\Hom_{\msf k}(V, \Hom_{\msf k}(V, V))$. 

 This paper will not spell out the details of a  differential graded version of vertex algebras and vertex coalgebras in the category of $G_\Gamma$-modules.  One can follow the setting in \cite{Caradot-Jiang-Lin-4} without any substantial difficulty. 
 The vertex Lie coalgebra is not treated as well. 
 This can be done more or less by following \cite{Caradot-Jiang-Lin-5} and applying the duality in this paper. 
 Instead, most of these settings  will be uniformly treated in a categorical approach in an upcoming paper \cite{Caradot-Lin-7}.  
 The categorical approach has the advantage that it can be applied in many different situations, both geometrically and cohomologically. 
 In \cite{Caradot-Jiang-Lin-1, Caradot-Jiang-Lin-2} the cohomological properties of representations of vertex algebras were discussed. In particular, the cohomology theory of the $C_2$-algebra provides geometric invariants to a vertex algebra. 
 It is expected that the same can be done for $G_\Gamma$-equivariant vertex algebras as well as $G_\Gamma$-equivariant representations. 
 The corresponding cohomological varieties would carry natural $ G_\Gamma$-actions. 
 Many other topics related to coalgebra representation theory such as contra-modules for vertex coalgebras as well as their relations to the dual vertex algebra representations are much more involved and require topological structures with respect to operator topology. 
 Another topic to be treated is the dual version of the Zhu algebra for vertex operator coalgebras. 
 In this case, we expect to define Zhu coalgebras as well as the functors between the various comodules of the vertex coalgebras and of the Zhu coalgebras.
 
Given an algebraic group $\msf k$-scheme $G$, the category of the finite-dimensional rational representations is a closed rigid symmetric monoidal category. Indeed, it is a Tannakian category, which characterizes the algebraic  (super)group $\msf k$-scheme from Deligne's Tannakian duality philosophy. However, given two infinite-dimensional representations $V$ and $W$, the linear space $ \Hom_{\msf k}(V, W)$ is no longer a $G$-module \cite[2.7]{Jantzen:book} 

 \delete{
 {\color{red} (WHY?).
 
With our definition of a rational $G_\Gamma$-module $V$, $V$ is $\Gamma$-graded, and thus $\Hom_{\msf k}(V, W)=\prod_{\gamma \in \Gamma}\Hom_{\msf k}(V_\gamma, W)$. 
It follows that if $U$ is infinite dimensional, then the $\on{Hom}_\cc$ might have infinitely many homogeneous components, and hence not be $\Gamma$-graded and so not be rational. 

 Humphreys' (and Jantzen's) definition of a rational $G$-module $V$ is a morphism of algebraic groups $G \longrightarrow \on{GL}(V)$ (see Humphreys - Linear Algebraic Groups pages 55 and 60).
 If $U,V$ are rational infinite dimensional $G$-modules, why is $\on{Hom}_\cc(U,V)$ not rational?
 }
 }
 
 In this paper, for $G$ diagonalisable, we define an internal hom $G$-module $\mc Hom(V, W)$ for any two rational $G$-modules $ V$ and $W$, providing $G\Mod$ with adjunction formulas (see Equations \eqref{eq:tensor_hom_duality} and \eqref{eq:tensor_hom_duality2}).    
 If the group $G$ acts on the affine line $\bb A^1$ linearly, which is determined by a character $\gamma_0$,  
 the action also extends to the local function field $\msf k((t))$ at $0$ with $t$ being a local coordinate. 
 But $\msf k((t))$ is not a rational $G$-module and it contains a dense rational $ G$-submodule $\msf k[t, t^{-1}]$. The work in this paper would easily extends to more general toric varieties. 

 In conformal field theory, it is important to keep the action of the algebra of differential operators of $ \bb A^1$ at $0$ (see \cite{Bo2} in terms of vertex groups). 
We would require the $G_\Gamma$-modules to be equipped with a compatible action of the algebra of differential operators.  
The derivation $ \frac{d}{dt}$ is not a $G_\Gamma$-homomorphism, but it is an element in $ \mc Hom(\msf k[t,t^{-1}],\msf k[t,t^{-1}]) $.  
In this setting, a vertex algebra structure would be a $G_\Gamma$-equivariant map $ V\longrightarrow \mc Hom(\msf k[t,t^{-1}], \mc Hom_{\msf k}(V,V))$ satisfying the $\beta$-Jacobi identity together with a vacuum element. 

Our definition of a vertex coalgebra is slightly different from that in \cite{Hub1,Hub2, Hub3, Hub4}. 
Hubbard used the duality over $\bb P^1$ to have the singularity at $\infty$ while we keep the singularity at $0$ following the spirit of \cite{Bo2}. 
 This is for the purpose of unifying the categorical approach in the spirit 
 that vertex operators are operator valued distributions at $0$ of $\bb A^1$ \cite{Bo2, Bo3, Kac}. 
 Thus the vertex algebra structure on a vector space $V$ is a distribution, i.e., an element of 
$\Hom^{\on{cont}}_{G_\Gamma}(\msf k[t,t^{-1}], \mc Hom(V\otimes V, V))$, and a vertex coalgebra 
 structure on a vector space $V$ is also a distribution at $0$ on $ \bb A^1$, i.e., an element of $\Hom^{\on{cont}}_{G_\Gamma}(\msf k[t,t^{-1}], \mc Hom(V, V\otimes V))$. 
 This approach is easily applied to more general closed tensor categories provided an operator topology is defined on the internal hom space $\mc Hom(V, W)$ for any two objects $V$ and $ W$.  

 Given a vertex coalgebra, one can naturally discuss the dual version of the $C_2$-algebra, i.e., 
 the $C_2$-coalgebra, which should be a co-Poisson coalgebra, as well as the comodules for vertex coalgebra. 
 This is natural as one can expect to be able to extend most of the results of \cite{Sweedler} in any closed symmetric monoidal. However, to consider the duality between vertex algebras and vertex coalgebras in general,
 compact objects are expected to satisfy the duality condition similar to the Poincare duality for compact spaces. 
 But in the category of $G_\Gamma$-modules, the compactness and those objects $V$ satisfying $(V')'=V$ with $ V'=\mc Hom(V, \mbf 1)$ are different. 
 In fact, the latter are called Harish-Chandra in the sense of representations of Lie algebras of Lie groups, meaning that when restricted to the maximal compact subgroup, they have finitely many components for each irreducible representation. 
 Thus our work in this paper can also be applied to Harish-Chandra module categories of Lie groups, which we will not discuss further.

 The paper is outlined as follows. In the Section \ref{sec:2}, we give the basic setting of the category of rational $G_\Gamma$-modules (which we will simply call as $G_\Gamma$-modules). 
 We construct the closed tensor product structure as well as the operator topology on the internal hom spaces. 
 In Section \ref{sec:va}, we define $(G_\Gamma, \beta, \gamma_0)$-vertex algebras/coalgebras (Def. \ref{def:va} and \ref{def:cova}).
 Here $ \gamma_0$ is the character defining the action of $G_\Gamma$ action on $ \bb A^1$. 
 When $ \Gamma=\bb Z/l\bb Z$, the group $ G_\Gamma$ is a cyclic group of order $l$. 
 Thus the definitions of vertex algebra and vertex coalgebra also include twisted vertex algebra and twisted vertex coalgebra by choosing an appropriate $\beta$ with $\gamma_0$ being the generator of $ \Gamma$. 
 Then we prove the duality between vertex algebras and vertex coalgebras provided they satisfy the Harish-Chandra condition (Thm. \ref{thm:duality}). Section \ref{sec:va_mod} 
 is devoted to the modules/comodules of $(G_\Gamma, \beta, \gamma_0)$-vertex algebras/coalgebras (Def. \ref{def:va_mod} and \ref{def:cova_mod}),
 as well as the duality existing between them (Thm. \ref{thm:duality_mod}). In Section \ref{sec:C2}, we look at the $C_2$-algebra of a $(G_\Gamma, \beta, \gamma_0)$-vertex algebra and determine some of its properties (Thm. \ref{thm:C_2_Poisson}) as well as those of its associated modules (Thm. \ref{thm:C_2_Poisson_mod}). We define the $C_2$-coalgebra of a $(G_\Gamma, \beta, \gamma_0)$-vertex coalgebra in Section \ref{sec:coC2} and proceed similarly as in Section \ref{sec:C2} (Thm. \ref{thm:C_2_co_Poisson} and \ref{thm:C_2_co_Poisson_mod}). Finally, in Section \ref{sec:duality}, we look at the duality between $C_2$-algebras and $C_2$-coalgebras (Thm. \ref{thm:Poisson_iso} and \ref{thm:co_Poisson_iso}), as well as the duality between their respective modules/comodules (Thm. \ref{thm:Poisson_iso_mod} and \ref{thm:co_Poisson_iso_mod}).

\medbreak

{\bf Acknowledgment.} This work was initiated when the first author was visiting Kansas State University in the fall of 2023. 
He expresses his appreciation to the department of mathematics for providing support and a research environment.  

\delete{why do we add $\beta$? to deform the vertex algebra, also deforms the Poisson and the Lie structure (example with $\beta=1$ and $\beta=0$).

line bundle=invertible object in the tensor category of vector bundles.    Set of invertibles = Picard group

Lax equation}

\section{The category of representations of diagonalisable groups}\label{sec:2}
\subsection{The category of rational $G_\Gamma$-modules}\label{sec:A.1}Let $ \Gamma$  be an abelian group and $ \msf k$ be a field. A $ \Gamma$-graded $\msf k$-vector space is of the form $V=\bigoplus_{\gamma \in \Gamma}V_\gamma$. 
A homomorphism $ f: U=\bigoplus_{\gamma \in \Gamma}U_\gamma\longrightarrow V=\bigoplus_{\gamma \in \Gamma}V_\gamma$ 
satisfies $f(U_\gamma)\subseteq V_\gamma$ for all $ \gamma\in \Gamma$.  

Let $ G_\Gamma=\onsf{spec}(\msf k\Gamma)$ be the affine algebraic group scheme over $ \msf k$ 
defined by the commutative Hopf algebra $ \msf k\Gamma$, 
the group algebra of $\Gamma$. We will write the basis elements of $ \cc \Gamma$ by $ e^{\gamma}$ so that $e^\gamma e^{\gamma'}=e^{\gamma+\gamma'}$ is the multiplication in $ \cc \Gamma$. Then a $\Gamma$-graded $\msf k$-vector space is nothing but a comodule of the coalgebra $ \msf k \Gamma$, with the comodule structure $V\longrightarrow V\otimes \cc \Gamma$ 
defined by $ v\longmapsto \sum_{\gamma}v_{\gamma} \otimes e^\gamma$ where $v=\sum_\gamma v_\gamma$ with $v_\gamma \in V_\gamma$. The $\cc \Gamma$-comodules are called \textbf{rational} $ G_\Gamma$-\textbf{modules} (\cite{Jantzen:book}). We denote this category by $ G_\Gamma\Mod$. Let $ \Hom_{G_\Gamma}(V,W)$ be the space of $ G_\Gamma$-module homomorphisms between two rational $ G_\Gamma$-modules (from now on we will simply say $G_\Gamma$-modules).
Then we have $ \Hom_{G_\Gamma}(V,W) =\prod_{\gamma\in\Gamma}\Hom_{\cc}(V_\gamma, W_\gamma)$. The category $G_\Gamma\Mod$ is an abelian category with arbitrary (small) direct limit (see \cite[I 2.11]{Jantzen:book}). 

The category  $G_\Gamma\Mod$ has 
$\{\msf k_\gamma\;|\; \gamma\in \Gamma\}$ parametrising the isomorphism classes of irreducible modules with the trivial $G_\Gamma$-module $ \cc=\cc_0$ being the tensor identity.
So $V=\bigoplus_{\gamma\in \Gamma}\Hom_{G_\Gamma}(\cc_\gamma, V) \otimes \cc_\gamma$.  In this paper we will freely use either $G_\Gamma$-modules or $\Gamma$-graded vector spaces for objects in $G_\Gamma\Mod$. Note that $ \Gamma=\Hom_{gp \, sch}(G_\Gamma, G_m)$ is the group of characters and the decomposition $V=\bigoplus_{\gamma \in \Gamma}V_\gamma$ is exactly the weight space decomposition of $V$, where the $G_\Gamma$-action is given by $g.v=\gamma(g)v$ for all $g \in G_\Gamma$ and $v \in V_\gamma$. Since we will also be taking different gradations, we avoid using the word ``degree'' to designate $\gamma$ whenever possible to avoid any possible confusion. 

The category $G_\Gamma\Mod$  is a closed symmetric monoidal category with tensor product 
\[V\otimes W=\bigoplus_{\gamma \in \Gamma}\left(\bigoplus_{\gamma'\in \Gamma}\left(V_{\gamma-\gamma'}\otimes W_{\gamma'}\right)\right)
\]
and the standard braiding of the category of $ \cc$-vector spaces $ T: V\otimes W\to W\otimes V$ given by $T(v\otimes w)=w\otimes v$.
The internal hom-space is the object $\mc Hom(V,W)$ in $G_\Gamma\Mod$ defined by 
\[
\mc Hom(V,W)_{\gamma}=\prod_{\gamma'\in \Gamma}\Hom_{\msf k}(V_{\gamma'}, W_{\gamma'+\gamma})=\Hom_{G_\Gamma}(V\otimes \msf k_{\gamma}, W).
\]
Thus we have 
\[
\mc Hom(V,W)=\bigoplus_{\gamma\in \Gamma}\Hom_{G_\Gamma}(V\otimes \msf k_{\gamma}, W).
\]
Note that $ \msf k_\gamma\otimes \msf k _{\gamma'}=\msf k_{\gamma+\gamma'}$ 
and $ \mc Hom(\msf k_\gamma,\msf k_{\gamma'})=\msf k_{\gamma'-\gamma}$. Therefore $\Gamma$ is the Picard group of $ G_\Gamma\Mod$, the groups of all invertible objects. 

\delete{
By \cite{Sweedler}, for any  coalgebra $ C$, the dual space $C^*= \Hom_{\cc}(C, \cc)$ is an associative algebra.  
 For each (right) $ C$-comodule $ M$ with $\delta: M\longrightarrow M\otimes C$, $M$ is a left $ C^*$-module with $ C^*\otimes M\longrightarrow M$ defined by $ fw= \sum_{i}w_i f(c_i)$ where $ \delta(w)=\sum_i w_i\otimes c_i$. 

This defines a functor $\iota_C:\onsf{Comod-}\!C\longrightarrow C^*\Mod$ which is faithful and exact. 
For a general coalgebra $C$, describing the subcategory $\onsf{Comod-}\!C $ in $C^*\Mod$ is a difficult question without certain finiteness conditions. For $C=\cc \Gamma$, We have a group homomorphism $\Gamma \longrightarrow \Hom_{Alg}(C^*, \cc)$ defined by evaluation at $ e^\gamma$ for each $ \gamma$. For each $C^*$-module $V$, set 
\[V_\gamma=\{v\in V\; |\; fv=f(e^\gamma)v \; \text{ for all } f\in C^*\}
\]
and $ \mc F(V)=\bigoplus_{\gamma\in \Gamma}V_\gamma$. Then $\mc F(V)$ is a $C$-comodule with the coaction $\delta: \mc F(V)\to \mc F(V)\otimes C$ defined by $ \delta(v)=\sum_{\gamma}v_\gamma\otimes e^\gamma$ if $ v=\sum_{\gamma} v_{\gamma}\in \mc F(V)$. The functor $\mc F: C^*\Mod\longrightarrow \onsf{Comod-}\!C$ is right adjoint to the functor $\iota_C$.

Let $ M, N$ be two objects in $G_\Gamma\Mod$. The linear space $ \Hom_{\cc}(M,N)$ is a rational $ G_\Gamma$-module if $ M$ is finite dimensional. In general case, for any submodule $M'\subseteq M$ such that $ M/M'$ is finite dimensional, then $\Hom_{\cc}(M/M', N)\subseteq \Hom_{\cc}(M,N)$ and the direct limit 
$\varinjlim_{M'}$ {\color{red}???????????????????????}
}

The following natural isomorphisms of functors can be easily verified. For any $ U, V, W \in G_\Gamma\Mod$
\begin{align}\label{eq:tensor_hom_duality}
\Hom_{G_\Gamma}(U\otimes V, W)&\cong\Hom_{G_\Gamma}(U, \mc Hom(V, W))
\end{align}
as vector spaces over $ \msf k$.
Then using the Yoneda embedding, one can show that
\begin{align}\label{eq:tensor_hom_duality2}
\mc Hom(U\otimes V, W)&=\mc Hom(U, \mc Hom(V, W))
\end{align}
 as  $ G_\Gamma$-modules.
 
As a consequence of \eqref{eq:tensor_hom_duality} and \eqref{eq:tensor_hom_duality2},
we have, for all $E, U, V, W $ in $G_\Gamma\Mod$,
\begin{align}
    \Hom_{G_\Gamma}(E\otimes U\otimes V, W)&=\Hom_{G_\Gamma}(E,\mc Hom(U, \mc Hom(V, W)));\\
 \mc Hom(E\otimes U\otimes V, W))&=\mc Hom(E,\mc Hom(U, \mc Hom(V, W))).   
\end{align}
Here the first isomorphism is in the category of $\msf k\Mod$ and the second isomorphism is the category $G_\Gamma\Mod$.

Although the identity $\Hom_{G_\Gamma}(\varinjlim_{\alpha}V^\alpha, W)=\varprojlim_{\alpha}\Hom_{G_\Gamma}(V^\alpha, W)$ holds, one does not have 
$\mc Hom_{G_\Gamma}(\varinjlim_{\alpha}V^\alpha, W)=\varprojlim_{\alpha}\mc Hom_{G_\Gamma}(V^\alpha, W)$. 
The reason is that the category $ G_\Gamma\Mod$ is not closed under inverse limit in general.

Given any rational $G_\Gamma$-module $V$, 
its linear dual $ V^*=\Hom_{\msf k}(V, \msf k)$ in general is not a rational $ G_{\Gamma}$-module, 
but the subspace 
\begin{align}\label{deg_restricted_dual}
V'=\mc Hom(V,\msf k)=\bigoplus_{\gamma}
\Hom_{G_\Gamma}(V\otimes \msf k_{\gamma}, \msf k)=\bigoplus_{\gamma}
\Hom_{\msf k}(V_{-\gamma}, \msf k)
\end{align}
is a rational $G_\Gamma$-module and $(V')_{\gamma}=(V_{-\gamma})^*$. 
This is what we will call \textbf{restricted dual}. While we have an evaluation map $ \on{ev}_{V}: V'\otimes V\longrightarrow \msf k$, the coevaluation map $\on{coev}_V: \msf k\longrightarrow V\otimes V'$ is not defined in general. 
We also have a natural homomorphism $ V'\otimes W'\longrightarrow (V\otimes W)'$ which is not an isomorphism in general. Moreover, one has a natural homomorphism $V\longrightarrow (V')'$, which is not an isomorphism unless $V_\gamma$ is finite dimensional for all $ \gamma$. 
We will call such a $G_\Gamma$-module of {\em Harish-Chandra} type. 
We will not go more in depth about this category although the full subcategory $G_\Gamma\on{-mod}$ of finite dimensional rational $ G_\Gamma$-modules has a rigid closed symmetric very tensor category structure. 
The full subcategory of all $G_\Gamma$-modules of Harish-Chandra type is not closed under tensor product. 

\begin{definition}
A  $G_\Gamma$-module $V$ is called \textbf{Harish-Chandra} if $\on{dim}V_\gamma < +\infty$ for all $\gamma \in \Gamma$. A  $G_\Gamma$-module $V$ is called \textbf{compact} if $V$ is Harish-Chandra and $V_\gamma=0$ for all but finitely many $ \gamma$. Thus $G_\Gamma\on{-mod}$ is consists of all compact objects. 
\end{definition}

We remark that given two Harish-Chandra $G_\Gamma$-modules $ V$ and $W$, the tensor product $V\otimes W$ and the internal hom space $ \mc Hom(V,W)$ are not  Harish-Chandra, unless one of them is compact.  However $ V$ is Harish-Chandra if and only if $V'=\mc Hom(V, \cc)$ is Harish-Chandra.
 The category of all compact $\Gamma$-graded vector spaces is a Tannakian category.
For a $\Gamma$-graded vector space $V=\bigoplus_{\gamma \in \Gamma}V_\gamma$, the pairing with its restricted dual is written as $\langle\cdot, \cdot \rangle : V' \otimes V \longrightarrow \cc$. For any $\Gamma$-graded vector spaces $U$ and $V$, we have a map (not necessarily an isomorphism)
\[
V' \otimes U' \longrightarrow (U \otimes V)'
\]
given by the pairing
\[
\langle f \otimes g, u \otimes v \rangle=g(u)f(v)
\]
for any $u \in U$, $v \in V$, $f \in V'$, and $g \in U'$. If $V$ is compact, then $ \mc Hom(V, W)=V'\otimes W$ for all $W$.

\subsection{The loop algebra and continuous functions}\label{sec:A.2} We note  that $\cc[t, t^{-1}]$ is the function algebra  of the punctured affine line 
$\bb A^1\setminus\{0\}$ equipped with the $t$-adic linear topology so the open neighbourhoods of $0$ are  $t^n\cc[t]$, 
$n \in \mathbb{Z}$. 
For any fixed $ \gamma_0\in \Gamma$, set $ \cc t=\cc_{\gamma_0}$ as a $G_\Gamma$-module. 
So $\cc t^n=\cc_{n\gamma_0}$. Then $\cc[t,t^{-1}]$ is a  $ G_\Gamma$-algebra defining a rational action of $G_\Gamma$ on $ \bb A^1\setminus \{0\}$.

Let $\cc((t))=\varprojlim_{n}\cc[t,t^{-1}]/t^n\cc[t]$. This space of power series $\cc((t))$ is not a $G_\Gamma$-module unless $\gamma_0$ has finite order. 
We compute $ \mc Hom(\cc[t,t^{-1}], \cc)$, which depends on $\gamma_0$. 
If $\gamma_0$ has finite order, say, $N=o(\gamma_0)$, then $\cc[t,t^{-1}]=\bigoplus_{i=0}^{N-1}\cc[t^N, t^{-N}]t^i$ and \[ \mc Hom(\cc[t,t^{-1}], \cc)=\bigoplus_{i=0}^{N-1} \Hom_{\cc}(\cc[t^N, t^{-N}]t^i, \cc).
\]

If $ \gamma_0$ is torsion free (i.e., not of finite order), then 
\[\mc Hom(\cc[t,t^{-1}], \cc)=\bigoplus_{n \in \mathbb{Z}}\Hom_{G}(\cc t^n,\cc_{n\gamma_0}).
\]

\delete{
We specify the pairing $ \cc[t,t^{-1}]\otimes \cc[t, t^{-1}]\longrightarrow \cc_{-\gamma_0}$ as follows. The differential operator $ \frac{d}{dt}$ acts on $ \cc[t,t^{-1}]$ naturally. 
Let $ \langle-,-\rangle: \cc[t,t^{-1}]\otimes \cc[t,t^{-1}]\to \cc_{-\gamma_0} =\msf kt^{-1}$ be defined by $ \langle f(t),g(t)\rangle=\on{Res}_t(f(t)g(t))$. 
Then the pairing is a homomorphism of $ G_\Gamma$-modules.  Since $\on{Res}_t(\frac{d}{dt}(f(t)g(t)))=0$
we get $\langle \frac{d}{dt}f(t),g(t)\rangle=-\langle f(t),\frac{d}{dt}g(t)\rangle$. \delete{Thus we want $\cc[x, x^{-1}]\longrightarrow \mc Hom(\cc[t,t^{-1}], \cc)$ to be an isomorphism as modules of the Lie algebra $ \onsf{Lie}(\bb G_a)$.} Then the pairing is given by $ \langle t^m, t^n\rangle =\delta_{n+m, -1}$ and is a module homomorphism for the Lie algebra $ \onsf{Lie}(\bb G_a)$. We write $\cc[x, x^{-1}]=\mc Hom(\cc[t, t^{-1}], k_{-\gamma_0})$ with $ \mc Hom(\cc t^n, \cc_{-\gamma_0})=\cc x^{-n-1}$  as $G_\Gamma$-modules. In particular, $\cc x^{-n-1}=\cc_{-n\gamma_0}\otimes \cc_{-\gamma_0} $ .
}

For a $\cc$-vector space $V$, we have
\[
\on{Hom}_{\cc}(\cc[t, t^{-1}], V) =\prod_{n\in \bb Z}\Hom_{\cc}(\cc t^n, V) =\prod_{n\in \bb Z}(\cc t^n)'\otimes V= V[[x, x^{-1}]]
\]
by seeing $\cc x^{n}=(\cc t^{-n-1})'=\mc Hom(\cc t^{-n}, \cc t)$ as $G_\Gamma$-modules.  We see that $|x^n|=n\gamma_0+\gamma_0$ as 
\begin{align*}
\begin{array}{rcl}
(\cc x^n)_\gamma &= &\mc Hom(\cc t^{-n-1}, \cc)_\gamma \\[5pt]
&= &\on{Hom}_\cc((\cc t^{-n-1})_{-\gamma}, \cc) \quad \quad \text{by Equation \eqref{deg_restricted_dual}}, \\[5pt]
\end{array}
\end{align*}
which is non zero if and only if $\gamma=-|t^{-n-1}|=n\gamma_0+\gamma_0$. 
It follows that $|x^n|=n|x|-(n-1)\gamma_0 \neq n|x|$. 
Hence $x \otimes x \neq x^2$ in $G_\Gamma\Mod$. In particular, $|x^{-1}|=0$.

Define the twisted tensor product $\stackrel{\gamma_0}{\otimes}$ on the category $ G_\Gamma\Mod$ by $U \stackrel{\gamma_0}{\otimes} V=(U \otimes \cc_{-\gamma_0}) \otimes (V \otimes \cc_{-\gamma_0}) \otimes \cc_{\gamma_0}$ (see \cite{Briggs-Witherspoon, Grimley-Nguyen-Witherspoon}). Then we have
\[
\cc x^m \stackrel{\gamma_0}{\otimes} \cc x^n=\cc x^{m+n}
\]
as $|x^m \stackrel{\gamma_0}{\otimes} x^n|=|x^m|-\gamma_0+|x^n|-\gamma_0+\gamma_0=(m+n)\gamma_0+\gamma_0=|x^{m+n}|$.

We note that if $V$ is a $G_\Gamma$-module, then $V[[x,x^{-1}]]$ is not a $\Gamma$-graded vector space. We also have
\[
\mc Hom(\cc[t,t^{-1}], V)
=\bigoplus_{\gamma\in\Gamma}\prod_n\Hom_{G_\Gamma}(\cc t^n, V\otimes \cc_{-\gamma})
=\bigoplus_{\gamma\in\Gamma}\prod_{n\in \bb Z}V_{\gamma+n\gamma_0}\subseteq V[[x,x^{-1}]].
\]

The linear map $D=-\frac{d}{dt}: \cc[t,t^{-1}]\longrightarrow \cc [t,t^{-1}] $ is not a homomorphism of $ G_\Gamma$-modules. However, $D\in \mc Hom(\cc [t,t^{-1}], \cc[t,t^{-1}])_{-\gamma_0}$ and $D$ also induces a map
\[
\tilde{D}:\Hom_{\cc}(\cc[t,t^{-1}], V)\longrightarrow \Hom_{\cc}(\cc[t,t^{-1}], V) 
\]
given by $\tilde{D}(f)=-f\circ D$. 
Then $\tilde{D}=\frac{d}{dx}$ under the identification of $ \cc x^{n}=(\cc t^{-n-1})'=\mc Hom(\cc t^{-n-1}, \cc)$ 
and we have $\tilde{D}(\mc Hom(\cc [t,t^{-1}], V))\subseteq \mc Hom(\cc [t,t^{-1}], V)$.

We then write $\on{Hom}^{\on{cont}}(\cc[t, t^{-1}], V)$ for the set of \textbf{continuous morphisms}. 
Based on the above equality, it is a subspace of $V[[x, x^{-1}]]$ and we can show that
\[
 \on{Hom}_{\cc}^{\on{cont}}(\cc[t,t^{-1}], V)=\varinjlim_n\Hom_{\cc}(\cc[t,t^{-1}]/t^n\cc[t], V)=\{f:\cc[t, t^{-1}] \longrightarrow V \ | \  f(t^n)=0 \ \text{for} \ n \gg 0  \};
\]

\[
 \on{Hom}_{G_\Gamma}^{\on{cont}}(\cc[t,t^{-1}], V)=\varinjlim_n\Hom_{G_\Gamma}(\cc[t,t^{-1}]/t^n\cc[t], V)=\{f:\cc[t, t^{-1}] \stackrel{G_\Gamma\on{-hom}}{\longrightarrow} V \ | \  f(t^n)=0 \ \text{for} \ n \gg 0  \};
\]

\begin{align*}
 \mc Hom^{\on{cont}}(\cc[t,t^{-1}], V) &
 =\displaystyle \varinjlim_n \mc Hom(\cc[t,t^{-1}]/t^n\cc[t], V) =\displaystyle\bigoplus_{\gamma \in \Gamma}\varinjlim_n \on{Hom}_{G_\Gamma}\big((\cc[t, t^{-1}]/t^n\cc[t]) \otimes k_\gamma, V\big)\\
 &=\displaystyle\bigoplus_{\gamma \in \Gamma}\{f:\cc[t, t^{-1}] \stackrel{G_\Gamma\text{-hom}}{\longrightarrow} V \otimes \cc_{-\gamma} \ | \  f(t^n)=0 \ \text{for} \ n \gg 0  \}  =V((x)).
\end{align*}
We note that $ \mc Hom^{\on{cont}}(\cc[t,t^{-1}], V)$ is an object in $ G_\Gamma\Mod$ and 
$\tilde{D}(\mc Hom^{\on{cont}}(\cc [t,t^{-1}], V))\subseteq \mc Hom^{\on{cont}}(\cc [t,t^{-1}], V)$. In particular, for any $U,V,W \in G_\Gamma\Mod$, we have
\[
\mc Hom^{\on{cont}}(\cc[t,t^{-1}], \mc Hom(U \otimes V,W))=\mc Hom(U \otimes V, \mc Hom^{\on{cont}}(\cc[t,t^{-1}],W)).
\]

\delete{
For a formal variable $x$ with $|x|=2\gamma_0 \in \Gamma$, although the space of formal series $V[[x, x^{-1}]]$ is not in $G_\Gamma\Mod$, it has a subspace $\mc Hom(\cc[t,t^{-1}], V)$ in $G_\Gamma\Mod$ where $|V_\gamma x^n|=\gamma+n\gamma_0+\gamma_0$.
}
\section{Vertex (co)algebras}\label{sec:va}

Let $(\Gamma, +)$ be an abelian group and $\beta: \Gamma \times \Gamma \longrightarrow \cc$ a map. 
For $U, V \in G_\Gamma\Mod$, we define a $G_\Gamma$-module homomorphism $T^\beta :U \otimes V \longrightarrow V \otimes U$ by $T^\beta(u \otimes v)=\beta(|u|, |v|) v \otimes u$ on the homogeneous components and extend by linearity. 
\delete{
A $\Gamma$-graded vector space $V$ equipped with $T^\beta$ is called a $(\Gamma, \beta)$-\textbf{graded vector space}.
}
For later use, we introduce a few conditions that $\beta$ may satisfy, which will facilitate the statements of later results:
\begin{align}
&\bullet \ \beta(\gamma_1, \gamma_2)=\beta(-\gamma_1, -\gamma_2) \text{ for all }\gamma_1, \gamma_2 \in \Gamma. \label{relation_1} \\[5pt]
&\bullet \  \beta(\gamma_1, \gamma_2)\beta(\gamma_1, \gamma_3)=\beta(\gamma_1, \gamma_2+\gamma_3) \text{ for all } \gamma_1, \gamma_2, \gamma_3 \in \Gamma. \label{relation_2} 
\end{align}

\begin{remark}
Using the map $T^\beta$ will allow for a more general Jacobi identity in the definition of a vertex (co)algebra. 
The category $G_\Gamma\Mod$ is a symmetric tensor category with the standard braiding. 
The $G_\Gamma$-module homomorphism $T^\beta$ does not define a new braiding in general since we do not impose the hexagon identities for  $T^\beta$. If we were to assume that $\beta$ is a $2$-cocyle, i.e.
\[
\beta(\gamma_1, \gamma_2)\beta(\gamma_1+\gamma_2, \gamma_3)=\beta(\gamma_2, \gamma_3)\beta(\gamma_1, \gamma_2+\gamma_3)
\]
for all $\gamma_1, \gamma_2, \gamma_3 \in \Gamma$, then $T^\beta$ would satisfy the cactus relation
\[
T^\beta_{V \otimes U, W}   (T^\beta_{U, V} \otimes \on{id}_W)=T^\beta_{U, W \otimes V}   (\on{id}_U \otimes T_{V, W}).
\]
This would make the category of $G_\Gamma$-modules close to a coboundary category (cf. \cite[Section 4]{Savage}).
\end{remark}

The notion of vertex algebra was introduced in \cite{Borcherds} and \cite{FLM}. We extend it below to the category $G_\Gamma\Mod$ with a parameter $\beta$ (see Section \ref{sec:A.1} for precisions on the notations).

\begin{definition}\label{def:va}
A $(G_\Gamma, \beta, \gamma_0)$-\textbf{vertex algebra} (over $\cc$) is a $G_\Gamma$-module $V$ equipped with a  map $Y(x)\in \Hom_{G_\Gamma}(V\otimes V, \mc Hom(\cc [t,t^{-1}], V))$ with $\cc t=\cc_{\gamma_0}$ as $G_\Gamma$-module:
\begin{align*}
\begin{array}{cccc}
Y(x): & V \otimes V & \longrightarrow & V[[x,x^{-1}]] \\
           & u \otimes v & \longmapsto      & \displaystyle \sum_{n \in \mathbb{Z}}Y_n(u\otimes v) \, x^{-n-1}
\end{array}
\end{align*}
called the vertex operator map, and an element $\vac \in V_0$ called the vacuum vector, satisfying the following axioms:
\begin{enumerate}[leftmargin=*, itemsep=5pt, label=(\roman*)]
\item\label{def:vacuum} Vacuum: for all $v \in V$,
\begin{align}\label{eq:vacuum}
Y(x)(\vac \otimes v)=v.
\end{align}

\item\label{def:creation} Creation: for all $v \in V$,
\begin{align}\label{eq:creation}
\begin{array}{cc}
Y(x)(v \otimes \vac) \in V[[x]]  \quad \text{ and} \quad 
\displaystyle \lim_{x \to 0} Y(x)(v \otimes \vac)=v.
 \end{array}
\end{align}

\item\label{def:truncation} Truncation: for any $u' \in V'$ and $v, w \in V$, we have
\[
\langle u', Y(x)(v \otimes w) \rangle \in \mc Hom^{\on{cont}}(\cc[t, t^{-1}], \cc).
\]

\item\label{def:Jacobi} Jacobi identity:
\begin{equation}\label{eq:Jacobi}
\begin{split}
\scalebox{0.92}{$ \displaystyle  x_0^{-1}\delta \left(\frac{x_1-x_2}{x_0}\right)Y(x_1)(\on{id}_V \otimes Y(x_2))-x_0^{-1}\delta \left(\frac{x_2-x_1}{-x_0}\right)Y(x_2)(\on{id}_V \otimes Y(x_1))(T^\beta \otimes \on{id}_V)$} \\[5pt]
\scalebox{0.92}{$\displaystyle =x_2^{-1}\delta \left(\frac{x_1-x_0}{x_2}\right)Y(x_2)(Y(x_0) \otimes \on{id}_V)$} 
 \end{split}
 \end{equation} 
 as  $\msf k$-linear maps $V^{\otimes 3} \longrightarrow V[[[x_0^{\pm1}, x_1^{\pm1}, x_2^{\pm1}]]$.

 \item\label{def:derivation} Derivation properties: set $D=\on{Res}_x\big(x^{-2} Y(x)   (\on{id}_V \otimes \vac)\big): V  \longrightarrow V$. Then we have
\begin{empheq}[left=\empheqlbrace]{align} 
D Y(x)-Y(x)(\on{id}_V \otimes D)  =  \frac{d}{dx}Y(x), \label{eq:derivation_1}\\[5pt]
Y(x)   (D \otimes \on{id}_V) =\frac{d}{dx}Y(x). \label{eq:derivation_2}
\end{empheq}

\delete{
\begin{align}\label{eq:derivation}
D  Y(x)-e^{xD}Y(-x)(D \otimes \on{id}_V)T^\beta= Y(x)   (D \otimes \on{id}_V)=\frac{d}{dx}Y(x),
\end{align}
with $e^{xD}=\sum_{n \geq 0}\frac{1}{n!}x^n D^n$.
}

\end{enumerate}
\end{definition}

In the remainder of the paper, we will write $\vac$ for the vacuum vector seen as an element of $V_0=\Hom_{G_\Gamma}(\cc, V)$.

We give the following lemma which is likely well-known for the experts, and we provide a proof as we could not find it in the literature.

\begin{lemma}\label{lem:truncation}
Let $V$ be a $\cc$-vector space with $V^*$ the dual vector space, and consider a sequence $\{v_n\}_{n \in \mathbb{Z}}$ of elements in $V$. Let $z$ be a formal variable. Then
\[
\sum_{n \in \mathbb{Z}}v_n z^{-n-1} \in V((z)) \Longleftrightarrow \langle w^*, \sum_{n \in \mathbb{Z}}v_n  z^{-n-1}\rangle \in \cc((z)) \text{ for all } w^* \in V^*.
\]
\end{lemma}

\begin{proof}
One direction is obvious. To prove the other direction, set $W=\msf k\on{-Span}\{v_n \ | \ n \geq 0\}$. 

If $\on{dim}W < +\infty$, let $\{w_1, \dots, w_k\}$ be a basis of $W$. 
Hence for any $n \geq 0$, there exist $c_i^n \in \cc$ ($1 \leq i \leq k$) such that $v_n=\sum_{i=1}^k c_i^n w_i$. 
Let $\{w_1^*, \cdots, w_k^*\}\subseteq W^*\subseteq V^*$ (by extending the basis of $W$ to that of $V$) be the dual basis of $\{w_1, \cdots, w_k\}$.  
Assume that $\langle w^*, \sum_{n \in \mathbb{Z}}v_n  z^{-n-1}\rangle \in \cc((z))$ for all $w^* \in V^*$. 
For any $1 \leq i \leq k$ we have $\sum_{n \in \mathbb{N}}c_i^n x^{-n-1}=\langle w_i^*, \sum_{n \in \mathbb{N}}v_n  z^{-n-1}\rangle \in \cc[z^{-1}]$, hence there exists $n_i \in \mathbb{N}$ such that $c_i^n=0$ for $n \geq n_i$. Set $N=\on{max}\{n_i \ | \ 1 \leq i \leq k\}$. Then for any $1 \leq i \leq k$, $c_i^n=0$ for $n \geq N$ and thus $v_n=0$ for $n \geq N$. It follows that $\sum_{n \in \mathbb{Z}}v_n z^{-n-1} \in V((z))$.

If $\on{dim}W = +\infty$, then $\sum_{n \in \mathbb{Z}}v_n z^{-n-1} \notin V((z))$. Thus there exists an infinite sequence $0 \leq n_1 < n_2 < \dots $ such that the set $\{v_{n_j}\}_{j \geq 1}$ is  linearly independent in $V$. Let $w^* \in V^*$ be such that $\langle w^*,v_{n_j}\rangle=1$ for all $j$. Then $\langle w^*, \sum_{n \in \mathbb{Z}}v_n  z^{-n-1}\rangle \notin \cc((z))$, which concludes the proof.
\end{proof}

\begin{remark}
\begin{enumerate}[leftmargin=*, wide, itemsep=5pt]

\item \label{rem:point1} Using Lemma \ref{lem:truncation} on each homogeneous component $V_\gamma$ of the $G_\Gamma$-module $V$, we see that the truncation condition \ref{def:truncation} in Definition \ref{def:va} is equivalent to saying that $Y(x)(v \otimes w) \in \mc Hom^{\on{cont}}(\cc[t, t^{-1}], V)$ for all $v, w \in V$. In particular, when $\Gamma=\{0\}$, it is equivalent to the usual truncation condition given in \cite[(3.1.2)]{Lepowsky-Li}. We could thus remove the truncation condition \ref{def:truncation} by imposing $Y(x)\in \Hom_{G_\Gamma}(V\otimes V, \mc Hom^{\on{cont}}(\cc [t,t^{-1}], V))$, but we will see in Theorem \ref{thm:duality} why our choice for \ref{def:truncation} is well-suited for this context. We also want to mention that the condition for $w*\in V^*$ cannot be weakened to $ w^*$ being in a  dense subspace. 

\delete{
The truncation condition \ref{def:truncation} is weaker than the usual truncation for vertex algebra (cf. \cite[Definition 3.1.1]{Lepowsky-Li}). Indeed, in the usual definition, the truncation condition means that the morphism $Y(x)$ has image in $\on{Hom}^{\on{cont}}(\cc[t, t^{-1}], V)$. If we use elements in $V$ to write down the condition, then for any $u, v \in V$, there exists $n_{u, v} \in \mathbb{Z}$ (depending on $u$ and $v$) such that $u_n v=0$ for $n \geq n_{u, v}$. In \ref{def:truncation}, for any $u, v \in V$ and $u' \in V'$, there exists $n_{u, v, u'} \in \mathbb{Z}$ (depending on $u$, $v$ and $u'$) such that $\langle u', u_n v\rangle =0$ for $n \geq n_{u, v, u'}$. So the integer $n_{u, v, u'}$ depends not only on $u$ and $v$, but also on $u'$.
}

\item Writing $ Y(x)=\sum_{n\in \mathbb Z}Y_n x^{-n-1}$ with $Y_n : V\otimes V\to V\otimes \cc x^{-n-1}$ being a $G_\Gamma$-module homomorphism, the Jacobi identity \eqref{eq:Jacobi} can be written in the component form

\begin{align}\label{eq:Jacobi_components}
\sum_{i \geq 0}(-1)^i \binom{l}{i} Y_{m+l-i}  (\on{id}_V  \otimes Y_{n+i})- (-1)^l & \sum_{i \geq 0} (-1)^i \binom{l}{i} Y_{n+l-i} (\on{id}_V \otimes Y_{m+i})(T^\beta \otimes \on{id}_{V} ) \notag \\
&  =\sum_{i \geq 0} \binom{m}{i} Y_{m+n-i}  (Y_{l+i} \otimes \on{id}_V)
\end{align}
for all $ l, m, n \in \mathbb Z$ (as the coefficients of $x_0^{-l-1}x_1^{-m-1}x_2^{-n-1}$).  Taking $\on{Res}_{x_0}$ (i.e., $l=0$), one gets
\begin{align}
Y_m   (\on{id}_V \otimes Y_n)-Y_n    (\on{id}_V \otimes Y_m)    (T^\beta\otimes \on{id}_V)=\sum_{i\geq 0}\binom{m}{i} Y_{m+n-i}    (Y_{i}\otimes \on{id}_V).
\end{align}
Taking $\on{Res}_{x_1}$ (i.e., $m=0$), one gets
\begin{align}
 Y_n   (Y_l\otimes \on{id}_V) =& \sum_{i\geq 0}(-1)^i\binom{l}{i}Y_{l-i}    (\on{id}_V \otimes Y_{n+i}) \notag\\
& -(-1)^l \sum_{i\geq 0}(-1)^i\binom{l}{i}Y_{n+l-i}    (\on{id}_V \otimes Y_{i})    (T^\beta \otimes \on{id}_V).
\end{align}
\item Using Point \eqref{rem:point1}, we see that if $\beta=1$ is the constant, then the axioms \ref{def:vacuum}-\ref{def:Jacobi} in Definition \ref{def:va} are the same as those of the usual definition of a vertex algebra.

\item If $\beta=0$, then $T^\beta=0$ and the Jacobi identity \eqref{eq:Jacobi} becomes
\begin{equation}\label{eq:Jacobi_beta=0}
 x_0^{-1}\delta \left(\frac{x_1-x_2}{x_0}\right)Y(x_1)(\on{id}_V \otimes Y(x_2)) =x_2^{-1}\delta \left(\frac{x_1-x_0}{x_2}\right)Y(x_2)(Y(x_0) \otimes \on{id}_V).
 \end{equation} 
 In component form, this is
 \begin{align}\label{eq:associative_components}
\sum_{i\geq 0}(-1)^i\binom{l}{i}Y_{m+l-i}   (\on{id}_V \otimes Y_{n+i})=&\sum_{i\geq 0}\binom{m}{i} Y_{m+n-i}   (Y_{l+i}\otimes \on{id}_V)
\end{align} 
for all $l, m, n \in \mathbb Z$.
Taking $\on{Res}_{x_1}$ (i.e., $m=0$), one gets
\begin{align}\label{eq:associative_components_1}
Y_n    (Y_l\otimes \on{id}_V)=\sum_{i\geq 0}(-1)^i\binom{l}{i} Y_{l-i}   (\on{id}_V \otimes Y_{n+i}),
\end{align}
 which is exactly
\begin{equation}\label{eq:associative}
Y(x_2)(Y(x_0) \otimes \on{id}_V)=Y(x_0+x_2)(\on{id}_V \otimes Y(x_2)).
 \end{equation} 
It follows that the weak associativity of the vertex algebra is in fact a strong associativity. So the case $\beta=0$ produces an associative vertex algebra. Moreover, we see that Equation \eqref{eq:associative} is exactly the associativity axiom with $N=0$ of a field algebra in \cite{Bakalov-Kac}.  
 
\delete{ 
(4) One could change the Jacobi identity to the following 
\begin{equation}\label{eq:Jacobi2}
\begin{split}
\scalebox{0.9}{$ \displaystyle  x_0^{-1}\delta \left(\frac{x_1-x_2}{x_0}\right)Y(x_1)(\on{id}_V \otimes Y(x_2))(T^\beta \otimes \on{id}_V)
-x_0^{-1}\delta \left(\frac{x_2-x_1}{-x_0}\right)Y(x_2)(\on{id}_V \otimes Y(x_1))$} \\[5pt]
\scalebox{0.9}{$\displaystyle =x_2^{-1}\delta \left(\frac{x_1-x_0}{x_2}\right)Y(x_2)(Y(x_0) \otimes \on{id}_V)$.} 
 \end{split}
 \end{equation} 
 Taking $\beta=0$, the equation \eqref{eq:Jacobi2} becomes
 \begin{align}-x_0^{-1}\delta \left(\frac{x_2-x_1}{-x_0}\right)Y(x_2)(\on{id}_V \otimes Y(x_1))=x_2^{-1}\delta \left(\frac{x_1-x_0}{x_2}\right)Y(x_2)(Y(x_0) \otimes \on{id}_V)
 \end{align}
 
 Replacing $x_0$ by $ -x_0$ we get 
 \begin{align}x_0^{-1}\delta \left(\frac{x_2-x_1}{x_0}\right)Y(x_2)(\on{id}_V \otimes Y(x_1))=x_2^{-1}\delta \left(\frac{x_1+x_0}{x_2}\right)Y(x_2)(Y(-x_0) \otimes \on{id}_V)
 \end{align}
Can one apply the same argument of the implication of (9) to (10) to get 
\begin{equation}\label{eq:associative2}
Y(x_2)(Y(x_0) \otimes \on{id}_V)=Y(-x_0+x_2)(\on{id}_V \otimes Y(x_2)).
 \end{equation}
 }

\item   Take $\on{Res}_{x_0}$ to Equation \eqref{eq:Jacobi_beta=0} (i.e., $l=0$), we get 
\begin{align}\label{eq:Huang_associative_components}
Y_m   (\on{id}_V \otimes Y_n)=\sum_{i\geq 0}\binom{m}{i} Y_{m+n-i}    (Y_{i}\otimes \on{id}_V)
\end{align}
for all $ m, n \in \mathbb Z$.  This is exactly 
\begin{align}\label{eq:Huang_associative}
Y(z_1)(\on{id}_V \otimes Y(z_2))=Y(z_2)(Y(z_1-z_2) \otimes \on{id}_V).
\end{align}
For open string vertex algebras in \cite{Huang},  the associativity condition \cite[(2.3)]{Huang} is\[
\langle v', Y(z_1)(\on{id}_V \otimes Y(z_2))(u_1 \otimes u_2 \otimes v) \rangle=\langle v', Y(z_2)(Y(z_1-z_2) \otimes \on{id}_V)(u_1 \otimes u_2 \otimes v) \rangle
\]
as rational functions in the domain $|z_1|>|z_2|>|z_1-z_2|>0$. Thus when $\beta=0$, the vertex algebra we defined is also an open string vertex algebra. Moreover, we see below that the associativity condition in a field algebra is here equivalent to that of an open string vertex algebra:

\begin{lemma} 
The Equations \eqref{eq:associative_components_1} and \eqref{eq:Huang_associative_components} are equivalent. Thus \eqref{eq:associative} and \eqref{eq:Huang_associative} are equivalent.
\end{lemma}

\delete{Then we apply the change of variables $x_2 \to x_2+x_0$, $x_0 \to x_2$ and we get
\[
\langle v', Y(x_2+x_0)(\on{id} \otimes Y(x_2))(u_1 \otimes u_2 \otimes v) \rangle=\langle v', Y(x_2)(Y(x_0) \otimes \on{id})(u_1 \otimes u_2 \otimes v) \rangle. 
\]}

\begin{proof}
By using the expansion of $z^m=((z-x)+x)^m$ for any $m\in \mathbb Z$ with $x, z$ being independent variables, one gets $\sum_{i=0}^k\binom{m}{i}\binom{m-i}{k-i}(-1)^{k-i}=\delta_{k,0} $ for all $ m\in \mathbb Z$ and $ k\geq 0$. Now a routine computation shows that Equations \eqref{eq:associative_components} and \eqref{eq:Huang_associative_components} are equivalent.
Therefore \eqref{eq:associative} and \eqref{eq:Huang_associative} are equivalent.
 \end{proof}

\end{enumerate}
\end{remark}

The map $Y(x):V \otimes V \longrightarrow V[[x, x^{-1}]]$ can be thought of as a collection of products $V \otimes V \stackrel{Y_n}{\longrightarrow} V$ indexed by the powers of $x$. The notion dual to that of an algebra product $V \otimes V \longrightarrow V$ is a coalgebra product $V \longrightarrow V \otimes V $. So in order to define a vertex coalgebra (introduced in \cite{Hub1} with more restrictions), we would want a collection of coalgebra products, i.e., a map $\Y(x):V {\longrightarrow} (V \otimes V)[[x, x^{-1}]]$. 
We can write $ \Y(x)=\sum_{n\in \mathbb Z} \Y_n x^{-n-1}$ with $\Y_n: V {\longrightarrow} (V \otimes V)$.
If we dualize the vacuum condition \eqref{eq:vacuum} in Definition \ref{def:va}, the composition will change direction and the tensor product will switch sides based on the relation between duals and tensor products. Regarding the element $\vac$ as a morphism in $\on{Hom}^{[0]}(\cc, V)$, the expression becomes $(\on{id}_{V'} \otimes \epsilon)   \Y(x)=\on{id}_{V'}$ with $\epsilon=\vac^* : V' \longrightarrow \cc$. We can proceed similarly and dualize the other axioms in Definition \ref{def:va} to obtain the definition of a $(G_\Gamma, \beta, \gamma_0)$-vertex coalgebra:

\begin{definition}\label{def:cova}
A $(G_\Gamma, \beta, \gamma_0)$-\textbf{vertex coalgebra} (over $\cc$) is a $G_\Gamma$-module $V$ equipped with a  map $\Y(x)\in \Hom_{G_\Gamma}(V, \mc Hom(\cc [t,t^{-1}], V \otimes V))$ with $\cc t=\cc_{\gamma_0}$ as $G_\Gamma$-module:
\begin{align*}
\begin{array}{cccc}
\Y(x): & V & \longrightarrow & (V \otimes V)[[x,x^{-1}]] \\
           & v & \longmapsto      &\displaystyle \sum_{n \in \mathbb{Z}}\Y_n(v) \, x^{-n-1}
\end{array}
\end{align*}
called the vertex cooperator map, and a morphism $\epsilon \in \on{Hom}^{[0]}(V, \cc)$ called the covacuum map, satisfying the following axioms:
\begin{enumerate}[leftmargin=*, itemsep=5pt, label=(\roman*)]
\item \label{def:covacuum}  Covacuum: for all $v \in V$,
\begin{align}\label{eq:covacuum}
(\on{id}_V \otimes \epsilon)\Y(x)v=v
\end{align}
\item\label{def:cocreation} Cocreation: for all $v \in V$,
\begin{align} \label{eq:cocreation}
\begin{array}{cc}
 (\epsilon \otimes \on{id}_{V} )\Y(x)v \in V[[x]]  \text{ and} \\[5pt]
 \underset{x \to 0}{\on{lim}} (\epsilon \otimes \on{id}_{V} ) \Y(x)v=v  
 \end{array}
\end{align}
\item\label{def:cotruncation} Truncation: for any $u', v' \in V'$ and $v \in V$,
\[
\langle u' \otimes v', \Y(x)(v) \rangle \in \mc Hom^{\on{cont}}(\cc[t, t^{-1}], \cc).
\]
\item\label{def:coJacobi} co-Jacobi identity:
\begin{equation}\label{eq:coJacobi}
\begin{split}
\scalebox{0.91}{$ \displaystyle x_0^{-1}\delta\left(\frac{x_1-x_2}{x_0}\right)(\Y(x_2) \otimes \on{id}_{V})\Y(x_1)- x_0^{-1}\delta\left(\frac{x_2-x_1}{-x_0}\right)(\on{id}_V\otimes T^\beta )(\Y(x_1) \otimes \on{id}_{V} )\Y(x_2)$} \\[5pt]
\scalebox{0.91}{$ \displaystyle = x_2^{-1}\delta\left(\frac{x_1-x_0}{x_2}\right)(\on{id}_V \otimes \Y(x_0))\Y(x_2).$} 
 \end{split}
 \end{equation} 
 
 \item\label{def:coderivation} Derivation properties: set $\D=\on{Res}_x\big(x^{-2} (\epsilon \otimes \on{id}_V) \Y(x)): V  \longrightarrow V$. Then we have
 \begin{empheq}[left=\empheqlbrace]{align} 
\Y(x)   \D-(\D \otimes \on{id}_V) \Y(x)=\frac{d}{dx}\Y(x), \label{eq:coderivation_1}\\[5pt]
(\on{id}_V \otimes \D) \Y(x) =\frac{d}{dx}\Y(x). \label{eq:coderivation_2}
\end{empheq}
 
\end{enumerate}
\end{definition}

\begin{remark}
The notion of $(G_\Gamma, \beta, \gamma_0)$-vertex coalgebra given above differs from the vertex operator coalgebra introduced in \cite{Hub1}. Indeed, our choice a pairing with the dual space is different so the tensor products are reversed. Our convention will make the generalisation to other categories more natural later on. Moreover, we do not consider any conformal structure on $V$.
\end{remark}

The parallel between both definitions is clear, and their relationship is made even clearer by the next result. It was proved in \cite[Theorem 2.9]{Hub2} that there is a correspondence between vertex operator algebras and vertex operator coalgebras using the duality. We extend this result to the $G_\Gamma$-module case.

\begin{theorem}\label{thm:duality}
Let $\Gamma$ be an abelian group with $\beta:\Gamma \times \Gamma \longrightarrow \cc$. Assume that $\beta$ satisfies Relation \eqref{relation_1}. The restricted dual gives a bijection
\begin{center}
\begin{tikzpicture}[baseline=(current  bounding  box.center), scale=1, transform shape, on top/.style={preaction={draw=white,-,line width=#1}},on top/.default=4pt]
\tikzset{>=stealth}

\node (1) at (-4,0) []{$
\shortstack{$\{
\text{Harish-Chandra}$ \\ $(G_\Gamma, \beta, \gamma_0)\text{-vertex algebras} 
\}$}
$};
\node (2) at (3.5,0) []{$
\shortstack{$\{
\text{Harish-Chandra}$ \\ $(G_\Gamma, \beta, \gamma_0)\text{-vertex coalgebras} 
\}$}
$};
\draw[<->]  (-0.8,0) -- node[above] {$\ (-)'$} (0.1,0);
\end{tikzpicture}
\end{center}
\end{theorem}

\begin{proof}
Let us prove the first direction of the equivalence. Set $(V, Y(x), \vac)$ a Harish-Chandra $(G_\Gamma, \beta, \gamma_0)$-vertex algebra and consider the homogeneous linear map
\[
\Y(x): V' \longrightarrow (V' \otimes V')[[x, x^{-1}]]
\]
defined by
\[
\langle \Y(x) u', v \otimes w \rangle = \langle  u', Y(x)(v \otimes w) \rangle
\]
for all $u' \in V'$, $v, w \in V$ and set $\epsilon:V' \longrightarrow \cc$ sending $u'$ to $u'(\vac)$. As $\on{id}_{V'} \otimes \epsilon= (\on{id}_V)^* \otimes \vac^* \cong (\vac \otimes \on{id}_V)^*$, for any $u' \in V'$ and $v \in V$, we have
\[
\begin{array}{rcl}
\langle (\on{id}_{V'} \otimes \epsilon)   \Y(x)u', v \rangle & = & \langle \Y(x)u', (\vac \otimes \on{id}_V)(v) \rangle \\[5pt]
& = & \langle u', Y(x)(\vac \otimes v) \rangle  \\[5pt]
& = & \langle u',  v \rangle \quad \quad \text{by Equation \eqref{eq:vacuum}}\\[5pt]
& = &\langle \on{id}_{V'}(u'),  v \rangle,
\end{array}
\]
which proves \eqref{eq:covacuum} of Definition \ref{def:cova}.

Moreover, we can see that for all $u' \in V'$ and $v \in V$, we have
\[
\begin{array}{rcl}
\langle (\epsilon \otimes \on{id}_{V'})   \Y(x)u', v \rangle & = & \langle \Y(x)u', (\on{id}_V \otimes \vac)(v) \rangle \\[5pt]
& = & \langle u', Y(x)(v \otimes \vac) \rangle \\[5pt]
 & = &\displaystyle  \sum_{n \leq -1} \langle u', Y_n(v \otimes \vac) \rangle x^{-n-1} \quad \quad  \text{by Equation \eqref{eq:creation}} \\[5pt]
 & \in &  \langle u' , v \rangle+x\cc[[x]].
\end{array}
\]
As the above equality is true for all $v \in V$, it follows that $(\epsilon \otimes \on{id}_{V'})   \Y(x)u' \in V'[[x]]$ and that $ \underset{x \to 0}{\on{lim}} (\epsilon \otimes \on{id}_{V'})\Y(x)u'=u'$, which proves the cocreation property \eqref{eq:cocreation}.

As $V$ is Harish-Chandra, we have $((V_\gamma)^*)^* \cong V_\gamma$ for all $\gamma \in \Gamma$, and thus there exists an isomorphism of vector spaces $V'' \cong V$. Let $u'', v'' \in V''$ and consider their respective images $u$ and $v$ by the previous isomorphism. It then follows that
\[
\begin{array}{rcl}
\langle u'' \otimes v'', \Y(x)v' \rangle & = & \langle \Y(x)v',u \otimes v \rangle \\[5pt]
& = & \langle v', Y(x)(u \otimes v) \rangle \\[5pt]
& \in &  \on{Hom}^{\on{cont}}(\cc[t, t^{-1}], \cc) \quad \quad \text{by Definition \ref{def:va}.\ref{def:truncation}}
\end{array}
\]
which proves the truncation property \ref{def:cotruncation} in Definition \ref{def:cova}.

For any $v' \in V'$ and $v_1, v_2, v_3 \in V$, we have
\begin{align}
\scalebox{0.95}{$\langle (\Y(x_2) \otimes \on{id}_{V})\Y(x_1)(v'), v_1 \otimes v_2 \otimes v_3  \rangle$} \ & \scalebox{0.95}{$= \langle \Y(x_1)(v'), (\on{id}_V \otimes Y(x_2)) (v_1 \otimes v_2 \otimes v_3) \rangle$} \nonumber \\[5pt]
& \scalebox{0.95}{$=  \langle v', Y(x_1)(\on{id}_V \otimes Y(x_2)) (v_1 \otimes v_2 \otimes v_3) \rangle.$} \label{thm:duality_eq:1}
\end{align}

For any $v_1', v_2', v_3' \in V'$ homogeneous, we have
\[
\begin{array}{rcl}
\langle (\on{id}_{V'} \otimes T^\beta)(v_1' \otimes v_2' \otimes v_3'), v_1 \otimes v_2 \otimes v_3  \rangle & = & \beta(|v_2'|, |v_3'|)\langle (v_1' \otimes v_3' \otimes v_2'), v_1 \otimes v_2 \otimes v_3  \rangle \\[5pt]
& = & \beta(|v_2'|, |v_3'|) v_1'(v_3)v_2'(v_1)v_3'(v_2) \\[5pt]
& = & \beta(|v_2'|, |v_3'|)\langle v_1' \otimes v_2' \otimes v_3', v_2 \otimes v_1 \otimes v_3  \rangle.
\end{array}
\]
We write $(\Y(x_1) \otimes \on{id}_{V'})\Y(x_2)v'=\sum_{v_1' , v_2' , v_3'}v_1' \otimes v_2' \otimes v_3'$ for the decomposition in homogeneous components in $V'^{\otimes 3}$. Then we see that
\begin{align*}
\langle (\on{id}_{V'} \otimes T^\beta)(\Y(x_1) \otimes &   \on{id}_{V'})\Y(x_2)v', v_1 \otimes v_2 \otimes v_3  \rangle   \\[5pt]
& =  \sum_{v_1' , v_2' , v_3'} \beta(|v_2'|, |v_3'|)\langle v_1' \otimes v_2' \otimes v_3', v_2 \otimes v_1 \otimes v_3  \rangle.
\end{align*}

If the above expression is not zero, then $v_2'(v_1) \neq 0 \neq v_3'(v_2)$. Hence $v_2' \in (V_{|v_1]})^*=(V')_{-|v_1|}$ and $v_3' \in (V')_{-|v_2|}$. Thus $\beta(|v_2'|, |v_3'|)=\beta(-|v_1|, -|v_2|)=\beta(|v_1|, |v_2|)$ by the assumption on $\beta$. It follows that
\begin{align}
\langle (\on{id}_{V'} \otimes T^\beta)(\Y(x_1) \otimes & \on{id}_{V'})\Y(x_2)v', v_1 \otimes v_2 \otimes v_3  \rangle  \nonumber\\[5pt]
& =  \sum_{v_1' , v_2' , v_3'} \beta(|v_1|, |v_2|) \langle v_1' \otimes v_2' \otimes v_3', v_2 \otimes v_1 \otimes v_3  \rangle  \nonumber\\[5pt]
& =  \beta(|v_1|, |v_2|)  \langle \sum_{v_1' , v_2' , v_3'} v_1' \otimes v_2' \otimes v_3', v_2 \otimes v_1 \otimes v_3  \rangle  \nonumber\\[5pt]
& =  \scalebox{0.95}{$\beta(|v_1|, |v_2|)   \langle (\Y(x_1) \otimes \on{id}_{V'}) \Y(x_2)v', v_2 \otimes v_1 \otimes v_3  \rangle$} \nonumber \\[5pt]
& =  \beta(|v_1|, |v_2|) \langle v', Y(x_2) (\on{id}_{V} \otimes Y(x_1))(v_2 \otimes v_1 \otimes v_3)  \rangle \nonumber\\[5pt]
& =   \langle v', Y(x_2) (\on{id}_{V'} \otimes Y(x_1))(T^\beta \otimes \on{id}_V)(v_1 \otimes v_2 \otimes v_3)  \rangle. \label{thm:duality_eq:2}
\end{align}

Finally, we have
\begin{align}\label{thm:duality_eq:3}
\scalebox{0.98}{$\langle (\on{id}_{V'} \otimes \Y(x_0))\Y(x_2)(v'),  v_1 \otimes v_2 \otimes v_3  \rangle   =   \langle v', Y(x_2) (Y(x_0) \otimes \on{id}_{V})(v_1 \otimes v_2 \otimes v_3 )  \rangle.$}
\end{align}
By combining Equations \eqref{thm:duality_eq:1}, \eqref{thm:duality_eq:2} and \eqref{thm:duality_eq:3} and using the fact that the Jacobi identity \eqref{eq:Jacobi} is satisfied, we see that the co-Jacobi identity \eqref{eq:coJacobi} is also satisfied. 

Set $\D=\on{Res}_x\big(x^{-2} (\epsilon \otimes \on{id}_{V'}) \Y(x)): V'  \longrightarrow V'$. Then for any $v' \in V'$ and $u \in V$, we have
 \begin{align*}
\langle  \D(v'), u \rangle&=\on{Res}_{x}x^{-2}\langle (\epsilon \otimes \on{id}_{V'}) \Y(x)(v'), u  \rangle \\
&= \on{Res}_{x}x^{-2}\langle v', Y(x_0)(\on{id}_V \otimes \vac)(u) \rangle \\
&= \langle v', D (u) \rangle.
\end{align*}
It follows that, for any $v' \in V'$ and $u,v \in V$,
\begin{align*}
\langle \big(\Y(x)   \D-(\D \otimes &\on{id}_{V'}) \Y(x)\big)(v'), u \otimes v \rangle \\
&= \langle \Y(x) \D(v'), u \otimes v \rangle - \langle (\D \otimes \on{id}_{V'}) \Y(x)(v'), u \otimes v \rangle \\
&= \langle v', DY(x)(u \otimes v) \rangle - \langle v', Y(x)(\on{id}_V \otimes D) (u \otimes v) \rangle \\
&= \langle v', \big(DY(x)-Y(x)(\on{id}_V \otimes D)\big) (u \otimes v) \rangle \\
&= \langle v', Y(x)(D \otimes \on{id}_V) (u \otimes v) \rangle \quad \text{by Equations \eqref{eq:derivation_1} and \eqref{eq:derivation_2} in Definition \ref{def:va}}\\
&= \langle (\on{id}_{V'} \otimes \D)\Y(x)(v'), u \otimes v \rangle.
\end{align*}
We thus see that
\begin{align}\label{thm:duality_eq:4}
\Y(x)   \D-(\D \otimes &\on{id}_{V'}) \Y(x)=(\on{id}_{V'} \otimes \D)\Y(x).
\end{align}
Moreover, we have
\begin{align*}
\langle (\on{id}_{V'} \otimes \D) \Y(x)(v'), u \otimes v \rangle &= \langle v', Y(x)(D \otimes \on{id}_V)(u \otimes v) \rangle  \\
&= \frac{d}{dx} \langle v', Y(x)(u \otimes v) \rangle \quad \quad \text{by Equation \eqref{eq:derivation_2} in Definition \ref{def:va}}\\
&=\langle  \frac{d}{dx} \Y(x)v', u \otimes v \rangle.
\end{align*}
We conclude that Equation \eqref{eq:coderivation_2} is satisfied. With Equations \eqref{eq:coderivation_2} and \eqref{thm:duality_eq:4}, we see that Equation \eqref{eq:coderivation_1} is also satisfied. This concludes the first direction of the equivalence. The other direction is proved in a similar fashion.
\end{proof}

Recall that for a $(G_\Gamma, \beta, \gamma_0)$-vertex algebra $V$, we defined a linear transformation $D: V \longrightarrow V$ by $D=\on{Res}_x\big(x^{-2} Y(x)   (\on{id}_V \otimes \vac)\big)$. When $\beta=1$, we get the usual vertex algebra and $D$ will satisfy the derivation property (\cite[(3.1.25)]{Lepowsky-Li}). For more general $\beta$, we have the following proposition:

\begin{proposition}\label{prop:translation}
Let $\beta: \Gamma \times \Gamma \longrightarrow \cc$ a set map and let $V$ be a $G_\Gamma$-module with $Y(x)$ satisfying \ref{def:vacuum}-\ref{def:Jacobi} in Definition \ref{def:va}. 
\begin{enumerate}
\item If $\beta(\gamma, 0)=1$ for all $\gamma \in \Gamma$, then we have
\begin{align}\label{prop:translation_eq:derivation}
Y(x)   (D \otimes \on{id}_V)=\frac{d}{dx}Y(x).
\end{align}
\item If $\beta(\gamma_1, 0)=1$ and $\beta(\gamma_1, \gamma_2)\beta(\gamma_2, \gamma_1)=1$ for all $\gamma_1, \gamma_2 \in \Gamma$, we have
\begin{align}\label{prop:translation_eq:skewsymmetry}
Y(x)=e^{xD}Y(-x)T^\beta 
\end{align}
and
\begin{align}\label{prop:translation_eq:derivationbeta}
D  Y(x)-e^{xD}Y(-x)(D \otimes \on{id}_V)T^\beta=\frac{d}{dx}Y(x).
\end{align}
\item Assuming that Equations \eqref{prop:translation_eq:skewsymmetry} and \eqref{prop:translation_eq:derivationbeta} are satisfied and that $\beta(\gamma_1, \gamma_2)=\beta(\gamma_1, \gamma_2-2\gamma_0)$ for all $\gamma_1, \gamma_2 \in \Gamma$, then
\begin{align}\label{prop:translation_eq:commutator}
D   Y(x)-Y(x) (\on{id}_V \otimes D)=\frac{d}{dx}Y(x).
\end{align}
\end{enumerate}
\end{proposition}

\begin{proof}
(1) We compute that
\begin{align}
& Y(x_2)  (D \otimes \on{id}_V) \nonumber  \\[5pt]
& =  \on{Res}_{x_0} Y(x_2)  x_0^{-2}(Y(x_0) \otimes \on{id}_V)  ((\on{id}_V \otimes \vac) \otimes \on{id}_V) \nonumber \\[5pt]
& =  \displaystyle   \on{Res}_{x_1} \on{Res}_{x_0} x_0^{-2}x_1^{-1}\delta\left(\frac{x_2+x_0}{x_1}\right) Y(x_2)  (Y(x_0) \otimes \on{id}_V)  (\on{id}_V \otimes \vac \otimes \on{id}_V)\nonumber \\[5pt]
& =  \displaystyle   \on{Res}_{x_1} \on{Res}_{x_0} x_0^{-2}x_2^{-1}\delta\left(\frac{x_1-x_0}{x_2}\right) Y(x_2)  (Y(x_0) \otimes \on{id}_V)  (\on{id}_V \otimes \vac \otimes \on{id}_V)\nonumber \\[5pt]
&  \hfill   \text{by \cite[(2.3.17)]{Lepowsky-Li}}\nonumber \\[5pt]
 &= \displaystyle \on{Res}_{x_1} \on{Res}_{x_0} x_0^{-2}\left[x_0^{-1}\delta\left(\frac{x_1-x_2}{x_0}\right) Y(x_1) (\on{id}_V \otimes Y(x_2))\right. \nonumber\\[5pt]
& \displaystyle \quad \left. -x_0^{-1}\delta\left(\frac{x_2-x_1}{-x_0}\right)Y(x_2) (\on{id}_V \otimes Y(x_1)) (T^\beta \otimes \on{id}_V)\right] (\on{id}_V \otimes \vac \otimes \on{id}_V) \nonumber\\[5pt]
&  \hfill \text{by the Jacobi identity \eqref{eq:Jacobi}} \nonumber\\[5pt]
&= \displaystyle \on{Res}_{x_1}\Big[(x_1-x_2)^{-2}Y(x_1) (\on{id}_V \otimes Y(x_2))\nonumber \\[5pt]
& \displaystyle \quad -(x_2-x_1)^{-2}Y(x_2) (\on{id}_V \otimes Y(x_1))(T^\beta \otimes \on{id}_V)\Big](\on{id}_V \otimes \vac \otimes \on{id}_V). \label{prop:translation_eq:1}
\end{align}
But as $Y(x_2)(\vac \otimes \on{id}_V)=\on{id}_V$ by the vacuum property \eqref{eq:vacuum}, the first term in Equation \eqref{prop:translation_eq:1} reads $Y(x_1)(\on{id}_V \otimes Y(x_2)) (\on{id}_V \otimes \vac \otimes \on{id}_V)=Y(x_1)$. Moreover, as $\vac$ is a morphism $\cc \longrightarrow V$, we see that for any $v \in V_\gamma$, we have
\[
\begin{array}{rcl}
T^\beta   (\on{id}_V \otimes \vac)(v)  & = & T^\beta (v \otimes \vac) \\[5pt]
& = & \beta(\gamma, 0) \vac \otimes v  \\[5pt]
& = & \vac \otimes v \quad \quad  \text{by the assumption on } \beta \\[5pt]
& = & (\vac \otimes \on{id}_V)(v).
\end{array}
\] 
It follows that $T^\beta   (\on{id}_V \otimes \vac) = \vac \otimes \on{id}_V$, and so
\[
\begin{array}{rl}
Y(x_2) (\on{id}_V \otimes Y(x_1))& \mkern-18mu \ (T^\beta \otimes \on{id}_V)(\on{id}_V \otimes \vac \otimes \on{id}_V) \\[5pt]
& =  Y(x_2) (\on{id}_V \otimes Y(x_1))(\vac \otimes \on{id}_V \otimes \on{id}_V) \\[5pt]
& =  Y(x_2) (\vac \otimes \on{id}_V) (\on{id}_\cc \otimes Y(x_1)) \\[5pt]
& =  Y(x_1).
\end{array}
\]
Therefore Equation \eqref{prop:translation_eq:1} becomes
\[
\begin{array}{rcl}
Y(x_2)  (D \otimes \on{id}_V)& = & \displaystyle \on{Res}_{x_1}\left((x_1-x_2)^{-2}-(x_2-x_1)^{-2}\right) Y(x_1) \\[5pt]
& = & \displaystyle \on{Res}_{x_1}\left(\frac{\partial}{\partial x_2}x_2^{-1}\delta\left(\frac{x_1}{x_2}\right)\right) Y(x_1) \quad \text{by \cite[(2.3.11)]{Lepowsky-Li}}\\[5pt]
& = & \displaystyle \frac{d}{dx_2} \on{Res}_{x_1}\left( x_2^{-1}\delta \left(\frac{x_1}{x_2}\right)  Y(x_1)\right) \\[5pt]
& = & \displaystyle \frac{d}{dx_2} \on{Res}_{x_1}\left( x_2^{-1}\delta\left(\frac{x_1}{x_2}\right) Y(x_2)\right) \quad \text{by \cite[(2.1.35)]{Lepowsky-Li}}\\[5pt]
& = & \displaystyle  \frac{d}{dx_2} Y(x_2).
\end{array}
\]

(2) Under the assumption on $\beta$, the left hand side of the Jacobi identity \eqref{eq:Jacobi} applied to $u \otimes v \otimes w$ is invariant by the transformation $(x_0, x_1, x_2, u \otimes v) \longmapsto (-x_0, x_2, x_1, T^\beta(u \otimes v))$. It follows that the right hand side of \eqref{eq:Jacobi} satisfies
\begin{align}
&x_2^{-1}\delta  \left(\frac{x_1-x_0}{x_2}\right) Y(x_2)(Y(x_0) \otimes \on{id}_V) \notag\\
&=x_1^{-1}\delta \left(\frac{x_2+x_0}{x_1}\right)Y(x_1)(Y(-x_0) \otimes \on{id}_V)(T^\beta \otimes \on{id}_V) \notag\\
&=x_1^{-1}\delta \left(\frac{x_2+x_0}{x_1}\right)Y(x_2+x_0)(Y(-x_0) \otimes \on{id}_V)(T^\beta \otimes \on{id}_V)\quad \text{by \cite[(2.3.56)]{Lepowsky-Li}}. \label{prop:translation_eq:2}
\end{align}
From Equation \eqref{prop:translation_eq:derivation}, we see that $Y(x)e^{x_0 (D \otimes \on{id}_V)}=e^{x_0 \frac{d}{dx}}Y(x)=Y(x+x_0)$ by Taylor's theorem. Using \cite[(2.3.17)]{Lepowsky-Li} and taking $\on{Res}_{x_1}$ of Equation \eqref{prop:translation_eq:2}, we obtain
\begin{align}\label{prop:translation_eq:3}
Y(x_2)(Y(x_0) \otimes \on{id}_V) =Y(x_2)e^{x_0 (D \otimes \on{id}_V)}(Y(-x_0) \otimes \on{id}_V)(T^\beta \otimes \on{id}_V).
\end{align}
From the creation property \eqref{eq:creation}, we know that $Y(x)(v \otimes \vac) = v +xV[[x]]$ for all $v \in V$. It follows that for any $u, v \in V$, we have
\begin{align*}
Y(x_2)(Y(x_0) \otimes \on{id}_V)(u \otimes v \otimes \vac)&=Y(x_2)(Y(x_0)(u \otimes v) \otimes \vac) \\
&= Y(x_0)(u \otimes v)+x_2V[[x_2]][[x_0]].
\end{align*}
Likewise
\begin{align*}
Y(x_2)e^{x_0 (D \otimes \on{id}_V)}(Y(-x_0) \otimes \on{id}_V)(T^\beta \otimes \on{id}_V)&(u \otimes v \otimes \vac) \\
&=Y(x_2)(\big(e^{x_0 D}Y(-x_0)T^\beta (u \otimes v)\big) \otimes \vac) \\
& = e^{x_0 D}Y(-x_0)T^\beta(u \otimes v)+x_2V[[x_2]][[x_0]].
\end{align*}
Hence by taking $\on{Res}_{x_2}x_2^{-1}$ of Equation \eqref{prop:translation_eq:3} applied to $(u \otimes v \otimes \vac)$, we get
\begin{align}\label{prop:translation_eq:4}
Y(x)= e^{x D}Y(-x)T^\beta,
\end{align}
which can be rewritten as Equation \eqref{prop:translation_eq:skewsymmetry} by applying $T^\beta$ to the right. Then by differentiating Equation \eqref{prop:translation_eq:4} by $x$, we get
\begin{align*}
\frac{d}{dx}Y(x) &= De^{x D}Y(-x)T^\beta+e^{xD}\frac{d}{dx}(Y(-x))T^\beta \\
&=DY(x)-e^{xD}Y(-x)(D \otimes \on{id}_V)T^\beta \quad \text{by Equations \eqref{prop:translation_eq:derivation} and \eqref{prop:translation_eq:4}},
\end{align*}
which is Equation \eqref{prop:translation_eq:derivationbeta}.

\medbreak

(3) We see that for any homogeneous $u, v \in V$,
\begin{align*}
(D \otimes \on{id}_V)T^\beta(u \otimes v)&=\beta(|u|, |v|)(D(v) \otimes u) \\
&=\beta(|u|, |v|-2\gamma_0)(D(v) \otimes u) \quad \text{by the assumption on $\beta$} \\
&=\beta(|u|, |D(v)|)(D(v) \otimes u) \\
&=T^\beta (\on{id}_V \otimes D)(u \otimes v).
\end{align*}
We can thus permute $(D \otimes \on{id}_V)$ and $T^\beta$ in Equation \eqref{prop:translation_eq:derivationbeta}, and use Equation \eqref{prop:translation_eq:skewsymmetry} to obtain Equation \eqref{prop:translation_eq:commutator}.
\end{proof}

\begin{remark}
\begin{enumerate}[wide]
\item If $\beta$ satisfies the conditions in Proposition \ref{prop:translation}, then \ref{def:derivation} is not necessary to impose in Definition \ref{def:va} as it becomes a consequence of \ref{def:vacuum}-\ref{def:Jacobi}.

\item In the case $\beta=0$, the proof of the proposition shows that having Equation \eqref{eq:derivation_2} is equivalent to imposing $\on{Res}_{x_1}(x_2-x_1)^{-2} Y(x_1)=0$. It is unclear what are the requirements in order to obtain Equation \eqref{eq:derivation_1}, as it cannot come from the Jacobi identity \eqref{eq:Jacobi} anymore due to the second term on the left hand side vanishing.
\end{enumerate}
\end{remark}

A similar result can be proved for $(G_\Gamma, \beta, \gamma_0)$-vertex coalgebras. The proof is done parallel to that of Proposition \ref{prop:translation}, but the order of the operations is reversed.

\begin{proposition}\label{prop:cotranslation}
Let $\beta: \Gamma \times \Gamma \longrightarrow \cc$ a set map and let $V$ be a $G_\Gamma$-module with $\Y(\cdot, x)$ satisfying \ref{def:covacuum}- \ref{def:coJacobi} in Definition \ref{def:cova}. 
\begin{enumerate}
\item If $\beta(\gamma, 0)=1$ for all $\gamma \in \Gamma$,  then we have
\begin{align}\label{prop:cotranslation_eq:coderivation}
(\on{id}_V \otimes \D)   \Y(x)=\frac{d}{dx}\Y(x).
\end{align}
\item If $\beta(\gamma_1, 0)=1$ and $\beta(\gamma_1, \gamma_2)\beta(\gamma_2, \gamma_1)=1$ for all $\gamma_1, \gamma_2 \in \Gamma$, we have
\begin{align}\label{prop:cotranslation_eq:coskewsymmetry}
\Y(x)=T^\beta\Y(-x)e^{x\scriptsize{\D}} 
\end{align}
and
\begin{align}\label{prop:cotranslation_eq:coderivationbeta}
\Y(x) \D-T^\beta(\on{id}_V \otimes \D)\Y(-x)e^{x\scriptsize{\D}}=\frac{d}{dx}\Y(x).
\end{align}
\item\label{prop:cotranslation(3)} Assuming Equations \eqref{prop:cotranslation_eq:coskewsymmetry} and \eqref{prop:cotranslation_eq:coderivationbeta} are satisfied and that $\beta(\gamma_1, \gamma_2)=\beta(\gamma_1, \gamma_2-2\gamma_0)$ for all $\gamma_1, \gamma_2 \in \Gamma$, then
\begin{align}\label{prop:cotranslation_eq:cocommutator}
\Y(x)\D-(\D \otimes \on{id}_V)\Y(x)=\frac{d}{dx}\Y(x).
\end{align}
\end{enumerate}
\end{proposition}

\delete{
\begin{proof}
(1) We compute that
\begin{align}
&(\on{id}_V \otimes \D)\Y(x_2) \nonumber  \\[5pt]
& =  \on{Res}_{x_0} (\on{id}_V \otimes \epsilon \otimes \on{id}_V)x_0^{-2}(\on{id}_V \otimes \Y(x_0))\Y(x_2) \nonumber  \\[5pt]
& =  \displaystyle  \on{Res}_{x_1} \on{Res}_{x_0} (\on{id}_V \otimes \epsilon \otimes \on{id}_V)x_0^{-2}x_1^{-1}\delta\left(\frac{x_2+x_0}{x_1}\right)(\on{id}_V \otimes \Y(x_0))\Y(x_2) \nonumber  \\[5pt]
&= \displaystyle  \on{Res}_{x_1} \on{Res}_{x_0} (\on{id}_V \otimes \epsilon \otimes \on{id}_V)x_0^{-2}x_2^{-1}\delta\left(\frac{x_1-x_0}{x_2}\right)(\on{id}_V \otimes \Y(x_0))\Y(x_2) \nonumber \\[5pt]
& \hfill   \text{by \cite[(2.3.17)]{Lepowsky-Li}} \nonumber  \\[5pt]
& = \displaystyle \on{Res}_{x_1} \on{Res}_{x_0} (\on{id}_V \otimes \epsilon \otimes \on{id}_V)x_0^{-2}\left[x_0^{-1}\delta\left(\frac{x_1-x_2}{x_0}\right)(\Y(x_2) \otimes \on{id}_V)\Y(x_1)\right. \nonumber \\[5pt]
&\displaystyle \quad \left. -x_0^{-1}\delta\left(\frac{x_2-x_1}{-x_0}\right)(\on{id}_V \otimes T^\beta)(\Y(x_1) \otimes \on{id}_V)\Y(x_2)\right] \nonumber  \\[5pt]
 &\hfill \text{by  the co-Jacobi identity} \nonumber  \\[5pt]
&= \displaystyle \on{Res}_{x_1} (\on{id}_V \otimes \epsilon \otimes \on{id}_V)\Big[(x_1-x_2)^{-2}(\Y(x_2) \otimes \on{id}_V)\Y(x_1) \nonumber  \\[5pt]
&\displaystyle \quad -(x_2-x_1)^{-2}(\on{id}_V \otimes T^\beta)(\Y(x_1) \otimes \on{id}_V)\Y(x_2)\Big]. \label{prop:cotranslation_eq:1}
\end{align}
But $(\on{id}_V \otimes \epsilon) \Y(x_2)=\on{id}_V$ by the covacuum property \eqref{eq:covacuum}, hence $(\on{id}_V \otimes \epsilon \otimes \on{id}_V)(\Y(x_2) \otimes \on{id}_V)\Y(x_1)=\Y(x_1)$. Moreover, as $\epsilon$ is a morphism, we know that $\epsilon(V_\gamma) \neq 0$ only if $\gamma=0$. Hence for any $u, v \in V$ homogeneous,

\[
\begin{array}{rl}
(\epsilon \otimes \on{id}_V)   T^\beta (u \otimes v) & =\beta(|u|, |v|)(\epsilon \otimes \on{id}_V)(v \otimes u) \\[5pt]
& =\beta(|u|, |v|) \epsilon(v) \otimes u \\[5pt]
& =\beta(|u|, 0)  u \otimes \epsilon(v)  \quad \text{by the identification} \quad  \cc \otimes V= V \otimes \cc = V\\[5pt]
& =  u \otimes \epsilon(v)\quad \quad  \text{by the assumption on }\beta \\[5pt]
& =  (\on{id}_V \otimes \epsilon)(u \otimes v).
\end{array}
\]
It follows that $(\epsilon \otimes \on{id}_V)   T^\beta = \on{id}_V \otimes \epsilon$, and so
\[
\begin{array}{rl}
(\on{id}_V \otimes \epsilon \otimes \on{id}_V) & \mkern-18mu \ (\on{id}_V \otimes T^\beta)(\Y(x_1) \otimes \on{id}_V) \Y(x_2) \\[5pt]
& =  (\on{id}_V \otimes (\on{id}_V \otimes \epsilon))(\Y(x_1) \otimes \on{id}_V) \Y(x_2) \\[5pt]
& =  (\Y(x_1) \otimes \on{id}_1)(\on{id}_V \otimes \epsilon) \Y(x_2) \\[5pt]
& =  \Y(x_1).
\end{array}
\]
Therefore Equation \eqref{prop:cotranslation_eq:1} becomes
\[
\begin{array}{rcl}
(\on{id}_V \otimes \D)\Y(x_2)& = & \displaystyle \on{Res}_{x_1}\left((x_1-x_2)^{-2}-(x_2-x_1)^{-2}\right)\Y(x_1) \\[5pt]
& = & \displaystyle \on{Res}_{x_1}\left(\frac{\partial}{\partial x_2}x_2^{-1}\delta\left(\frac{x_1}{x_2}\right)\right)\Y(x_1) \quad \text{by \cite[(2.3.11)]{Lepowsky-Li}}\\[5pt]
& = & \displaystyle \frac{d}{dx_2} \on{Res}_{x_1}\left( x_2^{-1}\delta \left(\frac{x_1}{x_2}\right) \Y(x_1)\right) \\[5pt]
& = & \displaystyle \frac{d}{dx_2} \on{Res}_{x_1}\left( x_2^{-1}\delta\left(\frac{x_1}{x_2}\right) \Y(x_2)\right) \quad \text{by \cite[(2.1.35)]{Lepowsky-Li}}\\[5pt]
& = & \displaystyle  \frac{d}{dx_2} \Y(x_2).
\end{array}
\]

(2) Under the assumption on $\beta$, the left hand side of the co-Jacobi identity \eqref{eq:coJacobi} is invariant by the transformation $(x_0, x_1, x_2) \longmapsto (-x_0, x_2, x_1)$ and applying $(\on{id}_V \otimes T^\beta)$ on the left. It follows that the right hand side of \eqref{eq:coJacobi} satisfies
\begin{align}
&x_2^{-1}\delta  \left(\frac{x_1-x_0}{x_2}\right) (\on{id}_V \otimes \Y(x_0))\Y(x_2) \notag\\
&=x_1^{-1}\delta \left(\frac{x_2+x_0}{x_1}\right)(\on{id}_V \otimes T^\beta)(\on{id}_V \otimes \Y(-x_0))\Y(x_1) \notag\\
&=x_1^{-1}\delta \left(\frac{x_2+x_0}{x_1}\right)(\on{id}_V \otimes T^\beta)(\on{id}_V \otimes \Y(-x_0))\Y(x_2+x_0) \quad \text{by \cite[(2.3.56)]{Lepowsky-Li}}. \label{prop:cotranslation_eq:2}
\end{align}
From Equation \eqref{prop:cotranslation_eq:coderivation}, we see that $e^{x_0 (\on{id}_V \otimes \scriptsize{\D})}\Y(x)=e^{x_0 \frac{d}{dx}}\Y(x)=\Y(x+x_0)$ by Taylor's theorem. Using \cite[(2.3.17)]{Lepowsky-Li} and taking $\on{Res}_{x_1}$ of Equation \eqref{prop:cotranslation_eq:2}, we obtain
\begin{align}\label{prop:cotranslation_eq:3}
(\on{id}_V \otimes \Y(x_0))\Y(x_2) =(\on{id}_V \otimes T^\beta)(\on{id}_V \otimes \Y(-x_0))e^{x_0 (\on{id}_V \otimes \scriptsize{\D})}\Y(x_2).
\end{align}
From the cocreation property \eqref{eq:cocreation}, we know that $(\epsilon \otimes \on{id}_V)\Y(x)(v) = v +xV[[x]]$ for all $v \in V$. It follows that for any $v \in V$, we have
\begin{align*}
(\epsilon \otimes \on{id}_V \otimes \on{id}_V)(\on{id}_V \otimes \Y(x_0))\Y(x_2)(v) &=(\on{id}_\cc \otimes \Y(x_0))(\epsilon \otimes \on{id}_V)\Y(x_2)(v) \\
&= \Y(x_0)(v)+x_2(V\otimes V)[[x_2]][[x_0]].
\end{align*}
Likewise
\begin{align*}
(\epsilon \otimes \on{id}_V \otimes \on{id}_V)(\on{id}_V \otimes &T^\beta)(\on{id}_V \otimes \Y(-x_0))e^{x_0 (\on{id}_V \otimes \scriptsize{\D})}\Y(x_2)(v) \\
&=(\on{id}_\cc \otimes T^\beta)(\on{id}_\cc \otimes \Y(-x_0))e^{x_0 (\on{id}_\cc \otimes \scriptsize{\D})} (\epsilon \otimes \on{id}_V)\Y(x_2)(v) \\
& =T^\beta \Y(-x_0)e^{x_0 \scriptsize{\D}}(v)+x_2(V \otimes V)[[x_2]][[x_0]].
\end{align*}
Hence by applying $(\epsilon \otimes \on{id}_V \otimes \on{id}_V)$ to the left of Equation \eqref{prop:cotranslation_eq:3} and taking  $\on{Res}_{x_2}x_2^{-1}$, we get
\begin{align}\label{prop:cotranslation_eq:4}
\Y(x)= T^\beta \Y(-x)e^{x  \scriptsize{\D}},
\end{align}
which can be rewritten as Equation \eqref{prop:cotranslation_eq:coskewsymmetry} by applying $T^\beta$ to the left. Then by differentiating Equation \eqref{prop:cotranslation_eq:4} by $x$, we get
\begin{align*}
\frac{d}{dx}\Y(x) &= T^\beta \frac{d}{dx}(\Y(-x))e^{x  \scriptsize{\D}}+T^\beta \Y(-x) \D e^{x  \scriptsize{\D}} \\
&=\Y(x) \D-T^\beta(\on{id}_V \otimes \D) \Y(-x) e^{x  \scriptsize{\D}} \quad \text{by Equations \eqref{prop:cotranslation_eq:coderivation} and \eqref{prop:cotranslation_eq:4}},
\end{align*}
which is Equation \eqref{prop:cotranslation_eq:coderivationbeta}.

\medbreak

(3) We see that for any homogeneous $u, v \in V$,
\begin{align*}
T^\beta(\on{id}_V \otimes \D)(u \otimes v)&=\beta(|u|, |\D(v)|)(\D(v) \otimes u) \\
&=\beta(|u|, |v|-2\gamma_0)(\D(v) \otimes u)  \\
&=\beta(|u|, |v|)(\D(v) \otimes u)\quad \text{by the assumption on $\beta$} \\
&=(\D \otimes \on{id}_V)T^\beta(u \otimes v).
\end{align*}
We can thus permute $T^\beta$ and $(\on{id}_V \otimes \D)$ in Equation \eqref{prop:cotranslation_eq:coderivationbeta}, and use Equation \eqref{prop:cotranslation_eq:coskewsymmetry} to obtain Equation \eqref{prop:cotranslation_eq:cocommutator}.
\end{proof}
}

\begin{remark}\label{rem:bilinear_form}
If there exists a skew-symmetric bilinear form $\langle \cdot, \cdot \rangle$ on $\Gamma$, then the set $\{\langle \gamma_1, \gamma_0 \rangle \ | \ \gamma_1 \in \Gamma \}=l_0\mathbb{Z}$ is an ideal in $\mathbb{Z}$. Then for $\beta(\gamma_1, \gamma_2)=q^{\langle \gamma_1, \gamma_2 \rangle}$ such that for all $\gamma_1, \gamma_2 \in \Gamma$ where $q$ is a root of unity of order $2l_0$, the conditions on $\beta$ given in Propositions \ref{prop:translation} and \ref{prop:cotranslation} are satisfied. Moreover, if $(\cdot, \cdot)$ is a symmetric bilinear form, then $\beta(\gamma_1, \gamma_2)=(-1)^{(\gamma_1, \gamma_2)}$ also satisfies the aforementioned conditions. Finally, suppose that $\beta'$ and $\beta''$ both satisfy the conditions of Propositions  \ref{prop:translation} and \ref{prop:cotranslation}, then the product $\beta' \beta''$ also satisfies these conditions. In particular, it will be true in the following cases:
\begin{enumerate}
    \item the abelian group is $\mathbb{Z}$ and $\beta(m,n)=(-1)^{mn}$,
    \item the abelian group is $\Gamma \times \mathbb{Z}$ and $\beta((\gamma_1,n_1),(\gamma_2,n_2))=q^{\langle \gamma_1, \gamma_2 \rangle}(-1)^{n_1n_2}$,
    \item the abelian group is $\Gamma \times \mathbb{Z}/2\mathbb{Z}$ and $\beta((\gamma_1,n_1),(\gamma_2,n_2))=q^{\langle \gamma_1, \gamma_2 \rangle}(-1)^{n_1n_2}$.
\end{enumerate}
\end{remark}

\section{Vertex (co)algebra (co)modules}\label{sec:va_mod}

There exists a notion of module for a vertex algebra (cf. \cite[Definition 4.1.1]{Lepowsky-Li}). We reformulate this definition in the context of $(G_\Gamma, \beta, \gamma_0)$-vertex algebras.

\begin{definition}\label{def:va_mod}
Let $(V, Y(x), \vac)$ be a $(G_\Gamma, \beta, \gamma_0)$-vertex algebra. A $(G_\Gamma, \beta, \gamma_0)$-\textbf{left module} for $V$ is a $G_\Gamma$-module $M$ equipped with a  map $Y^M(x)\in \Hom_{G_\Gamma}(V\otimes M, \mc Hom(\cc [t,t^{-1}], M))$ with $\cc t=\cc_{\gamma_0}$ as $G_\Gamma$-module:
\begin{align*}
Y^M(x):  V \otimes M  \longrightarrow M[[x,x^{-1}]]
\end{align*}
satisfying the following axioms:
\begin{enumerate}[leftmargin=*, itemsep=5pt, label=(\roman*)]

\item \label{def:vacuum_mod} Vacuum:
\begin{align}\label{eq:vacuum_mod}
Y^M(x)(\vac \otimes \on{id}_{M} )=\on{id}_M.
\end{align}

\item\label{def:truncation_mod} Truncation: for any $w' \in M'$, $u \in V$, and $m \in M$, we have
\[
\langle w', Y(x)(u \otimes m) \rangle \in \mc Hom^{\on{cont}}(\cc[t, t^{-1}], \cc).
\]

\item\label{def:Jacobi_mod} Jacobi identity:
\begin{equation}\label{eq:Jacobi_mod}
\begin{split}
\scalebox{0.84}{$\displaystyle x_0^{-1}\delta \left(\frac{x_1-x_2}{x_0}\right)Y^M(x_1)(\on{id}_V \otimes Y^M(x_2))-x_0^{-1}\delta \left(\frac{x_2-x_1}{-x_0}\right)Y^M(x_2)(\on{id}_V \otimes Y^M(x_1))(T^\beta \otimes \on{id}_M)$} \\[5pt]
\scalebox{0.84}{$\displaystyle =x_2^{-1}\delta \left(\frac{x_1-x_0}{x_2}\right)Y^M(x_2)(Y(x_0) \otimes \on{id}_M)$} 
\end{split}
\end{equation}

 \item\label{def:derivation_mod} Derivation: There exists a linear map $D^M:M \longrightarrow M$ such that
\begin{empheq}[left=\empheqlbrace]{align} 
D^MY^M(x)-Y^M(x)(\on{id}_V \otimes D^M)=\frac{d}{dx}Y^M(x), \label{eq:derivation_mod_1}\\[5pt]
Y^M(x)   (D \otimes \on{id}_M) =\frac{d}{dx}Y^M(x). \label{eq:derivation_mod_2}
\end{empheq}

\end{enumerate}
\end{definition}

A notion of module for a vertex operator coalgebra was introduced in \cite{Hub2}. We do the same here for $(G_\Gamma, \beta, \gamma_0)$-vertex coalgebras:

\begin{definition}\label{def:cova_mod}
Let $(V, \Y(x), \epsilon)$ be a $(G_\Gamma, \beta, \gamma_0)$-vertex coalgebra. A $(G_\Gamma, \beta, \gamma_0)$-\textbf{right comodule} for $V$ is a $G_\Gamma$-module $M$ equipped with a  map $\Y^M(x)\in \Hom_{G_\Gamma}(M, \mc Hom(\cc [t,t^{-1}], M \otimes V))$ with $\cc t=\cc_{\gamma_0}$ as $G_\Gamma$-module:
\begin{align*}
\Y^M(x):  M  \longrightarrow  (M \otimes V)[[x,x^{-1}]]
\end{align*}
satisfying the following axioms:
\begin{enumerate}[leftmargin=*, itemsep=5pt, label=(\roman*)]
\item \label{covacuum_mod}  Covacuum:
\begin{align}\label{eq:covacuum_mod}
(\on{id}_{M} \otimes \epsilon)   \Y^M(x)=\on{id}_M.
\end{align}

\item\label{def:cotruncation_mod} Truncation: for any $v' \in V'$, $w' \in W'$ and $m \in M$,
\[
\langle v' \otimes w', \Y(x)(m) \rangle \in \mc Hom^{\on{cont}}(\cc[t, t^{-1}], \cc).
\]

\item\label{coJacobi_mod} co-Jacobi identity:
\begin{equation}\label{eq:coJacobi_mod}
\begin{split}
\scalebox{0.83}{$\displaystyle  x_0^{-1}\delta\left(\frac{x_1-x_2}{x_0}\right)(\Y^M(x_2) \otimes \on{id}_V)\Y^M(x_1)- x_0^{-1}\delta\left(\frac{x_2-x_1}{-x_0}\right)(\on{id}_M \otimes T^\beta)(\Y^M(x_1) \otimes \on{id}_V)\Y^M(x_2)$} \\[5pt]
\scalebox{0.83}{$ \displaystyle = x_2^{-1}\delta\left(\frac{x_1-x_0}{x_2}\right)(\on{id}_M \otimes\Y(x_0))\Y^M(x_2)$}.
\end{split}
\end{equation}

\item\label{def:coderivation_mod} Derivation: There exists a linear map $\D^M:M \longrightarrow M$ such that
\begin{empheq}[left=\empheqlbrace]{align} 
\Y^M(x)\D^M-(\D^M \otimes \on{id}_V)\Y^M(x)=\frac{d}{dx}\Y^M(x), \label{eq:coderivation_mod_1}\\[5pt]
(\on{id}_M \otimes \D) \Y^M(x) =\frac{d}{dx}\Y^M(x). \label{eq:coderivation_mod_2}
\end{empheq}

\end{enumerate}
\end{definition}

We can same obtain results similar to Theorem \ref{thm:duality}, Proposition \ref{prop:translation} and Proposition \ref{prop:cotranslation}. We start with the theorem:

\begin{theorem}\label{thm:duality_mod}
Let $\Gamma$ be an abelian group and $\beta: \Gamma \times \Gamma \longrightarrow \cc$ a map. Assume that $\beta$ satisfies Relation \eqref{relation_1}. If $V$ is a Harish-Chandra $(G_\Gamma, \beta, \gamma_0)$-vertex algebra, then the restricted dual gives a bijection:
\begin{center}
\begin{tikzpicture}[baseline=(current  bounding  box.center), scale=1, transform shape, on top/.style={preaction={draw=white,-,line width=#1}},on top/.default=4pt]
\tikzset{>=stealth}

\node (1) at (-4,0) [][]{$
\shortstack{$\{
\text{Harish-Chandra}$ \\ $(G_\Gamma, \beta, \gamma_0)\text{-left $V$-modules} 
\}$}
$};
\node (2) at (3.5,0) []{$
\shortstack{$\{
\text{Harish-Chandra}$ \\ $(G_\Gamma, \beta, \gamma_0)\text{-right $V'$-comodules} 
\}$}
$};
\draw[<->]  (-0.9,0) -- node[above] {$\ (-)'$} (-0.05,0);
\end{tikzpicture}
\end{center}
\end{theorem}

\begin{proof}
Let us prove the first direction of the equivalence. Consider $(V, Y(x), \vac)$ a Harish-Chandra $(G_\Gamma, \beta, \gamma_0)$-vertex algebra and set write $(V', \Y(x), \epsilon)$ for its dual vertex coalgebra. Let $(M, Y^M(x))$ be a $(G_\Gamma, \beta, \gamma_0)$-left $V$-module, and define a homogeneous linear map
\[
\Y^{M'}(x):M' \longrightarrow (M' \otimes V')[[x, x^{-1}]]
\]
by
\[
\langle \Y^{M'}(x) (w'), v \otimes m \rangle = \langle  w', Y^M(x)(v \otimes m) \rangle
\]
for any $w' \in M', m \in M$ and $v \in V$. As $\on{id}_{M'} \otimes \epsilon= (\on{id}_M)^* \otimes \vac^* \cong (\vac \otimes \on{id}_M)^*$, for any $w' \in M'$ and $m \in M$, we have
\[
\begin{array}{rcl}
\langle (\on{id}_{M'} \otimes \epsilon)   \Y^{M'}(x)w', m \rangle & = & \langle \Y^{M'}(x)u', (\vac \otimes \on{id}_M)(m) \rangle \\[5pt]
& = & \langle w', Y^M(x)(\vac \otimes \on{id}_M)(m) \rangle  \\[5pt]
& = & \langle w',  m \rangle \quad \quad \text{by Equation \eqref{eq:vacuum_mod}}\\[5pt]
& = &\langle \on{id}_{M'}(w'),  m \rangle,
\end{array}
\]
which proves \eqref{eq:covacuum_mod} of Definition \ref{def:cova_mod}.

As $M$ and $V$ are Harish-Chandra, we have $(M')' \cong M$ and $(V')' \cong V$. Let $v'' \in V''$ and $m'' \in M''$ and consider their respective images $v \in V$ and $m \in M$ the previous isomorphisms. It then follows that
\[
\begin{array}{rcl}
\langle v'' \otimes m'', \Y^{M'}(x)(w') \rangle & = & \langle \Y^{M'}(x)v', v \otimes m \rangle \\[5pt]
& = & \langle w', Y^M(x)(v \otimes m) \rangle \\[5pt]
& \in &  \on{Hom}^{\on{cont}}(\cc[t, t^{-1}], \cc) \quad \quad \text{by Definition \ref{def:va_mod}.\ref{def:truncation_mod}}
\end{array}
\]
which proves the truncation property \ref{def:cotruncation_mod} in Definition \ref{def:cova_mod}.

For any $m' \in M'$, $m \in M$ and $v_1, v_2 \in V$, we have
\begin{align}
\scalebox{0.95}{$\langle (\Y^{M'}(x_2) \otimes \on{id}_{V'})\Y^{M'}(x_1)(m'), v_1 \otimes v_2 \otimes m  \rangle$} \ & \scalebox{0.95}{$= \langle \Y^{M'}(x_1)(m'), (\on{id}_V \otimes Y^M(x_2)) (v_1 \otimes v_2 \otimes m) \rangle$} \nonumber \\[5pt]
& \scalebox{0.95}{$=  \langle m', Y^M(x_1)(\on{id}_V \otimes Y^M(x_2)) (v_1 \otimes v_2 \otimes m) \rangle.$} \label{thm:duality_mod_eq:1}
\end{align}

For any $w' \in M'$ and $v_1', v_2' \in V'$ homogeneous, we have
\[
\begin{array}{rcl}
\langle (\on{id}_{M'} \otimes T^\beta)(w' \otimes v_1' \otimes v_2'), v_1 \otimes v_2 \otimes m  \rangle & = & \beta(|v_1'|, |v_2'|)\langle (w' \otimes v_2' \otimes v_1'), v_1 \otimes v_2 \otimes m  \rangle \\[5pt]
& = & \beta(|v_1'|, |v_2'|) v_1'(v_1)v_2'(v_2)w'(m) \\[5pt]
& = & \beta(|v_1'|, |v_2'|)\langle w' \otimes v_1' \otimes v_2', v_2 \otimes v_1 \otimes m  \rangle.
\end{array}
\]
We write $(\Y^{M'}(x_1) \otimes \on{id}_{V'})\Y^{M'}(x_2)(m')=\sum_{w' , v_1' , v_2'}w' \otimes v_1' \otimes v_2'$ for the decomposition in homogeneous components in $M' \otimes V' \otimes V'$. Then we see that
\begin{align*}
\langle (\on{id}_{M'} \otimes T^\beta)(\Y^{M'}(x_1) & \otimes \on{id}_{V'})\Y^{M'}(x_2)(m'), v_1 \otimes v_2 \otimes m  \rangle   \\[5pt]
& = \sum_{w' , v_1' , v_2'} \beta(|v_1'|, |v_2'|)\langle w' \otimes v_1' \otimes v_2', v_2 \otimes v_1 \otimes m  \rangle.
\end{align*}

If the above expression is not zero, then $v_1'(v_1) \neq 0 \neq v_2'(v_2)$. Hence $v_1' \in (V_{|v_1]})^*=(V')_{-|v_1|}$ and $v_2' \in (V')_{-|v_2|}$. Thus $\beta(|v_1'|, |v_2'|)=\beta(-|v_1|, -|v_2|)=\beta(|v_1|, |v_2|)$ by the assumption on $\beta$. It follows that
\begin{align}
\langle (\on{id}_{V'} \otimes T^\beta)(\Y^{M'}(x_1)& \otimes \on{id}_{V'})\Y^{M'}(x_2)(m'), v_1 \otimes v_2 \otimes m  \rangle  \nonumber\\[5pt]
& =  \sum_{w' , v_1' , v_2'} \beta(|v_1|, |v_2|) \langle w' \otimes v_1' \otimes v_2', v_2 \otimes v_1 \otimes m  \rangle  \nonumber\\[5pt]
\delete{
& =  \beta(|v_1|, |v_2|)  \langle \sum_{w' , v_1' , v_2'} v_1' \otimes v_2' \otimes v_3', v_2 \otimes v_1 \otimes m  \rangle  \nonumber\\[5pt]}
& =  \scalebox{0.95}{$\beta(|v_1|, |v_2|)   \langle (\Y^{M'}(x_1) \otimes \on{id}_{V'})\Y^{M'}(x_2)(m'), v_2 \otimes v_1 \otimes m  \rangle$} \nonumber \\[5pt]
& =  \beta(|v_1|, |v_2|) \langle m', Y^M(x_2) (\on{id}_{V} \otimes Y^M(x_1))(v_2 \otimes v_1 \otimes m)  \rangle \nonumber\\[5pt]
& =   \langle m', Y^M(x_2) (\on{id}_{V} \otimes Y^M(x_1))(T^\beta \otimes \on{id}_M)(v_1 \otimes v_2 \otimes m)  \rangle. \label{thm:duality_mod_eq:2}
\end{align}

Finally, we have
\begin{align}\label{thm:duality_mod_eq:3}
\scalebox{0.99}{$\langle (\on{id}_{M'} \otimes \Y(x_0))\Y^{M'}(x_2)(m'),  v_1 \otimes v_2 \otimes m  \rangle   =   \langle w', Y^M(x_2) (Y(x_0) \otimes \on{id}_{M})(v_1 \otimes v_2 \otimes m)  \rangle.$}
\end{align}
By combining Equations \eqref{thm:duality_mod_eq:1}, \eqref{thm:duality_mod_eq:2} and \eqref{thm:duality_mod_eq:3} and using the fact that the Jacobi identity \eqref{eq:Jacobi_mod} is satisfied, we see that the co-Jacobi identity \eqref{eq:coJacobi_mod} is also satisfied. 

Define $\D^{M'}$ as the dual of $D^M$ given by
\begin{align*}
\langle  \D^{M'}(m'), m \rangle= \langle m', D^M(m) \rangle
\end{align*}
for all $m' \in M'$ and $m \in M$. It follows that, for any $v' \in V'$ and $u,v \in V$,
\begin{align*}
\langle \big(\Y^{M'}(x) \D^{M'}-(\D^{M'} \otimes &\on{id}_{V'}) \Y^{M'}(x)\big)(m'), v \otimes m \rangle \\
&= \langle \Y^{M'}(x) \D^{M'}(m'), m \otimes v \rangle - \langle (\D^{M'} \otimes \on{id}_{V'}) \Y^{M'}(x)(m'), v \otimes m \rangle \\
&= \langle m', D^MY^M(x)(v \otimes m) \rangle - \langle m', Y^M(x)(\on{id}_V \otimes D^M) (v \otimes m) \rangle \\
&= \langle m', \big(D^MY^M(x)-Y^M(x)(\on{id}_V \otimes D^M)\big) (v \otimes m) \rangle \\
&= \langle m', Y^M(x)(D \otimes \on{id}_M) (v \otimes m) \rangle \quad \text{by Equations \eqref{eq:derivation_mod_1} and \eqref{eq:derivation_mod_2} in Definition \ref{def:va_mod}}\\
&= \langle (\on{id}_{M'} \otimes \D)\Y^{M'}(x)(m'), v \otimes m \rangle.
\end{align*}
We thus see that
\begin{align}\label{thm:duality_mod_eq:4}
\Y^{M'}(x) \D^{M'}-(\D^{M'} \otimes &\on{id}_{V'}) \Y^{M'}(x)=(\on{id}_{M'} \otimes \D)\Y^{M'}(x).
\end{align}
Moreover, we have
\begin{align*}
\langle (\on{id}_{M'} \otimes \D) \Y^{M'}(x)(m'), v \otimes m \rangle &= \langle m', Y^M(x)(D \otimes \on{id}_M)(v \otimes m) \rangle  \\
&= \frac{d}{dx} \langle m', Y^M(x)(v \otimes m) \rangle \quad \text{by Equation \eqref{eq:derivation_mod_2} in Definition \ref{def:va_mod}}\\
&=\langle  \frac{d}{dx} \Y^{M'}(x)(m'), v \otimes m \rangle.
\end{align*}
We conclude that Equation \eqref{eq:coderivation_mod_2} is satisfied. With Equations \eqref{eq:coderivation_mod_2} and \eqref{thm:duality_mod_eq:4}, we see that Equation \eqref{eq:coderivation_mod_1} is also satisfied. This concludes the first direction of the equivalence. The other direction is proved in a similar fashion.
\end{proof}

\begin{proposition}\label{prop:translation_mod}
Let $V$ be a $(G_\Gamma, \beta, \gamma_0)$-vertex algebra and $M$ a $G_\Gamma$-module satisfying \ref{def:vacuum_mod}-\ref{def:Jacobi_mod} in Definition \ref{def:va_mod}. 
\begin{enumerate}
\item Assume that $\beta(\gamma, 0)=1$ for all $\gamma \in \Gamma$. Then
\begin{align*}
Y^M(x)   (D \otimes \on{id}_M)=\frac{d}{dx}Y^M(x).
\end{align*}

\item\label{prop:translation_mod(2)} Assume that $\beta(\gamma, 0)=1$ and that $\beta(-2\gamma_0, \gamma)=1$ for all $\gamma \in \Gamma$. Assume also that there exists $\omega \in V$ such that $D=Y_0(\omega \otimes \on{id}_V)$. Set $D^M=Y_0^M(\omega \otimes \on{id}_M)$. Then
\begin{align*}
D^MY^M(x)-Y^M(x)(\on{id}_V \otimes D^M)=\frac{d}{dx}Y^M(x).
\end{align*}

\end{enumerate}
\end{proposition}

\begin{proof}
(1) The proof follows the same reasoning as the first point of Proposition \ref{prop:translation}.
\delete{
We compute that
\begin{align}
& Y^M(x_2)  (D \otimes \on{id}_M) \nonumber  \\[5pt]
& =  \on{Res}_{x_0} Y^M(x_2)  x_0^{-2}(Y(x_0) \otimes \on{id}_M)  ((\on{id}_V \otimes \vac) \otimes \on{id}_M) \nonumber \\[5pt]
& =  \displaystyle   \on{Res}_{x_1} \on{Res}_{x_0} x_0^{-2}x_1^{-1}\delta\left(\frac{x_2+x_0}{x_1}\right) Y^M(x_2)  (Y(x_0) \otimes \on{id}_M)  (\on{id}_V \otimes \vac \otimes \on{id}_M)\nonumber \\[5pt]
& =  \displaystyle   \on{Res}_{x_1} \on{Res}_{x_0} x_0^{-2}x_2^{-1}\delta\left(\frac{x_1-x_0}{x_2}\right) Y^M(x_2)  (Y(x_0) \otimes \on{id}_M)  (\on{id}_V \otimes \vac \otimes \on{id}_M)\nonumber \\[5pt]
&  \hfill   \text{by \cite[(2.3.17)]{Lepowsky-Li}}\nonumber \\[5pt]
 &= \displaystyle \on{Res}_{x_1} \on{Res}_{x_0} x_0^{-2}\left[x_0^{-1}\delta\left(\frac{x_1-x_2}{x_0}\right) Y^M(x_1) (\on{id}_V \otimes Y^M(x_2))\right. \nonumber\\[5pt]
& \displaystyle \quad \left. -x_0^{-1}\delta\left(\frac{x_2-x_1}{-x_0}\right)Y^M(x_2) (\on{id}_V \otimes Y^M(x_1)) (T^\beta \otimes \on{id}_M)\right] (\on{id}_V \otimes \vac \otimes \on{id}_M) \nonumber\\[5pt]
&  \hfill \text{by the Jacobi identity \eqref{eq:Jacobi_mod}} \nonumber\\[5pt]
&= \displaystyle \on{Res}_{x_1}\Big[(x_1-x_2)^{-2}Y^M(x_1) (\on{id}_V \otimes Y^M(x_2))\nonumber \\[5pt]
& \displaystyle \quad -(x_2-x_1)^{-2}Y^M(x_2) (\on{id}_V \otimes Y^M(x_1))(T^\beta \otimes \on{id}_M)\Big](\on{id}_V \otimes \vac \otimes \on{id}_M). \label{prop:translation_mod_eq:1}
\end{align}
But as $Y^M(x_2)(\vac \otimes \on{id}_M)=\on{id}_M$ by the vacuum property \eqref{eq:vacuum_mod}, the first term in Equation \eqref{prop:translation_mod_eq:1} reads $Y^M(x_1)(\on{id}_V \otimes Y^M(x_2)) (\on{id}_V \otimes \vac \otimes \on{id}_M)=Y^M(x_1)$. Moreover, as $T^\beta   (\on{id}_V \otimes \vac) = \vac \otimes \on{id}_V$ (cf. the proof of Proposition \ref{prop:translation}), it follows that
\[
\begin{array}{rl}
Y^M(x_2) (\on{id}_V \otimes Y^M(x_1))& \mkern-18mu \ (T^\beta \otimes \on{id}_M)(\on{id}_V \otimes \vac \otimes \on{id}_M) \\[5pt]
& =  Y^M(x_2) (\on{id}_V \otimes Y^M(x_1))(\vac \otimes \on{id}_V \otimes \on{id}_M) \\[5pt]
& =  Y^M(x_2) (\vac \otimes \on{id}_M) (\on{id}_\cc \otimes Y^M(x_1)) \\[5pt]
& =  Y^M(x_1).
\end{array}
\]
Therefore Equation \eqref{prop:translation_mod_eq:1} becomes
\[
\begin{array}{rcl}
Y^M(x_2)  (D \otimes \on{id}_M)& = & \displaystyle \on{Res}_{x_1}\left((x_1-x_2)^{-2}-(x_2-x_1)^{-2}\right) Y^M(x_1) \\[5pt]
& = & \displaystyle \on{Res}_{x_1}\left(\frac{\partial}{\partial x_2}x_2^{-1}\delta\left(\frac{x_1}{x_2}\right)\right) Y^M(x_1) \quad \text{by \cite[(2.3.11)]{Lepowsky-Li}}\\[5pt]
& = & \displaystyle \frac{d}{dx_2} \on{Res}_{x_1}\left( x_2^{-1}\delta \left(\frac{x_1}{x_2}\right)  Y^M(x_1)\right) \\[5pt]
& = & \displaystyle \frac{d}{dx_2} \on{Res}_{x_1}\left( x_2^{-1}\delta\left(\frac{x_1}{x_2}\right) Y^M(x_2)\right) \quad \text{by \cite[(2.1.35)]{Lepowsky-Li}}\\[5pt]
& = & \displaystyle  \frac{d}{dx_2} Y^M(x_2).
\end{array}
\]
}

(2) For any $v \in V$ and $m \in M$, in the Jacobi identity \eqref{eq:Jacobi_mod} applied to $\omega \otimes v \otimes m$, take $\on{Res}_{x_1}\on{Res}_{x_0}$ and use \cite[(2.3.17)]{Lepowsky-Li}. We obtain
\[
Y^M_0\left(\omega \otimes Y^M(x_2)(v \otimes m)\right)-\beta(|\omega |, |v|)Y^M(x_2)\left(v \otimes Y^M_0(\omega \otimes m)\right)=Y^M(x_2)(Y_0(\omega \otimes v) \otimes m)
\]
which is equivalent to
\[
D^MY^M(x_2)(v \otimes m)-\beta(-2\gamma_0, |v|)Y^M(x_2)(\on{id}_V \otimes D^M)(v \otimes m)=Y^M(x_2)(D \otimes \on{id}_M) (v \otimes m)
\]
as $|\omega|=|Y_0(\omega \otimes v)|+|x^{-1}|-|v|=|D(v)|+|x^{-1}|-|v|=|x^{-1}|-|x|=-2\gamma_0$. Hence by the assumption on $\beta$, we obtain
\[
D^MY^M(x_2)(v \otimes m)-Y^M(x_2)(\on{id}_V \otimes D^M)(v \otimes m)=Y^M(x_2)(D \otimes \on{id}_M) (v \otimes m)
\]
for all $v \in V$ and $m \in M$. Finally, using the fact that Equation \eqref{eq:derivation_mod_2} was proved in the first point, we conclude that Equation \eqref{eq:derivation_mod_1} is also satisfied.
\end{proof}

A similar result can be proved for $(G_\Gamma, \beta, \gamma_0)$-vertex coalgebras:

\begin{proposition}\label{prop:cotranslation_mod}
Let $V$ be a $(G_\Gamma, \beta, \gamma_0)$-vertex coalgebra and $M$ a $G_\Gamma$-module satisfying \ref{covacuum_mod}-\ref{coJacobi_mod} in Definition \ref{def:cova_mod}. 
\begin{enumerate}[wide]
\item Assume that $\beta(\gamma, 0)=1$ for all $\gamma \in \Gamma$. Then
\begin{align*}
(\on{id}_M \otimes \D)   \Y^M(x)=\frac{d}{dx}\Y^M(x).
\end{align*}

\item Assume that $\beta(\gamma, 0)=1$ and that $\beta(2\gamma_0, \gamma)=1$ for all $\gamma \in \Gamma$. Assume also that there exists $\rho \in \on{Hom}(V, \cc)$ such that $\D=(\on{id}_V \otimes \rho)\Y_0$. Set $\D^M=(\on{id}_M \otimes \rho)\Y^M_0$. Then
\begin{align*}
\Y^M(x)\D^M-(\D^M \otimes \on{id_V})\Y^M(x)=\frac{d}{dx}\Y^M(x).
\end{align*}

\end{enumerate}
\end{proposition}

\begin{proof}
(1) The proof follows the same reasoning as the first point of Proposition \ref{prop:cotranslation}.

\delete{
We compute that
\begin{align}
&(\on{id}_M \otimes \D)\Y^M(x_2) \nonumber  \\[5pt]
& =  \on{Res}_{x_0} (\on{id}_M \otimes \epsilon \otimes \on{id}_V)x_0^{-2}(\on{id}_M \otimes \Y(x_0))\Y^M(x_2) \nonumber  \\[5pt]
& =  \displaystyle  \on{Res}_{x_1} \on{Res}_{x_0} (\on{id}_M \otimes \epsilon \otimes \on{id}_V)x_0^{-2}x_1^{-1}\delta\left(\frac{x_2+x_0}{x_1}\right)(\on{id}_M \otimes \Y(x_0))\Y^M(x_2) \nonumber  \\[5pt]
&= \displaystyle  \on{Res}_{x_1} \on{Res}_{x_0} (\on{id}_M \otimes \epsilon \otimes \on{id}_V)x_0^{-2}x_2^{-1}\delta\left(\frac{x_1-x_0}{x_2}\right)(\on{id}_M \otimes \Y(x_0))\Y^M(x_2) \nonumber \\[5pt]
& \hfill   \text{by \cite[(2.3.17)]{Lepowsky-Li}} \nonumber  \\[5pt]
& = \displaystyle \on{Res}_{x_1} \on{Res}_{x_0} (\on{id}_M \otimes \epsilon \otimes \on{id}_V)x_0^{-2}\left[x_0^{-1}\delta\left(\frac{x_1-x_2}{x_0}\right)(\Y^M(x_2) \otimes \on{id}_V)\Y^M(x_1)\right. \nonumber \\[5pt]
&\displaystyle \quad \left. -x_0^{-1}\delta\left(\frac{x_2-x_1}{-x_0}\right)(\on{id}_M \otimes T^\beta)(\Y^M(x_1) \otimes \on{id}_V)\Y^M(x_2)\right] \nonumber  \\[5pt]
 &\hfill \text{by  the co-Jacobi identity \eqref{eq:coJacobi_mod}} \nonumber  \\[5pt]
&= \displaystyle \on{Res}_{x_1} (\on{id}_M \otimes \epsilon \otimes \on{id}_V)\Big[(x_1-x_2)^{-2}(\Y^M(x_2) \otimes \on{id}_V)\Y^M(x_1) \nonumber  \\[5pt]
&\displaystyle \quad -(x_2-x_1)^{-2}(\on{id}_M \otimes T^\beta)(\Y^M(x_1) \otimes \on{id}_V)\Y^M(x_2)\Big]. \label{prop:cotranslation_mod_eq:1}
\end{align}
But $(\on{id}_M \otimes \epsilon) \Y^M(x_2)=\on{id}_M$ by the covacuum property \eqref{eq:covacuum_mod}, and therefore $(\on{id}_M \otimes \epsilon \otimes \on{id}_V)(\Y^M(x_2) \otimes \on{id}_V)\Y^M(x_1)=\Y^M(x_1)$. Moreover, we know that $(\epsilon \otimes \on{id}_V)   T^\beta = \on{id}_V \otimes \epsilon$, and so
\[
\begin{array}{rl}
(\on{id}_M \otimes \epsilon \otimes \on{id}_V) & \mkern-18mu \ (\on{id}_M \otimes T^\beta)(\Y^M(x_1) \otimes \on{id}_V) \Y^M(x_2) \\[5pt]
& =  (\on{id}_M \otimes (\on{id}_V \otimes \epsilon))(\Y^M(x_1) \otimes \on{id}_V) \Y^M(x_2) \\[5pt]
& =  (\Y^M(x_1) \otimes \on{id}_1)(\on{id}_M \otimes \epsilon) \Y^M(x_2) \\[5pt]
& =  \Y^M(x_1).
\end{array}
\]
Therefore Equation \eqref{prop:cotranslation_mod_eq:1} becomes
\[
\begin{array}{rcl}
(\on{id}_M \otimes \D)\Y^M(x_2)& = & \displaystyle \on{Res}_{x_1}\left((x_1-x_2)^{-2}-(x_2-x_1)^{-2}\right)\Y^M(x_1) \\[5pt]
& = & \displaystyle \on{Res}_{x_1}\left(\frac{\partial}{\partial x_2}x_2^{-1}\delta\left(\frac{x_1}{x_2}\right)\right)\Y^M(x_1) \quad \text{by \cite[(2.3.11)]{Lepowsky-Li}}\\[5pt]
& = & \displaystyle \frac{d}{dx_2} \on{Res}_{x_1}\left( x_2^{-1}\delta \left(\frac{x_1}{x_2}\right) \Y^M(x_1)\right) \\[5pt]
& = & \displaystyle \frac{d}{dx_2} \on{Res}_{x_1}\left( x_2^{-1}\delta\left(\frac{x_1}{x_2}\right) \Y^M(x_2)\right) \quad \text{by \cite[(2.1.35)]{Lepowsky-Li}}\\[5pt]
& = & \displaystyle  \frac{d}{dx_2} \Y^M(x_2).
\end{array}
\]
}

(2) Apply $(\on{id}_M \otimes \on{id}_V \otimes \rho)$ to the co-Jacobi identity \eqref{eq:coJacobi_mod} applied to $m \in M$, take $\on{Res}_{x_1}\on{Res}_{x_0}$ and use \cite[(2.3.17)]{Lepowsky-Li}. We obtain
\begin{equation}\label{prop:cotranslation_mod_eq:2}
\begin{split}
(\Y^M(x_2) \otimes \on{id}_\cc)(\on{id}_M \otimes \rho)\Y^M_0(m)-(\on{id}_M \otimes \on{id}_V \otimes \rho)(\on{id}_M \otimes T^\beta)(\Y_0^M \otimes \on{id}_V)\Y^M(x_2)(m) \\
=(\on{id}_M \otimes \left((\on{id}_V \otimes \rho)\Y_0\right))\Y^M(x_2)(m).
\end{split}
\end{equation}
We write $\Y^M(x_2)(m)=\sum_n \Y^M_n(m)x_2^{-n-1}=\sum_n \sum_i m_{n;1, i} \otimes m_{n;2, i}x_2^{-n-1}$ (notice that $m_{n;2, i} \in V$). The second term on the left hand side of Equation \eqref{prop:cotranslation_mod_eq:2} can be written as
\begin{align*}
&(\on{id}_M \otimes \on{id}_V \otimes \rho)(\on{id}_M \otimes T^\beta)(\Y_0^M \otimes \on{id}_V)\Y^M(x_2)(m) \\
&=\sum_n\sum_i (\on{id}_M \otimes \on{id}_V \otimes \rho)(\on{id}_M \otimes T^\beta)(\Y_0^M \otimes \on{id}_V) (m_{n;1, i} \otimes m_{n; 2, i})x_2^{-n-1} \\
&=\sum_n\sum_i\sum_j (\on{id}_M \otimes \on{id}_V \otimes \rho)(\on{id}_M \otimes T^\beta) \big((m_{n;1, i})_{0; 1, j} \otimes (m_{n;1, i})_{0; 2, j} \otimes m_{n; 2, i}\big)x_2^{-n-1} \\
&=\sum_n\sum_i\sum_j \beta(|(m_{n;1, i})_{0; 2, j}|, |m_{n; 2, i}|) (\on{id}_M \otimes \on{id}_V \otimes \rho) \big((m_{n;1, i})_{0; 1, j} \otimes  m_{n; 2, i} \otimes (m_{n;1, i})_{0; 2, j}\big)x_2^{-n-1} \\
&=\sum_n\sum_i\sum_j \beta(|(m_{n;1, i})_{0; 2, j}|, |m_{n; 2, i}|) \big((m_{n;1, i})_{0; 1, j} \otimes  m_{n; 2, i} \otimes \rho((m_{n;1, i})_{0; 2, j})\big)x_2^{-n-1} \\
&=\sum_n\sum_i\sum_j \beta(|\rho((m_{n;1, i})_{0; 2, j})|+2\gamma_0, |m_{n; 2, i}|) \big((m_{n;1, i})_{0; 1, j} \otimes \rho((m_{n;1, i})_{0; 2, j}) \otimes  m_{n; 2, i} \big)x_2^{-n-1}
\end{align*}
as $|\rho(v)|=|v|-|x|=|v|-2\gamma_0$ for any $v \in V$. But if $\rho(v) \neq 0$, then $|\rho(v)|=0$, and so
\begin{align*}
&(\on{id}_M \otimes \on{id}_V \otimes \rho)(\on{id}_M \otimes T^\beta)(\Y_0^M \otimes \on{id}_V)\Y^M(x_2)(m) \\
&=\sum_n\sum_i\sum_j \beta(2\gamma_0, |m_{n; 2, i}|) \big((m_{n;1, i})_{0; 1, j} \otimes \rho((m_{n;1, i})_{0; 2, j}) \otimes  m_{n; 2, i} \big)x_2^{-n-1} \\
&=\sum_n\sum_i\sum_j  \big((m_{n;1, i})_{0; 1, j} \otimes \rho((m_{n;1, i})_{0; 2, j}) \otimes  m_{n; 2, i} \big)x_2^{-n-1} \quad \text{by the assumption on $\beta$}\\
&= \big((\on{id}_M \otimes \rho)\Y_0^M \otimes \on{id}_V\big) \Y^M(x_2)(m)\\
\end{align*}
We can thus rewrite Equation \eqref{prop:cotranslation_mod_eq:2} as 
\begin{equation}\label{prop:cotranslation_mod_eq:3}
\begin{split}
\Y^M(x_2)(\on{id}_M \otimes \rho)\Y^M_0(m)-\big((\on{id}_M \otimes \rho)\Y_0^M \otimes \on{id}_V\big) \Y^M(x_2)(m)\\
=(\on{id}_M \otimes \left((\on{id}_V \otimes \rho)\Y_0\right))\Y^M(x_2)(m).
\end{split}
\end{equation}
Then using the facts that $\D=(\on{id}_V \otimes \rho)\Y_0$ and that $\D^M=(\on{id}_M \otimes \rho)\Y^M_0$, Equation \eqref{prop:cotranslation_mod_eq:3} becomes
\begin{equation}\label{prop:cotranslation_mod_eq:4}
\Y^M(x_2)\D^M(m)-\big(\D^M \otimes \on{id}_V\big) \Y^M(x_2)(m)=(\on{id}_M \otimes \D )\Y^M(x_2)(m).
\end{equation}
Finally, we know that Equation \eqref{eq:coderivation_mod_2} is satisfied, and thus Equation \eqref{prop:cotranslation_mod_eq:4} becomes equivalent to Equation \eqref{eq:coderivation_mod_1}, which concludes the proof.
\end{proof}

\begin{remark}\label{rem:bilinear_form_mod}
As it was the case in Remark \ref{rem:bilinear_form} in the vertex algebra/coalgebra context, if there exists a skew-symmetric or symmetric bilinear form on $\Gamma$, we can construct $\beta$ satisfying the conditions given in Propositions \ref{prop:translation_mod} and \ref{prop:cotranslation_mod}.
\end{remark}

\section{The $C_2$-algebra and its associated modules}\label{sec:C2}
\subsection{The structure of the $C_2$-algebra}
Given a vertex algebra $V$, Zhu proved in \cite{Zhu} that the quotient $R(V)=V/C_2(V)$ with $C_2(V)=\on{Span}\left\{u_{-2}v \ | \ u, v \in V \right\}$ is a Poisson algebra where the Poisson structure is given by:
\[
\begin{array}{ccc}
\overline{u}. \overline{v}& =&\overline{u_{-1}v}, \\[5pt]
\{\overline{u},\overline{v}\}& =&\overline{u_{0}v}.
\end{array}
\]
We can obtain something similar in the $G_\Gamma$-module context.

Let $V$ be $(G_\Gamma, \beta, \gamma_0)$-vertex algebra and set $C_2(V)=\on{Im}(Y_{-2})$. Consider the space
\[
R(V)=\on{CoKer}Y_{-2}=V/C_2(V). 
\]
It is a $G_\Gamma$-module because $|Y_{-2}|=-|x|=-2\gamma_0$.

We know from \ref{def:derivation} in Definition \ref{def:va} that $Y_n   (D \otimes \on{id}_V)=-n Y_{n-1}$ for all $n \in \mathbb{Z}$, so we have
\begin{align}\label{eq:Y_{-2-n}}
Y_{-2-n}=\frac{1}{(n+1)!} Y_{-2}   (D \otimes \on{id}_V)^n
\end{align}
for all $n \geq 0$.  This leads to the following lemma:

\begin{lemma}\label{lem:Y_{-2-n}}
Let $V$ be $(G_\Gamma, \beta, \gamma_0)$-vertex algebra. Then $\on{Im}(Y_{-2-n}) \subseteq C_2(V)$ for all $n \geq 0$.
\end{lemma}

\begin{definition}\label{def:algebra}
A \textbf{$G_\Gamma$-algebra} is a triple $(A, \varpi, \mathbf{1})$ such that $A$ is a $G_\Gamma$-module, and $\varpi: A \otimes A \longrightarrow A$ (product) and $\mathbf{1}: \cc \longrightarrow A$ (unit) are homogeneous linear maps satisfying:
\begin{enumerate}[itemsep=5pt]
\item $\varpi   (\on{id}_A \otimes \varpi) = \varpi   (\varpi \otimes \on{id}_A) $ (associativity).

\item $\varpi   (\on{id}_A \otimes \mathbf{1}) =\on{id}_A=\varpi   (\mathbf{1} \otimes \on{id}_A)$ (unit).
\end{enumerate}
A $G_\Gamma$-algebra $(A, \varpi, \mathbf{1})$ is $\beta$-\textbf{commutative} if $\varpi   T^\beta=\varpi$.
\end{definition}

\begin{remark}
Notice that in the commutativity statement, $T^\beta$ is not the braiding. So we use the term $\beta$-commutativity to avoid confusion with the usual commutativity involving the braiding.
\end{remark}

The Jacobi identity \eqref{eq:Jacobi_components} with $l=0$, $m=-1$, $n=-2$ and with $l=-2$, $m=1$, $n=-2$ lead to 
\[
\left\{\begin{array}{lcl}
Y_{-1}   ( \on{id}_V \otimes Y_{-2}) \mod C_2(V)& = & 0, \\[5pt]
Y_{-1}   (Y_{-2} \otimes \on{id}_V)\mod C_2(V) & = & 0.
\end{array}\right.
\]
This shows that $Y_{-1}$ induces a homogeneous linear map
\[
\pi: R(V) \otimes R(V) \longrightarrow  R(V).
\] 
Moreover, the composition of $\mathbbm{1}$ with the projection $V \longrightarrow R(V)$ gives a homogeneous linear map $\mathbf{1}_{R(V)}: \cc \longrightarrow R(V)$.

We now prove the analogue of Zhu's result:

\begin{theorem}\label{thm:C_2_algebra}
Let $V$ be a $(G_\Gamma, \beta, \gamma_0)$-vertex algebra. Then the tuple $(R(V), \pi, \mathbf{1}_{R(V)})$ is a $\beta$-commutative $G_\Gamma$-algebra. It will be referred as the \textbf{$C_2$-algebra} of $V$.
\end{theorem}

\begin{proof}
By setting $l=n=-1$ and $m=0$ in Equation \eqref{eq:Jacobi_components}, we get
\begin{align}\label{thm:C_2_algebra_eq:1}
Y_{-1}    (Y_{-1} \otimes \on{id}_{V}) =Y_{-1}   (\on{id}_{V} \otimes Y_{-1})  \mod C_2(V).
\end{align}
This proves the associativity of $\pi$.

We also have $ \sum_{n \in \mathbb{Z}}Y_n x^{-n-1}(\vac \otimes \on{id}_V)=\on{id}_V$ from the vacuum condition \eqref{eq:vacuum} in Definition \ref{def:va}, so
\[
Y_{-1}   (\vac \otimes \on{id}_V)=\on{id}_V.
\]
Moreover, the creation condition \eqref{eq:creation} in Definition \ref{def:va} implies that $Y_{-1}   (\on{id}_V \otimes \vac)=\on{id}_{V}$. Thus we get
\begin{align}\label{thm:C_2_algebra_eq:2}
Y_{-1}   (\vac \otimes \on{id}_V)=\on{id}_{V}=Y_{-1}   (\on{id}_V \otimes \vac),
\end{align}
proving that $\mathbf{1}_{R(V)}$ is a unit.

The Jacobi identity \eqref{eq:Jacobi_components} with $l=0$, $m=n=-1$ leads to 
\begin{align*}
Y_{-1}   (\on{id}_V \otimes Y_{-1})  =Y_{-1}   (\on{id}_V \otimes Y_{-1})   (T^\beta \otimes \on{id}_{V}) \mod C_2(V).
\end{align*}
If we compose the above equation on the right by $(\on{id}_V \otimes \on{id}_V \otimes \vac)$, we get
\begin{equation}\label{thm:C_2_algebra_eq:3}
\begin{split}
&Y_{-1}   (\on{id}_V \otimes Y_{-1})   (\on{id}_V \otimes \on{id}_V \otimes \vac)  = \\[5pt]
&Y_{-1}   (\on{id}_V \otimes Y_{-1})   (T^\beta \otimes \on{id}_{V})   (\on{id}_V \otimes \on{id}_V \otimes \vac) \mod C_2(V).
\end{split}
\end{equation}
With Equation \eqref{thm:C_2_algebra_eq:2}, we see that the left hand side of Equation \eqref{thm:C_2_algebra_eq:3}  becomes $ Y_{-1}$. The right hand side can be rewritten as $Y_{-1}   (\on{id}_V \otimes Y_{-1})   (T^\beta \otimes \on{id}_{V})   (\on{id}_V \otimes \on{id}_V \otimes \vac) =Y_{-1}   (\on{id}_V \otimes Y_{-1})   (\on{id}_V \otimes \on{id}_V \otimes \vac)   (T^\beta \otimes \on{id}_{\cc}) =Y_{-1}   T^\beta$. We then see that Equation \eqref{thm:C_2_algebra_eq:3} becomes
\begin{align}\label{thm:C_2_algebra_eq:4}
Y_{-1}   T^\beta= Y_{-1} \mod C_2(V).
\end{align}
This proves the $\beta$-commutativity.
\end{proof}

Given any $G_\Gamma$-module $V$ and map of sets $\beta:\Gamma \times \Gamma \longrightarrow \cc$, we define the following permutation:
\[
\xi^\beta=(\on{id}_V \otimes T^\beta)   (T^\beta \otimes \on{id}_V):  V^{\otimes 3}  \longrightarrow  V^{\otimes 3}.
\]

\begin{definition}\label{def:Lie_algebra}
A \textbf{$(G_\Gamma, \beta)$-Lie algebra of degree $\gamma$} is a tuple $(L, \varpi_0)$ where $L$ is a $G_\Gamma$-module (over a field of characteristic not $2$) and $\varpi_0 : L \otimes L \longrightarrow L$ is a linear map of degree $\gamma$ such that
\begin{enumerate}[itemsep=5pt]
\item $\varpi_0+ \varpi_0   T^\beta=0$ (skew-symmetry).

\item $\varpi_0   (\varpi_0 \otimes \on{id_L} )   (\on{id}_{L^{\otimes 3}} + \xi^\beta + (\xi^\beta)^2)=0$ (Jacobi identity).
\end{enumerate}
\end{definition}

\begin{remark}
For $\beta=1$, the above definition is that of a Lie algebra object in the symmetric tensor category of $G_\Gamma$-modules. If $\beta=0$, then we obtain a commutative Lie algebra.
\end{remark}

The Jacobi identity \eqref{eq:Jacobi_components} with $l=m=0$, $n=-2$ and with $l=-2$, $m=n=0$ lead to 
\begin{empheq}[left = \empheqlbrace]{align}
&Y_0   (\on{id}_V  \otimes Y_{-2}) \mod C_2(V) =  0, \label{eq:Y_0_C_2:1} \\[5pt]
&Y_0   (Y_{-2} \otimes \on{id}_V)  \hspace*{-3pt} \mod C_2(V) =  0. \label{eq:Y_0_C_2:2}
\end{empheq}
This shows that $Y_0$ induces a linear map of degree $0$
\[
\pi_0: R(V) \otimes R(V) \longrightarrow  R(V)
\] 
as $|x^{-1}|=0$ (see Section \ref{sec:A.2}). 

\begin{definition}\label{def:Poisson}
A \textbf{$(G_\Gamma, \beta)$-Poisson algebra of degree} $\gamma \in \Gamma$ is a tuple $(A, \varpi, \mathbf{1}, \varpi_0)$ such that:
\begin{enumerate}[itemsep=5pt]
\item $(A, \mathbf{1}, \varpi)$ is a $G_\Gamma$-algebra.

\item $(A, \varpi_0)$ is a $(G_\Gamma, \beta)$-Lie algebra of degree $\gamma$.

\item $\varpi_0   (\on{id}_A \otimes \varpi) = \varpi   (\varpi_0 \otimes \on{id}_A)-\varpi   (\on{id}_A \otimes \varpi_0)   \xi^\beta $.
\end{enumerate}
\end{definition}

\begin{remark}
The case $\beta=1$ gives a Poisson algebra object in the symmetric tensor category of $G_\Gamma$-modules. Moreover, the case $\beta=0$ is then a Poisson algebra with trivial Lie bracket.
\end{remark}

We can now prove the next result.

\begin{theorem}\label{thm:C_2_Poisson}
Let $V$ be a $(G_\Gamma, \beta, \gamma_0)$-vertex algebra. Assume that $\beta$ satisfies Relation \eqref{relation_2}. Then the tuple $(R(V), \pi, \mathbf{1}_{R(V)}, \pi_0)$ is a $\beta$-commutative $(G_\Gamma, \beta)$-Poisson algebra of degree $0$.
\end{theorem}

\begin{proof}
In the Jacobi identity \eqref{eq:Jacobi_components}, set $l=1$, $m=n=-1$. We get
\[
\begin{array}{r}
Y_0   (\on{id}_V \otimes Y_{-1})-Y_{-1}   (\on{id}_V \otimes Y_{0}) +Y_0   (\on{id}_V \otimes Y_{-1})   (T^\beta \otimes \on{id}_V) \\[5pt]
- Y_{-1}   (\on{id}_V \otimes Y_{0})   (T^\beta \otimes \on{id}_V) \mod C_2(V)=0.
\end{array}
\]
We compose with $(\on{id}_{V} \otimes \on{id}_{V} \otimes \vac)$ on the right of the above equation, and using the fact that $Y_0   (\on{id}_V \otimes \vac)=0$ and $Y_{-1}   (\on{id}_V \otimes \vac)=\on{id}_V$ coming from the creation property \eqref{eq:creation} in Definition \ref{def:va}, we can rewrite the equation as
\begin{align}\label{thm:C_2_Poisson_eq:1}
Y_0 + Y_0   T^\beta \mod C_2(V) =0.
\end{align}
Thus $\pi_0$ is skew-symmetric.

We now consider the Jacobi identity \eqref{eq:Jacobi_components} with $l=m=n=0$:
\begin{align}\label{thm:C_2_Poisson_eq:2}
Y_0   (\on{id}_V \otimes Y_0)-Y_0   (\on{id}_V \otimes Y_0)   (T^\beta \otimes \on{id}_V)= Y_0   (Y_0 \otimes \on{id}_V).
\end{align}
For homogeneous elements $v_i \in V_{\gamma_i}$, $i=1, 2, 3$, we have
\[
\begin{array}{rcl}
(Y_0 \otimes \on{id}_V)   \xi^\beta (v_1 \otimes v_2 \otimes v_3)  & = & \beta(\gamma_1, \gamma_{2})\beta(\gamma_1, \gamma_{3}) (Y_0 \otimes \on{id}_V)(v_2 \otimes v_3 \otimes v_1)\\[5pt]
& = &\beta(\gamma_1, \gamma_{2}+ \gamma_{3}) Y_0(v_2 \otimes v_3) \otimes  v_1\\[5pt]
& = & \beta(\gamma_1, |Y_0(v_2 \otimes v_3) |) Y_0(v_2 \otimes v_3) \otimes  v_1\\[5pt]
& = &T^\beta( v_1 \otimes  Y_0(v_2 \otimes v_3))\\[5pt] 
& = &T^\beta   (\on{id}_V \otimes Y_0 )(v_1 \otimes v_2 \otimes v_3).
\end{array}
\]
It follows that $Y_0   (Y_0 \otimes \on{id}_V)   \xi^\beta = Y_0   T^\beta   (\on{id}_V \otimes Y_0 )  = -  Y_0   (\on{id}_V \otimes Y_0 ) \mod C_2(V)$ by Equation \eqref{thm:C_2_Poisson_eq:1}. It follows that
\begin{align}\label{thm:C_2_Poisson_eq:3}
Y_0   (\on{id}_{V} \otimes Y_0 )=-Y_0   (Y_0 \otimes \on{id}_{V})   \xi^\beta \mod C_2(V)
\end{align}
and
\begin{align}\label{thm:C_2_Poisson_eq:4}
Y_0   (\on{id}_{V} \otimes Y_0 )   \xi^\beta=-Y_0   (Y_0 \otimes \on{id}_{V})   (\xi^\beta)^2 \mod C_2(V).
\end{align}
Moreover,
\begin{align}
Y_0   (\on{id}_V & \otimes Y_0)    (T^\beta \otimes \on{id}_V)  \nonumber \\[5pt]
& = -Y_0   (\on{id}_V \otimes Y_0   T^\beta)   (T^\beta \otimes \on{id}_V) \mod C_2(V)  \nonumber \\[5pt]
& = - Y_0   (\on{id}_V \otimes Y_0)   (\on{id}_V \otimes T^\beta)   (T^\beta \otimes \on{id}_V)  \mod C_2(V) \nonumber \\[5pt]
& = - Y_0   (\on{id}_V \otimes Y_0)   \xi^\beta \mod C_2(V)  \nonumber \\[5pt]
& =Y_0   (Y_0 \otimes \on{id}_{V})   (\xi^\beta)^2 \mod C_2(V)  \quad \text{by Equation \eqref{thm:C_2_Poisson_eq:4}}.\label{thm:C_2_Poisson_eq:5}
\end{align} 
Then, by replacing the terms in Equation \eqref{thm:C_2_Poisson_eq:2} using Equations \eqref{thm:C_2_Poisson_eq:3} and \eqref{thm:C_2_Poisson_eq:5}, we get
\begin{align}\label{thm:C_2_Poisson_eq:6}
Y_0   (Y_0 \otimes \on{id}_V)   (\on{id}_{V^{\otimes 3}} + \xi^\beta + (\xi^\beta)^2) =0 \mod C_2(V).
\end{align}

From Definition \ref{def:Lie_algebra} and Equations \eqref{thm:C_2_Poisson_eq:1} and \eqref{thm:C_2_Poisson_eq:6}, we see that $(R(V), \pi_0)$ is a $(G_\Gamma, \beta)$-Lie algebra of degree $0$. We already know from Theorem \ref{thm:C_2_algebra} that $(R(V), m, \mathbf{1}_{R(V)})$ is a $\beta$-commutative $G_\Gamma$-algebra. It remains to prove the third condition in Definition \ref{def:Poisson}. 

Consider the Jacobi identity \eqref{eq:Jacobi_components} with $l=m=0$, $n=-1$:
\begin{align}\label{thm:C_2_Poisson_eq:7}
Y_0   (\on{id}_V \otimes Y_{-1}) - Y_{-1}   (\on{id}_V \otimes Y_{0})    (T^\beta \otimes \on{id}_V) = Y_{-1}   (Y_0 \otimes \on{id}_V).
\end{align}
Moreover
\begin{align*}
 Y_{-1}   (\on{id}_V \otimes Y_{0})    (T^\beta \otimes \on{id}_V) & =- Y_{-1}   (\on{id}_V \otimes Y_{0}   T^\beta)    (T^\beta \otimes \on{id}_V) \mod C_2(V)\\[5pt]
& =- Y_{-1}   (\on{id}_V \otimes Y_{0})  (\on{id}_V \otimes T^\beta)    (T^\beta \otimes \on{id}_V)  \ \on{mod} \ C_2(V) \\[5pt]
& =-  Y_{-1}   (\on{id}_V \otimes Y_{0})  \xi^\beta \mod C_2(V).
\end{align*} 
Therefore Equation \eqref{thm:C_2_Poisson_eq:7} leads to
\begin{align*}
 Y_0   (\on{id}_V \otimes Y_{-1}) +Y_{-1}   (\on{id}_V \otimes Y_{0})  \xi^\beta =Y_{-1}   (Y_0 \otimes \on{id}_V) \mod C_2(V),
\end{align*}
which then implies
\begin{align*}
\pi_0   (\on{id}_{R(V)} \otimes \pi)= \pi   (\pi_0 \otimes \on{id}_{R(V)}) - \pi   (\on{id}_{R(V)} \otimes \pi_{0})  \xi^\beta
\end{align*}
and this is exactly the third condition in Definition \ref{def:Poisson}.
\end{proof}

\subsection{The associated $C_2$-modules}

Let $V$ be $(G_\Gamma, \beta, \gamma_0)$-vertex algebra and $M$ a $(G_\Gamma, \beta, \gamma_0)$-left $V$-module. Set $C_2(M)=\on{Im}(Y^M_{-2})$ and consider the space
\[
R(M)=\on{CoKer}Y^M_{-2}=M/C_2(M). 
\]
It is a $G_\Gamma$-module because $|Y^M_{-2}|=-|x|=-2\gamma_0$.

By looking at the coefficient of $x_0^{-l-1}x_1^{-m-1}x_2^{-n-1}$ in the Jacobi identity \eqref{eq:Jacobi_mod} of a $(G_\Gamma, \beta, \gamma_0)$-left $V$-module $M$, we get
\begin{equation}\label{eq:Jacobi_mod_coeff}
\begin{split}
\scalebox{0.95}{$\displaystyle \sum_{i \geq 0}(-1)^i \binom{l}{i} Y^M_{m+l-i}  (\on{id}_V  \otimes Y^M_{n+i})-(-1)^l\sum_{i \geq 0} (-1)^i \binom{l}{i} Y^M_{n+l-i} (\on{id}_V \otimes Y^M_{m+i}) (T^\beta \otimes \on{id}_{M} )$} \\[5pt]
\scalebox{0.95}{$  \displaystyle    =\sum_{i \geq 0} \binom{m}{i} Y^M_{m+n-i} (Y_{l+i} \otimes \on{id}_M)$}.
\end{split}
\end{equation}
We know from \ref{def:derivation_mod} in Definition \ref{def:va_mod} that $Y^M_n   (D \otimes \on{id}_M)=-n Y^M_{n-1}$ for all $n \in \mathbb{Z}$, so we have
\begin{align}\label{eq:Y^M_{-2-n}}
Y^M_{-2-n}=\frac{1}{(n+1)!} Y^M_{-2}   (D \otimes \on{id}_M)^n
\end{align}
for all $n \geq 0$.  This leads to the following lemma:

\begin{lemma}\label{lem:Y^M_{-2-n}}
Let $V$ be $(G_\Gamma, \beta, \gamma_0)$-vertex algebra and $M$ a $(G_\Gamma, \beta, \gamma_0)$-left $V$-module. Then $\on{Im}(Y^M_{-2-n}) \subseteq C_2(M)$ for all $n \geq 0$.
\end{lemma}

\begin{definition}\label{def:algebra_mod}
Let $(A, \varpi, \mathbf{1}_A)$ be a $G_\Gamma$-algebra. A $G_\Gamma$\textbf{-left module} for $A$ is a $G_\Gamma$-module $M$ with a homogeneous linear map $\varpi^M: A \otimes M \longrightarrow M$ such that
\[
\begin{tikzcd}
A \otimes A \otimes M \arrow[d, "\varpi \otimes \on{id}_M"']  \arrow[r, "\on{id}_A \otimes \varpi^M"] & A \otimes M \arrow[d, "\varpi^M"]  \\[10pt]
A \otimes M \arrow[r, "\varpi^M", labels=below] & M
\end{tikzcd} \text{ and } \begin{tikzcd}
M \arrow[dr, "\on{id}_{M}"']  \arrow[r, "\mathbf{1}_A \otimes \on{id}_M"] & A \otimes M \arrow[d, "\varpi^M"]  \\[10pt]
&  M
\end{tikzcd}
\]
are commutative diagrams.
\end{definition}

The Jacobi identity \eqref{eq:Jacobi_mod_coeff} with $l=0$, $m=-1$, $n=-2$ and with $l=-2$, $m=1$, $n=-2$ lead to 
\[
\left\{\begin{array}{lcl}
Y^M_{-1}   ( \on{id}_V \otimes Y^M_{-2}) \mod C_2(M)& = & 0, \\[5pt]
Y^M_{-1}   (Y_{-2} \otimes \on{id}_M)\mod C_2(M) & = & 0.
\end{array}\right.
\]
This shows that $Y^M_{-1}$ induces a homogeneous linear map
\[
\pi^M: R(V) \otimes R(M) \longrightarrow  R(M).
\]

\begin{theorem}\label{thm:C_2_algebra_mod}
Let $V$ be a $(G_\Gamma, \beta, \gamma_0)$-vertex algebra and $M$ a $(G_\Gamma, \beta, \gamma_0)$-left $V$-module. Then the tuple $(R(M), \pi^M)$ is a $G_\Gamma$-left module for the $G_\Gamma$-algebra $(R(V), \pi, \mathbf{1}_{R(V)})$.
\end{theorem}

\begin{proof}
By setting $l=n=-1$ and $m=0$ in Equation \eqref{eq:Jacobi_mod_coeff}, we get
\begin{align}\label{algebra_eq:1}
Y^M_{-1}    (Y_{-1} \otimes \on{id}_{M}) =Y^M_{-1}   (\on{id}_{V} \otimes Y^M_{-1})  \mod C_2(M).
\end{align}
This proves the commutativity of the first diagram in Definition \ref{def:algebra_mod}.

We also have $ \sum_{n \in \mathbb{Z}}Y^M_n x^{-n-1}(\vac \otimes \on{id}_M)=\on{id}_M$ from the vacuum condition \eqref{eq:vacuum_mod} in Definition \ref{def:va_mod}, so
\[
Y^M_{-1}   (\vac \otimes \on{id}_M)=\on{id}_M,
\]
proving that the second diagram in Definition \ref{def:algebra_mod} is commutative.
\end{proof}

\begin{definition}\label{def:Lie_algebra_mod}
Let $(L, \varpi_0)$ be a $(G_\Gamma, \beta)$-Lie algebra of degree $\gamma$. A $(G_\Gamma, \beta)$\textbf{-left Lie module} for $L$ is a tuple $(M, \varpi^M_0)$ where $M$ is a $G_\Gamma$-module (over a field of characteristic not $2$) and $\varpi^M_0 : L \otimes M \longrightarrow M$ is a linear map of degree $\gamma$ such that
\[
\varpi_0^M   (\varpi_0 \otimes \on{id}_M)=\varpi_0^M   (\on{id}_L \otimes \varpi_0^M)-\varpi_0^M   (\on{id}_L \otimes \varpi_0^M)   (T^\beta \otimes \on{id}_M).
\]
\end{definition}

The Jacobi identity \eqref{eq:Jacobi_mod_coeff} with $l=m=0$, $n=-2$ and with $l=-2$, $m=n=0$ lead to 
\begin{empheq}[left = \empheqlbrace]{align}
&Y^M_0   (\on{id}_V  \otimes Y^M_{-2}) \mod C_2(M) =  0, \label{eq:Y_0M_C_2:1} \\[5pt]
&Y^M_0   (Y_{-2} \otimes \on{id}_M)  \hspace*{-3pt} \mod C_2(M) =  0. \label{eq:Y_0M_C_2:2}
\end{empheq}

This shows that $Y^M_0$ induces a linear map of degree $0$
\[
\pi^M_0: R(V) \otimes R(M) \longrightarrow  R(M).
\]

By setting $l=m=n=0$ in the Jacobi identity \eqref{eq:Jacobi_mod_coeff} we get:
\begin{align*}
Y_0^M   (\on{id}_V \otimes Y_0^M) - Y_0^M   (\on{id}_V \otimes Y_0^M)   (T^\beta \otimes \on{id}_M) =  Y_0^M   (Y_0 \otimes \on{id}_M).
\end{align*}
We then mod out by $C_2(V)$ and $C_2(M)$ and obtain:
\begin{align}\label{eq:algebra_eq:2}
\scalebox{0.95}{$\pi_0^M   (\on{id}_{R(V)} \otimes \pi_0^M) - \pi_0^M   (\on{id}_{R(V)} \otimes \pi_0^M)   (T^\beta \otimes \on{id}_{R(M)}) = \pi_0^M   (\pi_0 \otimes \on{id}_{R(M)}).$}
\end{align}
Using Equation \eqref{eq:algebra_eq:2} as well as Definition \ref{def:Lie_algebra_mod} and the proof of Theorem \ref{thm:C_2_Poisson}, we immediately verify the next result:

\begin{theorem}\label{thm:C_2_Lie_algebra_mod}
Let $V$ be a $(G_\Gamma, \beta, \gamma_0)$-vertex algebra and $M$ a $(G_\Gamma, \beta, \gamma_0)$-left $V$-module. Assume that $\beta$ satisfies Relation \eqref{relation_2}. Then $(R(M), \pi_0^M)$ is a $(G_\Gamma, \beta)$-left Lie module for the $(G_\Gamma, \beta)$-Lie algebra $(R(V), \pi_{0})$ of degree $0$.
\end{theorem}

The notion of Poisson module was introduced in \cite{Farkas}. We formulate below the definition in the context of $G_\Gamma$-modules:

\begin{definition}\label{def:Poisson_mod}
Let $(A, \varpi, \mathbf{1}_A, \varpi_0)$ be a $(G_\Gamma, \beta)$-Poisson algebra of degree $\gamma$. A $(G_\Gamma, \beta)$\textbf{-left Poisson module} for $A$ is a triple $(M, \varpi^M, \varpi^M_0)$ such that $M$ is a $G_\Gamma$-module, $\varpi^M: A \otimes M \longrightarrow M $ and $\varpi^M_0: A \otimes M \longrightarrow M$ are homogeneous linear maps of degree $0$ and $\gamma$ respectively, and they satisfy
\begin{enumerate}[itemsep=5pt]
\item $(M, \varpi^M)$ is a left module for the algebra $(A, \varpi, \mathbf{1}_A)$.

\item $(M, \varpi^M_0)$ is a left Lie module for the Lie algebra $(A, \varpi_0)$ of degree $\gamma$.

\item \label{def:Poissonmod3} $\varpi_0^M   (\on{id}_A \otimes \varpi^M) = \varpi^M   (\varpi_0 \otimes \on{id}_M)+ \varpi^M   (\on{id}_A \otimes \varpi_0^M)   (T^\beta \otimes \on{id}_M)$.

\item \label{def:Poissonmod4} $\varpi_0^M   (\varpi \otimes \on{id}_M) = \varpi^M   (\on{id}_A \otimes \varpi_0^M) + \varpi^M   (\on{id}_A \otimes \varpi_0^M)   (T^\beta \otimes \on{id}_M)$.
\end{enumerate}
\end{definition}

We now have a result similar to Theorem \ref{thm:C_2_Poisson} but for modules:

\begin{theorem}\label{thm:C_2_Poisson_mod}
Set $V$ a $(G_\Gamma, \beta, \gamma_0)$-vertex algebra and $M$ a $(G_\Gamma, \beta, \gamma_0)$-left $V$-module. Assume that $\beta$ satisfies Relation \eqref{relation_2}. Then the tuple $(R(M), \pi^M, \pi^M_0)$ is a $(\Gamma, \beta)$-left Poisson module for the $\beta$-commutative $(G_\Gamma, \beta)$-Poisson algebra $(R(V), \pi, \mathbf{1}_{R(V)}, \pi_{0})$ of degree $0$.
\end{theorem}

\begin{proof}
We know that $R(V)$ is a $\beta$-commutative $(G_\Gamma, \beta)$-Poisson algebra by Theorem \ref{thm:C_2_Poisson}. The first two points in Definition \ref{def:Poisson_mod} are verified by Theorems \ref{thm:C_2_algebra_mod} and \ref{thm:C_2_Lie_algebra_mod}. Then, in the Jacobi identity \eqref{eq:Jacobi_mod_coeff}, set $l=m=0$, $n=-1$. We get
\begin{align*}
Y_0^M   (\on{id}_V \otimes Y_{-1}^M) - Y_{-1}^M   (\on{id}_V \otimes Y_0^M)   (T^\beta \otimes \on{id}_M) =Y_{-1}^M   (Y_0 \otimes \on{id}_M).
\end{align*}
We mod out by $C_2(V)$ and $C_2(M)$ to obtain
\begin{align*}
\pi_0^M   (\on{id}_{R(V)} \otimes \pi^M) - \pi^M   (\on{id}_{R(V)} \otimes \pi_0^M)   (T^\beta \otimes \on{id}_{R(M)}) = \pi^M   (\pi_0 \otimes \on{id}_{R(M)}),
\end{align*}
which is exactly the third relation in Definition \ref{def:Poisson_mod}. The last relation is obtained by setting $l=-1$ and $m=n=0$  in the Jacobi identity \eqref{eq:Jacobi_mod_coeff}.
\end{proof}

\section{The $C_2$-coalgebra and its associated comodules}\label{sec:coC2}
\subsection{The structure of the $C_2$-coalgebra}
For a $(G_\Gamma, \beta, \gamma_0)$-vertex coalgebra $V$, consider the space
\[
\R(V)=\on{Ker}\Y_{-2}. 
\]
It is a $G_\Gamma$-module because $|\Y_{-2}|=-|x|=-2\gamma_0$.

We will show that this subspace plays a role similar to that of the $C_2$-algebra of a $(G_\Gamma, \beta, \gamma_0)$-vertex algebra.

By looking at the coefficient of $x_0^{-l-1}x_1^{-m-1}x_2^{-n-1}$ in the co-Jacobi identity \eqref{eq:coJacobi} of a $V$, we get
\begin{equation}\label{eq:coJacobi_components}
\begin{split}
\scalebox{0.92}{$\displaystyle \sum_{i \geq 0}(-1)^i \binom{l}{i}( \Y_{n+i} \otimes \on{id}_{V})   \Y_{m+l-i}-(-1)^l\sum_{i \geq 0} (-1)^i \binom{l}{i}( \on{id}_{V} \otimes T^\beta)    (\Y_{m+i} \otimes \on{id}_{V})\Y_{n+l-i}$} \\[5pt]
\scalebox{0.92}{$  \displaystyle  =\sum_{i \geq 0} \binom{m}{i}(\on{id}_{V} \otimes \Y_{l+i})   \Y_{m+n-i}$}.
\end{split}
\end{equation}
We know from \ref{def:coderivation} in Definition \ref{def:cova} that $(\on{id}_V \otimes \D )   \Y_n=-n \Y_{n-1}$ for all $n \in \mathbb{Z}$, so we have
\begin{align}\label{eq:YY_{-2-n}}
\Y_{-2-n}=\frac{1}{(n+1)!}(\on{id}_V \otimes \D)^n   \Y_{-2}
\end{align}
for all $n \geq 0$. This leads to the following lemma:

\begin{lemma}\label{lem:YY_{-2-n}}
Let $V$ be $(G_\Gamma, \beta, \gamma_0)$-vertex coalgebra. Then $\R(V) \subseteq \on{Ker}\Y_{-2-n}$ for all $n \geq 0$. 
\end{lemma}

\begin{definition}\label{def:coalgebra}
A \textbf{$G_\Gamma$-coalgebra} is a triple $(C,\delta, c)$ such that $C$ is a $G_\Gamma$-module, and $\delta: C \longrightarrow C \otimes C$ (coproduct) and $c:C \longrightarrow \cc$ (counit) are homogeneous linear maps satisfying:
\begin{enumerate}[itemsep=5pt]
\item $(\on{id}_C \otimes \delta)   \delta=(\delta \otimes \on{id}_C)    \delta$ (coassociativity).

\item $(\on{id}_C \otimes c)    \delta=\on{id}_C=(c\otimes \on{id}_C)    \delta$ (counit).
\end{enumerate}
A $G_\Gamma$-coalgebra $(C,\delta, c)$ is $\beta$-\textbf{cocommutative} if $T^\beta   \delta  =\delta$.
\end{definition}

\begin{remark}
As we did in the algebra case, we use the term $\beta$-cocommutativity to emphasise that we use the map $T^\beta$ and not the braiding in the cocommutativity formulation.
\end{remark}

We will need the following lemma:
\begin{lemma}\label{lem:intersection_vector_spaces}
Given vector spaces $W \subseteq V$, we have
\[
(W \otimes V) \cap (V \otimes W)= W \otimes W.
\]
\end{lemma}

\begin{proof}
Consider $(w_i)_{i \in I}$ a basis of $W$ and complete it into $(v_j)_{j \in J}$ a basis of $V$, where $w_i=v_i$ and $I \subseteq J$.

For $x \in W \otimes V$, we can write $x=\sum_{i, j}x_{i, j} w_i \otimes v_j$, and for $x \in V \otimes W$, we can write $x=\sum_{i, j}x_{j, i}' v_j \otimes w_i$. So if $x \in (W \otimes V) \cap (V \otimes W)$, then
\[
\sum_{i, j}x_{i, j} w_i \otimes v_j=\sum_{i, j}x_{j, i}' v_j \otimes w_i.
\]
Thus, if on the left hand side $x_{i, j} \neq 0$, then $w_i \otimes v_j$ should appear on the right hand side, meaning that $v_j \in W$. Likewise if $x_{j, i}' \neq 0$. So the components with non zero coefficients are in $W \otimes W$, implying that $(W \otimes V) \cap (V \otimes W) \subseteq W \otimes W$. The other inclusion is obvious.
\end{proof}

The co-Jacobi identity \eqref{eq:coJacobi_components} with $l=0$, $m=-1$, $n=-2$ and with $l=-2$, $m=1$, $n=-2$ lead to 
\[
\left\{\begin{array}{lcl}
(\Y_{-2} \otimes \on{id}_{V})   \Y_{-1}(\on{Ker}\Y_{-2})& = & 0, \\[5pt]
(\on{id}_{V} \otimes \Y_{-2})   \Y_{-1}(\on{Ker}\Y_{-2})& = & 0.
\end{array}\right.
\]
Set $v \in \on{Ker}\Y_{-2}$. We write $\Y_{-1}(v)=\sum_{i}v_{1,i} \otimes v_{2, i}$ with $\{v_{1, i}\}$ a linearly independent family. Then $0=(\on{id}_{V} \otimes \Y_{-2})   \Y_{-1}(v)=\sum_{i}v_{1,i} \otimes \Y_{-2}(v_{2,i})$. But as the $v_{1,i}$'s are linearly independent, it follows that $\Y_{-2}(v_{2, i})=0$ for all $i$. Hence $\Y_{-1}(\on{Ker}\Y_{-2}) \subseteq V \otimes \on{Ker}\Y_{-2}$. By doing the same with $\{v_{2, i}\}$ linearly independent, we get $\Y_{-1}(\on{Ker}\Y_{-2}) \subseteq  \on{Ker}\Y_{-2} \otimes V$. Thus
\[
\Y_{-1}(\on{Ker}\Y_{-2}) \subseteq  (V \otimes \on{Ker}\Y_{-2}) \cap (\on{Ker}\Y_{-2} \otimes V).
\]
By Lemma \ref{lem:intersection_vector_spaces} and the previously obtained inclusion, we get
\begin{align}\label{coalgebra_eq:5}
\Y_{-1}(\on{Ker}\Y_{-2}) \subseteq  \on{Ker}\Y_{-2} \otimes \on{Ker}\Y_{-2}.
\end{align}
Hence there is a homogeneous linear map
\[
\Delta=\left.\Y_{-1}\right._{\mkern 1mu \vrule height 2ex\mkern2mu \R(V)}: \R(V) \longrightarrow  \R(V) \otimes \R(V).
\] 
As $\epsilon$ is homogeneous, its restriction $c_{\R(V)}=\left.\epsilon\right._{\mkern 1mu \vrule height 2ex\mkern2mu \R(V)}$ is also homogeneous.

We now prove the following result:

\begin{theorem}\label{thm:C_2_coalgebra}
Let $V$ be a $(G_\Gamma, \beta, \gamma_0)$-vertex coalgebra. Then the tuple $(\R(V), \Delta, c_{\R(V)})$ is a $\beta$-cocommutative $G_\Gamma$-coalgebra.
\end{theorem}

\begin{proof}
By setting $l=n=-1$, $m=0$, and $v \in \on{Ker}\Y_{-2}$ in Equation \eqref{eq:coJacobi_components}, we get
\begin{align}\label{coalgebra_eq:1}
( \Y_{-1} \otimes \on{id}_{V})    \Y_{-1}(v)=(\on{id}_{V} \otimes \Y_{-1})   \Y_{-1}(v).
\end{align}
This proves the coassociativity of $\Delta$.

We see from the covacuum condition \eqref{eq:covacuum} in Definition \ref{def:cova} that $(\on{id}_V \otimes \epsilon)   \Y_{-1}= \on{id}_{V}$. Moreover, the cocreation condition \eqref{eq:cocreation} in Definition \ref{def:cova} implies that $(\epsilon \otimes \on{id}_V)   \Y_{-1}=\on{id}_{V}$. Thus we get
\begin{align}\label{coalgebra_eq:2}
(\on{id}_V \otimes \epsilon)   \Y_{-1}=\on{id}_{V}=(\epsilon \otimes \on{id}_V)    \Y_{-1},
\end{align}
proving that $c_{\R(V)}$ is a counit.

Set $v \in \on{Ker}\Y_{-2}$. The co-Jacobi identity \eqref{eq:coJacobi_components} with $l=0$, $m=n=-1$ applied to $v$ leads to 
\begin{align*}
( \Y_{-1} \otimes \on{id}_{V})    \Y_{-1}(v)=( \on{id}_{V} \otimes T^\beta)     (\Y_{-1} \otimes \on{id}_{V})   \Y_{-1}(v).
\end{align*}
If we compose on the left by $(\epsilon \otimes \on{id}_V \otimes \on{id}_V)$, we get
\begin{equation}\label{coalgebra_eq:3}
\begin{split}
(\epsilon \otimes \on{id}_V \otimes \on{id}_V) &   ( \Y_{-1} \otimes \on{id}_{V})   \Y_{-1}(v) \\[5pt]
&=(\epsilon \otimes \on{id}_V \otimes \on{id}_V)   ( \on{id}_{V} \otimes T^\beta)   (\Y_{-1} \otimes \on{id}_{V})   \Y_{-1}(v).
\end{split}
\end{equation}
With Equation \eqref{coalgebra_eq:2}, we see that the left hand side of Equation \eqref{coalgebra_eq:3}  becomes $ \Y_{-1}(v)$. On the right hand side, we permute $(\epsilon \otimes \on{id}_V \otimes \on{id}_V)$ and $( \on{id}_{V} \otimes T^\beta)$ and using Equation \eqref{coalgebra_eq:2}, the right hand side becomes $T^\beta   \Y_{-1}(v)$. It follows that Equation \eqref{coalgebra_eq:3} becomes
\begin{align}\label{coalgebra_eq:4}
T^\beta   \Y_{-1}(v)= \Y_{-1}(v),
\end{align}
which proves the $\beta$-cocommutativity.
\end{proof}

The notion dual to that of a Lie algebra is the Lie coalgebra.

\begin{definition}\label{def:Lie_coalgebra}
A \textbf{$(G_\Gamma, \beta)$-Lie coalgebra of degree $\gamma$} is a tuple $(C, \delta_0)$ where $C$ is a $G_\Gamma$-module (over a field of characteristic not $2$) and $\delta_0 :C \longrightarrow C \otimes C$ is a linear map of degree $\gamma$ such that
\begin{enumerate}[itemsep=5pt]
\item $\delta_0+T^\beta   \delta_0=0$ (skew-symmetry).

\item $(\on{id}_{C^{\otimes 3}} + \xi^\beta + (\xi^\beta)^2)   ( \on{id}_C \otimes \delta_0)    \delta_0=0$ (co-Jacobi identity).
\end{enumerate}
\end{definition}

\begin{remark}
\begin{enumerate}[wide]
\item There is a general definition for a Lie coalgebra as a vector space over a field, whatever the characteristic may be, but the first condition is different and less practical (see \cite{Michaelis}). Moreover, in the characteristic $0$ case, the two definitions are equivalent.

\item The case $\beta=0$ corresponds to a trivial Lie cobracket.
\end{enumerate}
\end{remark}

The co-Jacobi identity \eqref{eq:coJacobi_components} with $l=m=0$, $n=-2$ and with $l=-2$, $m=n=0$ lead to 
\[
\left\{\begin{array}{lcl}
(\Y_{-2} \otimes \on{id}_{V})   \Y_{0}(\on{Ker}\Y_{-2})& = & 0, \\[5pt]
(\on{id}_{V} \otimes \Y_{-2})   \Y_{0}(\on{Ker}\Y_{-2})& = & 0.
\end{array}\right.
\]
With a reasoning similar to the one used for $\Y_{-1}$, we show that
\begin{align}\label{coalgebra_eq:6}
\Y_{0}(\on{Ker}\Y_{-2}) \subseteq  \on{Ker}\Y_{-2} \otimes \on{Ker}\Y_{-2}.
\end{align}
Hence we get a homogeneous linear map 
\[
\Delta_0=\left.\Y_{0}\right._{\mkern 1mu \vrule height 2ex\mkern2mu \R(V)}: \R(V) \longrightarrow  \R(V) \otimes \R(V).
\]

The notion of co-Poisson coalgebra was introduced in \cite{CP}. In fact the notion of Co-Poisson Hopf algebra was introduced by Drinfeld in \cite{Drinfeld}. We extend the definition to the $G_\Gamma$-module context.

\begin{definition}\label{def:coPoisson}
A \textbf{$(G_\Gamma, \beta)$-co-Poisson coalgebra of degree} $\gamma \in \Gamma$ is a tuple $(C, \delta, c, \delta_0)$ such that:
\begin{enumerate}[itemsep=5pt]
\item $(C, \delta, c)$ is a $G_\Gamma$-coalgebra.
\item $(C, \delta_0)$ is a $(G_\Gamma, \beta)$-Lie coalgebra of degree $\gamma$.
\item $(\delta \otimes \on{id}_C )  \delta_0 =(\on{id}_C \otimes \delta_0)   \delta-\xi^\beta   (\delta_0 \otimes \on{id}_C)   \delta$.
\end{enumerate}
\end{definition}

We can now prove the next result.

\begin{theorem}\label{thm:C_2_co_Poisson}
Let $V$ be a $(G_\Gamma, \beta, \gamma_0)$-vertex coalgebra. Assume that $\beta$ satisfies Relation \eqref{relation_2}. Then the tuple $(\R(V), \Delta, c_{\R(V)}, \Delta_0)$ is a $\beta$-cocommutative $(G_\Gamma, \beta)$-co-Poisson coalgebra of degree $0$.
\end{theorem}

\begin{remark}
The case $\beta=0$ provides a co-Poisson coalgebra with trivial Lie cobracket.
\end{remark}

\begin{proof}
In the co-Jacobi identity \eqref{eq:coJacobi_components} applied to $v \in \on{Ker}\Y_{-2}$, set $l=1$, $m=n=-1$. We get

\[
\begin{array}{r}
( \Y_{-1} \otimes \on{id}_{V})    \Y_{0}(v)-( \Y_{0} \otimes \on{id}_{V})    \Y_{-1}(v)+( \on{id}_{V} \otimes T^\beta)     (\Y_{-1} \otimes \on{id}_{V})   \Y_{0}(v) \\[5pt]
-( \on{id}_{V} \otimes T^\beta)     (\Y_{0} \otimes \on{id}_{V})    \Y_{-1}(v) =0.
\end{array}
\]
We compose with $(\epsilon \otimes \on{id}_{V} \otimes \on{id}_{V})$ on the left and use the fact that $(\epsilon \otimes \on{id}_V)   \Y_0=0$ and $(\epsilon \otimes \on{id}_V)   \Y_{-1}=\on{id}_V$ coming from the cocreation property \eqref{eq:cocreation} in Definition \ref{def:cova}. this leads us to
\begin{align}\label{thm:C_2_co_Poisson_eq:1}
\Y_0(v)+T^\beta   \Y_0(v)=0.
\end{align}
Thus $\delta_0$ is skew-symmetric.

We now consider the co-Jacobi identity \eqref{eq:coJacobi_components} with $l=m=n=0$:
\begin{align}\label{thm:C_2_co_Poisson_eq:2}
( \Y_{0} \otimes \on{id}_{V})    \Y_{0}-( \on{id}_{V} \otimes T^\beta)     (\Y_{0} \otimes \on{id}_{V})   \Y_{0} = (\on{id}_{V} \otimes \Y_{0})    \Y_{0}
\end{align}
For any $v \in V$, we write $\Y_0(v)=\sum_i v_{1, i} \otimes v_{2, i}$. We compute
\[
\begin{array}{rl}
\xi^\beta   (\on{id}_V & \mkern-15mu \otimes \Y_0)   \Y_0(v)  \\[5pt]
& = \sum_i \xi^\beta   (v_{1, i} \otimes \Y_0(v_{2, i}))\\[5pt]
& = \displaystyle \sum_i \sum_j \xi^\beta (v_{1, i} \otimes (v_{2, i})_{1, j} \otimes (v_{2, i})_{2, j})\\[5pt]
& = \displaystyle \sum_i \sum_j \beta(|v_{1, i}|, |(v_{2, i})_{1, j}|) \beta(|v_{1, i}|,|(v_{2, i})_{2, j}|) ((v_{2, i})_{1, j} \otimes (v_{2, i})_{2, j} \otimes v_{1, i})\\[5pt]
& = \displaystyle \sum_i \sum_j \beta(|v_{1, i}|, |(v_{2, i})_{1, j}|+|(v_{2, i})_{2, j}|) ((v_{2, i})_{1, j} \otimes (v_{2, i})_{2, j} \otimes v_{1, i})\\[5pt]
& = \displaystyle \sum_i \sum_j \beta(|v_{1, i}|, |v_{2, i}|) ((v_{2, i})_{1, j} \otimes (v_{2, i})_{2, j} \otimes v_{1, i})\\[5pt]
& = \displaystyle \sum_i \beta(|v_{1, i}|, |v_{2, i}|) (\Y_0(v_{2, i}) \otimes v_{1, i})\\[5pt]
& = \displaystyle  (\Y_0 \otimes \on{id}_V) \sum_i \beta(|v_{1, i}|, |v_{2, i}|) (v_{2, i} \otimes v_{1, i})\\[5pt]
& = \displaystyle  (\Y_0 \otimes \on{id}_V)   T^\beta   \Y_0(v).
\end{array}
\]
It follows that for any $v \in \on{Ker}\Y_{-2}$, we have $\xi^\beta   (\on{id}_V \otimes \Y_0)   \Y_0(v) =  (\Y_0 \otimes \on{id}_V)   T^\beta   \Y_0(v)= -(\Y_0 \otimes \on{id}_V)    \Y_0(v)$ by Equation \eqref{thm:C_2_co_Poisson_eq:1}. It follows that
\begin{align}\label{thm:C_2_co_Poisson_eq:3}
\xi^\beta   (\on{id}_V \otimes \Y_0)   \Y_0(v) =   -(\Y_0 \otimes \on{id}_V)    \Y_0(v)
\end{align}
and
\begin{align}\label{thm:C_2_co_Poisson_eq:4}
(\xi^\beta)^2   (\on{id}_V \otimes \Y_0)   \Y_0(v) =   -\xi^\beta   (\Y_0 \otimes \on{id}_V)    \Y_0(v)
\end{align}
Moreover,
\begin{align}
 (\on{id}_V \otimes T^\beta)   & (\Y_0 \otimes \on{id}_V)   \Y_0(v)  \nonumber \\[5pt]
& = - (\on{id}_V \otimes T^\beta)    ((T^\beta   \Y_0) \otimes \on{id}_V)   \Y_0(v)   \nonumber \\[5pt]
& = - (\on{id}_V \otimes T^\beta)    (T^\beta \otimes \on{id}_V)   ( \Y_0 \otimes \on{id}_V)   \Y_0(v)   \nonumber \\[5pt]
& = -\xi^\beta   ( \Y_0 \otimes \on{id}_V)   \Y_0(v)   \nonumber \\[5pt]
& =(\xi^\beta)^2   (\on{id}_V \otimes \Y_0)   \Y_0(v)   \quad \text{by Equation \eqref{thm:C_2_co_Poisson_eq:4}}.\label{thm:C_2_co_Poisson_eq:5}
\end{align} 
Then, by replacing the terms in Equation \eqref{thm:C_2_co_Poisson_eq:2} using Equations \eqref{thm:C_2_co_Poisson_eq:3} and \eqref{thm:C_2_co_Poisson_eq:5}, we get
\begin{align}\label{thm:C_2_co_Poisson_eq:6}
(\on{id}_{V^{\otimes 3}}+\xi^\beta+(\xi^\beta)^2)   (\on{id}_V \otimes \Y_0)   \Y_0(v) =0.
\end{align}

From Definition \ref{def:Lie_coalgebra} and Equations \eqref{thm:C_2_co_Poisson_eq:1} and \eqref{thm:C_2_co_Poisson_eq:6}, we see that $(\R(V), \Delta_0)$ is a $(G_\Gamma, \beta)$-Lie coalgebra of degree $0$. We already know from Theorem \ref{thm:C_2_coalgebra} that $(\R(V), \Delta, c_{\R(V)})$ is a $\beta$-cocommutative $G_\Gamma$-coalgebra. It remains to prove the third condition in Definition \ref{def:coPoisson}.

Consider the co-Jacobi identity \eqref{eq:coJacobi_components} with $l=m=0$, $n=-1$:
\begin{align}\label{thm:C_2_co_Poisson_eq:7}
( \Y_{-1} \otimes \on{id}_{V})    \Y_{0}-( \on{id}_{V} \otimes T^\beta)     (\Y_{0} \otimes \on{id}_{V})   \Y_{-1} =(\on{id}_{V} \otimes \Y_{0})    \Y_{-1}.
\end{align}
Moreover, for $v \in \on{Ker}\Y_{-2}$, we have
\begin{align*}
-\xi^\beta   (\Y_0 \otimes \on{id}_V)   \Y_{-1}(v) & =\xi^\beta   ((T^\beta   \Y_0) \otimes \on{id}_V)   \Y_{-1}(v)\\[5pt]
&=\xi^\beta   (T^\beta \otimes \on{id}_V)    (\Y_0 \otimes \on{id}_V)   \Y_{-1}(v) \\[5pt]
&=(\on{id}_V \otimes T^\beta)   ((T^\beta)^2 \otimes \on{id}_V)    (\Y_0 \otimes \on{id}_V)   \Y_{-1}(v) \\[5pt]
&=(\on{id}_V \otimes T^\beta)   (\Y_0 \otimes \on{id}_V)   \Y_{-1}(v)
\end{align*} 
Therefore Equation \eqref{thm:C_2_co_Poisson_eq:7} leads to
\begin{align*}
( \Y_{-1} \otimes \on{id}_{V})    \Y_{0}(v)+\xi^\beta   (\Y_0 \otimes \on{id}_V)   \Y_{-1}(v) =(\on{id}_{V} \otimes \Y_{0})    \Y_{-1}(v),
\end{align*}
which then implies
\begin{align*}
( \Delta \otimes \on{id}_{\R(V)})    \Delta_{0}=(\on{id}_{\R(V)} \otimes \Delta_{0})    \Delta-\xi^\beta   (\Delta_0 \otimes \on{id}_{\R(V)})   \Delta
\end{align*}
and this is exactly the third condition in Definition \ref{def:coPoisson}.
\end{proof}

\subsection{The associated $C_2$-comodules}
Let $V$ be a $(G_\Gamma, \beta, \gamma_0)$-vertex coalgebra and $M$ a $(G_\Gamma, \beta, \gamma_0)$-right $V$-comodule. Consider the space
\[
\R(M)=\on{Ker}\Y^M_{-2}. 
\]
It is a $G_\Gamma$-module because $|\Y^M_{-2}|=-|x|=-2\gamma_0$.

By looking at the coefficient of $x_0^{-l-1}x_1^{-m-1}x_2^{-n-1}$ in the co-Jacobi identity \eqref{eq:coJacobi_mod} of a $V$, we get
\begin{equation}\label{eq:coJacobi_mod_components}
\begin{split}
\scalebox{0.95}{$\displaystyle \sum_{i \geq 0}(-1)^i \binom{l}{i}( \Y^M_{n+i} \otimes \on{id}_{V})    \Y^M_{m+l-i}-(-1)^l\sum_{i \geq 0} (-1)^i \binom{l}{i}( \on{id}_{M} \otimes T^\beta)     (\Y^M_{m+i} \otimes \on{id}_{V})$} \\[5pt]
\scalebox{0.95}{$  \displaystyle   \Y^M_{n+l-i} =\sum_{i \geq 0} \binom{m}{i}(\on{id}_{M} \otimes \Y_{l+i})    \Y^M_{m+n-i}$}.
\end{split}
\end{equation}
We know from \ref{def:coderivation_mod} in Definition \ref{def:cova_mod} that $(\on{id}_M \otimes \D )   \Y^M_n=-n \Y^M_{n-1}$ for all $n \in \mathbb{Z}$, so we have
\begin{align}\label{eq:YY^M_{-2-n}}
\Y^M_{-2-n}=\frac{1}{(n+1)!}(\on{id}_M \otimes \D)^n   \Y^M_{-2}
\end{align}
for all $n \geq 0$. This leads to the following lemma:

\begin{lemma}\label{lem:YY^M_{-2-n}}
Let $V$ be $(G_\Gamma, \beta, \gamma_0)$-vertex coalgebra and $M$ a $(G_\Gamma, \beta, \gamma_0)$-right $V$-comodule. Then $\R(M) \subseteq \on{Ker}\Y^M_{-2-n}$ for all $n \geq 0$. 
\end{lemma}

\begin{definition}\label{def:coalgebra_mod}
Let $(C, \delta, c)$ be a $G_\Gamma$-coalgebra. A $G_\Gamma$\textbf{-right comodule} for $C$ is a $G_\Gamma$-module $M$ with a homogeneous linear map $\delta^M:  M \longrightarrow M \otimes C$ such that
\[
\begin{tikzcd}
M \arrow[d, "\delta^M"']  \arrow[r, "\delta^M"] & M \otimes C \arrow[d, "\on{id}_M \otimes \delta"]  \\[10pt]
M \otimes C \arrow[r, "\delta^M \otimes \on{id}_C", labels=below] & M \otimes C \otimes C
\end{tikzcd} \text{ and } \begin{tikzcd}
M \arrow[dr, "\on{id}_{M}"']  \arrow[r, "\delta^M"] & M \otimes C \arrow[d, "\on{id}_M \otimes c"]  \\[10pt]
&  M
\end{tikzcd}
\]
are commutative diagrams.
\end{definition}

We will need the following lemma:
\begin{lemma}\label{lem:intersection_vector_spaces_2}
Given vector spaces $W_1 \subseteq V_1$ and $W_2 \subseteq V_2$, we have
\[
(W_1 \otimes V_2) \cap (V_1 \otimes W_2)= W_1 \otimes W_2.
\]
\end{lemma}

\begin{proof}
For $\epsilon=1, 2$, consider $(w_i^\epsilon)_{i \in I_\epsilon}$ a basis of $W_\epsilon$ and complete it into $(v_j^\epsilon)_{j \in J_\epsilon}$ a basis of $V$, where $w_i^\epsilon=v_i^\epsilon$ and $I_\epsilon \subseteq J_\epsilon$.

For $x \in W_1 \otimes V_2$, we can write $x=\sum_{i, j}x_{i, j} w_i^1 \otimes v_j^2$, and for $x \in V_1 \otimes W_2$, we can write $x=\sum_{i, j}x_{j, i}' v_j^1 \otimes w_i^2$. So if $x \in (W_1 \otimes V_2) \cap (V_1 \otimes W_2)$, then
\[
\sum_{i, j}x_{i, j} w_i^1 \otimes v_j^2=\sum_{i, j}x_{j, i}' v_j^1 \otimes w_i^2.
\]
Thus, if on the left hand side $x_{i, j} \neq 0$, then $w_i^1 \otimes v_j^2$ should appear on the right hand side, meaning that $v_j^2 \in W_2$. Likewise if $x_{j, i}' \neq 0$.

So the components with non zero coefficients are in $W_1 \otimes W_2$. Hence $(W_1 \otimes V_2) \cap (V_1 \otimes W_2) \subseteq W_1 \otimes W_2$. The other inclusion is obvious.
\end{proof}

The co-Jacobi identity \eqref{eq:coJacobi_mod_components} with $l=0$, $m=-1$, $n=-2$ and with $l=-2$, $m=1$, $n=-2$ lead to 
\[
\left\{\begin{array}{lcl}
(\Y^M_{-2} \otimes \on{id}_{V})   \Y^M_{-1}(\on{Ker}\Y^M_{-2})& = & 0, \\[5pt]
(\on{id}_{M} \otimes \Y_{-2})   \Y^M_{-1}(\on{Ker}\Y^M_{-2})& = & 0.
\end{array}\right.
\]
Set $m \in \on{Ker}\Y^M_{-2}$. We write $\Y^M_{-1}(m)=\sum_{i}m_{i} \otimes v_{i}$ with $\{m_{i}\}$ a linearly independent family. Then $0=(\on{id}_{M} \otimes \Y_{-2})   \Y^M_{-1}(m)=\sum_{i}m_{i} \otimes \Y_{-2}(v_{i})$. But as the $m_{i}$'s are linearly independent, it follows that $\Y_{-2}(v_{i})=0$ for all $i$. Hence $\Y^M_{-1}(\on{Ker}\Y^M_{-2}) \subseteq M \otimes \on{Ker}\Y_{-2}$. By doing the same with $\{v_{i}\}$ linearly independent, we get $\Y^M_{-1}(\on{Ker}\Y^M_{-2}) \subseteq  \on{Ker}\Y^M_{-2} \otimes V$. Thus
\[
\Y^M_{-1}(\on{Ker}\Y^M_{-2}) \subseteq  (M \otimes \on{Ker}\Y_{-2}) \cap (\on{Ker}\Y^M_{-2} \otimes V).
\]
By Lemma \ref{lem:intersection_vector_spaces_2} and the previously obtained inclusion, we get
\begin{align}\label{coalgebra_mod_eq:5}
\Y^M_{-1}(\on{Ker}\Y^M_{-2}) \subseteq  \on{Ker}\Y^M_{-2} \otimes \on{Ker}\Y_{-2}.
\end{align}
Hence there is a homogeneous linear map
\[
\Delta^M=\left.\Y^M_{-1}\right._{\mkern 1mu \vrule height 2ex\mkern2mu \R(M)}: \R(M) \longrightarrow  \R(M) \otimes \R(V).
\] 

We now prove the following result:

\begin{theorem}\label{thm:C_2_coalgebra_mod}
Let $V$ be a $(G_\Gamma, \beta, \gamma_0)$-vertex coalgebra and $M$ a $(G_\Gamma, \beta, \gamma_0)$-right $V$-comodule. Then the tuple $(\R(M), \Delta^M)$ is a $G_\Gamma$-right comodule for the $\beta$-cocommutative $G_\Gamma$-coalgebra $(\R(V), \Delta, c_{\R(V)})$.
\end{theorem}

\begin{proof}
By setting $l=n=-1$, $m=0$, and $m \in \on{Ker}\Y^M_{-2}$ in Equation \eqref{eq:coJacobi_mod_components}, we get
\begin{align}\label{coalgebra_mod_eq:1}
( \Y^M_{-1} \otimes \on{id}_{V})    \Y^M_{-1}(m)=(\on{id}_{M} \otimes \Y_{-1})   \Y^M_{-1}(m).
\end{align}
This proves that the first diagram in Definition \ref{def:coalgebra_mod} is commutative.

We see from the covacuum condition \eqref{eq:covacuum_mod} in Definition \ref{def:cova_mod} that $(\on{id}_M \otimes \epsilon)   \Y^M_{-1}= \on{id}_{M}$, and by quotienting we see that the second diagram in Definition \ref{def:coalgebra_mod} is commutative.
\end{proof}

\begin{definition}\label{def:Lie_coalgebra_mod}
Let $(C, \delta_0)$ be a $(G_\Gamma, \beta)$-Lie coalgebra of degree $\gamma$. A $(G_\Gamma, \beta)$\textbf{-right Lie comodule} for $C$ is a tuple $(M, \delta^M_0)$ where $M$ is a $G_\Gamma$-module (over a field of characteristic not $2$) and $\delta_0^M : M \longrightarrow M \otimes C$ is a linear map of degree $\gamma$ such that
\[
(\on{id}_M \otimes \delta_0)   \delta_0^M=(\delta_0^M \otimes \on{id}_C)  \delta_0^M - (\on{id}_M \otimes T^\beta)  (\delta_0^M \otimes \on{id}_C)   \delta_0^M.
\]
\end{definition}

The co-Jacobi identity \eqref{eq:coJacobi_mod_components} with $l=m=0$, $n=-2$ and with $l=-2$, $m=n=0$ lead to 
\[
\left\{\begin{array}{lcl}
(\Y^M_{-2} \otimes \on{id}_{V})   \Y^M_{0}(\on{Ker}\Y^M_{-2})& = & 0, \\[5pt]
(\on{id}_{M} \otimes \Y_{-2})   \Y^M_{0}(\on{Ker}\Y^M_{-2})& = & 0.
\end{array}\right.
\]
With a reasoning similar to the one used for $\Y^M_{-1}$, we show that
\begin{align}\label{coalgebra_mod_eq:2}
\Y^M_{0}(\on{Ker}\Y^M_{-2}) \subseteq  \on{Ker}\Y^M_{-2} \otimes \on{Ker}\Y_{-2}.
\end{align}
Hence we get a homogeneous linear map 
\[
\Delta^M_0=\left.\Y^M_{0}\right._{\mkern 1mu \vrule height 2ex\mkern2mu \R(M)}: \R(M) \longrightarrow  \R(M) \otimes \R(V).
\] 

By setting $l=m=n=0$ in the co-Jacobi identity \eqref{eq:coJacobi_mod_components} we get:
\begin{align*}
( \Y^M_{0} \otimes \on{id}_{V})    \Y^M_{0}-( \on{id}_{M} \otimes T^\beta)     (\Y^M_{0} \otimes \on{id}_{V})  \Y^M_{0} = (\on{id}_{M} \otimes \Y_{0})    \Y^M_{0},
\end{align*}
and this leads to
\begin{align}\label{eq:coalgebra_eq:2}
\scalebox{0.95}{$( \Delta^M_{0} \otimes \on{id}_{\R(V)})    \Delta^M_{0}-( \on{id}_{\R(M)} \otimes T^\beta)     (\Delta^M_{0} \otimes \on{id}_{\R(V)})  \Delta^M_{0} = (\on{id}_{\R(M)} \otimes \Delta_{0})    \Delta^M_{0}.$}
\end{align}
Using Equation \eqref{eq:coalgebra_eq:2} as well as Definition \ref{def:Lie_coalgebra_mod} and the proof of Theorem \ref{thm:C_2_co_Poisson}, we immediately verify the next result:

\begin{theorem}\label{thm:C_2_Lie_coalgebra_mod}
Let $V$ be a $(G_\Gamma, \beta, \gamma_0)$-vertex coalgebra and $M$ a $(G_\Gamma, \beta, \gamma_0)$-right $V$-comodule. Assume that $\beta$ satisfies Relation \eqref{relation_2}. Then $(\R(M), \Delta_0^M)$ is a $(G_\Gamma, \beta)$-right Lie comodule for the $(G_\Gamma, \beta)$-Lie coalgebra $(\R(V), \Delta_{0})$ of degree $0$.
\end{theorem}

As we introduced the notion of $(\Gamma, \beta)$-co-Poisson coalgebra, we can introduce the corresponding comodule notion:

\begin{definition}\label{def:co-Poisson_mod}
Let $(C, \delta, c, \delta_0)$ be a co-Poisson coalgebra of degree $\gamma$. A $(G_\Gamma, \beta)$-\textbf{right co-Poisson comodule} for $C$ is a triple $(M, \delta^M, \delta^M_0)$ such that $M$ is a $G_\Gamma$-module, $\delta^M: M \longrightarrow M \otimes C$ and $\delta^M_0: M \longrightarrow M \otimes C$ are homogeneous linear maps of degree $0$ and $\gamma$ respectively, and they satisfy
\begin{enumerate}[itemsep=5pt]
\item $(M, \delta^M)$ is a right comodule for the coalgebra $(C, \delta, c)$.

\item $(M, \delta^M_0)$ is a right Lie comodule for the Lie coalgebra $(C, \delta_0)$ of degree $\gamma$.

\item \label{def:coPoissoncomod3} $(\delta^M \otimes \on{id}_C)  \delta^M_0 =(\on{id}_M \otimes \delta_0)   \delta^M+(\on{id}_M \otimes T^\beta)   (\delta^M_0 \otimes \on{id}_C)   \delta^M$.

\item \label{def:coPoissoncomod4} $(\on{id}_M \otimes \delta)  \delta^M_0 =(\delta^M_0 \otimes \on{id}_C)   \delta^M + (\on{id}_M \otimes T^\beta)   (\delta^M_0 \otimes \on{id}_C)   \delta^M$.
\end{enumerate}
\end{definition}

We can now prove the next result.

\begin{theorem}\label{thm:C_2_co_Poisson_mod}
Let $V$ be a $(G_\Gamma, \beta, \gamma_0)$-vertex coalgebra and $M$ a $(G_\Gamma, \beta, \gamma_0)$-right $V$-comodule. Assume that $\beta$ satisfies Relation \eqref{relation_2}. Then the tuple $(\R(M), \Delta^M, \Delta_0^M)$ is a $(G_\Gamma, \beta)$-right co-Poisson comodule for the $\beta$-cocommutative $(G_\Gamma, \beta)$-co-Poisson coalgebra $(\R(V), \Delta, c_{\R(V)} \Delta_0)$ of degree $0$.
\end{theorem}

\begin{proof}
We know that $\R(V)$ is a $\beta$-cocommutative $(G_\Gamma, \beta)$-co-Poisson coalgebra of degree $0$ by Theorem \ref{thm:C_2_co_Poisson}. The first two points in Definition \ref{def:co-Poisson_mod} are verified by Theorems \ref{thm:C_2_coalgebra_mod} and \ref{thm:C_2_Lie_coalgebra_mod}. Then, in the Jacobi identity \eqref{eq:coJacobi_mod_components}, set $l=m=0$, $n=-1$. We get
\begin{align*}
( \Y^M_{-1} \otimes \on{id}_{V})    \Y^M_{0}-( \on{id}_{M} \otimes T^\beta)     (\Y^M_{0} \otimes \on{id}_{V})    \Y^M_{-1} = (\on{id}_{M} \otimes \Y_{0})    \Y^M_{-1},
\end{align*}
and this leads to
\begin{align*}
( \Delta^M\otimes \on{id}_{\R(V)})    \Delta^M_{0}-( \on{id}_{\R(M)} \otimes T^\beta)     (\Delta^M_{0} \otimes \on{id}_{\R(V)})    \Delta^M = (\on{id}_{\R(M)} \otimes \Delta_{0})    \Delta^M,
\end{align*}
which is exactly the third relation in Definition \ref{def:co-Poisson_mod}. The last relation is obtained by setting $l=-1$ and $m=n=0$  in the Jacobi identity \eqref{eq:Jacobi_mod_coeff}.
\end{proof}

\section{Duality for $C_2$-(co)algebras and $C_2$-(co)modules}\label{sec:duality}

\subsection{The algebra-coalgebra correspondence}
We know from Theorem \ref{thm:duality} that Harish-Chandra $(G_\Gamma, \beta, \gamma_0)$-vertex algebras and vertex coalgebra are dual notions. The pairing $\langle \cdot, \cdot \rangle$ between a vector space and its dual will help in establishing a duality between $C_2$-algebras and $C_2$-coalgebras.

Let $V$ be a Harish-Chandra $(G_\Gamma, \beta, \gamma_0)$-vertex algebra and assume that $\beta$ satisfies Relation \eqref{relation_1}. We set $C_2(V)^\perp=\{u' \in V' \ | \ u'(C_2(V))=0  \}$. By the proof of  Theorem \ref{thm:duality}, we know that for any $u' \in V'$ and $v, w \in V$, we have $\langle u', Y_{-2}(v \otimes w) \rangle=\langle \Y_{-2}(u'), v \otimes w \rangle$. It follows that
\[
C_2(V)^\perp=\on{Ker}\Y_{-2}
\]
as $G_\Gamma$-modules. We know that $R(V)=V/C_2(V)=\bigoplus_{\gamma \in \Gamma} V_\gamma/C_2(V)_\gamma$, so $R(V)'=\bigoplus_{\gamma \in \Gamma} C_2(V)_\gamma^\perp$ where $C_2(V)_\gamma^\perp=\{u' \in (V_\gamma)^* \ | \ u'(C_2(V)_\gamma)=0  \}$. We hence see that $C_2(V)^\perp=\bigoplus_{\gamma \in \Gamma} C_2(V)_\gamma^\perp$ because the $(V_\gamma)^*$ are in direct sum, and thus
\begin{align}\label{eq:duality_(co)_algebras_vector_spaces}
R(V)'=C_2(V)^\perp=\on{Ker}\Y_{-2}=\R(V')
\end{align}
as $G_\Gamma$-modules.

We have seen in Theorem \ref{thm:C_2_coalgebra} that $\R(V')$ is a $G_\Gamma$-coalgebra with coproduct given by $\Delta$. As the dual of a coalgebra is always an algebra, we have a product $\ast$ on $\R(V')'$ given by
\begin{align}\label{eq:product_from_coproduct}
\langle f \ast g, u' \rangle=\langle f \otimes g , \Delta(u') \rangle
\end{align}
for all $f, g \in \R(V')'$ and $u' \in \R(V')$. 

\begin{theorem}\label{thm:algebra_iso}
Let $V$ be a Harish-Chandra $(G_\Gamma, \beta, \gamma_0)$-vertex algebra and assume that $\beta$ satisfies Relation \eqref{relation_1}. There is an algebra isomorphism
\begin{align*}
\R(V')' \cong R(V).
\end{align*}
\end{theorem}

\begin{proof}
From Equation \eqref{eq:duality_(co)_algebras_vector_spaces} we see that $\R(V')'=R(V)''=\bigoplus_{\gamma \in \Gamma} (V_n/C_2(V)_n)^{**} \cong \bigoplus_{\gamma \in \Gamma} V_n/C_2(V)_n=R(V)$ as $R(V)$ is Harish-Chandra. So the product $\ast$ is defined on $R(V)$. For $u' \in \R(V')$, we write $\Delta(u')=\sum_i u'_{1, i} \otimes u'_{2, i}$. For any $\overline{u}, \overline{v} \in R(V)$, we have
\begin{align*}
\begin{array}{rcl}
\langle \overline{u} \ast \overline{v}, u' \rangle & = & \langle \overline{u} \otimes \overline{v},  \Delta(u') \rangle\\[5pt]
 & = & \displaystyle \sum_i \overline{u}(u'_{2, i})\overline{v}(u'_{1, i}) \\[5pt]
 & = & \displaystyle \sum_i u_{2, i}'(\overline{u})u_{1, i}'(\overline{v}) \text{ well-defined because }u_{1, i}', u_{2, i}' \in \on{Ker}\Y_{-2}=C_2(V)^\perp \\[5pt]
 & = &\langle  \Y_{-1}(u'), (u+C_2(V)) \otimes (v+C_2(V)) \rangle \\[5pt]
 & = & \langle u', Y_{-1}\big((u+C_2(V)) \otimes (v+C_2(V))\big) \rangle \\[5pt]
 & = & \langle u', Y_{-1}(u \otimes v)+C_2(V)) \text{ as } u'(C_2(V))=0 \text{ and } Y_{-1} \text{ preserves } C_2(V) \\[5pt]
 & = & \langle u', \overline{Y_{-1}(u \otimes v}) \rangle \\[5pt]
 & = & \langle \overline{Y_{-1}(u \otimes v}), u' \rangle \\[5pt]
 & = &\langle \pi(\overline{u} \otimes \overline{v}), u'\rangle.
\end{array}
\end{align*}
It follows that the isomorphism of $G_\Gamma$-modules $\R(V')' \cong R(V)$ preserves the product structures $\ast$ and $\pi$. It remains to check that it also preserves the identity elements.

Using Equation \eqref{eq:product_from_coproduct}, we see that the unit on $\R(V')'$ is $\mathbf{1}_{\R(V')'} : \cc \longrightarrow \R(V')'$ sending $1$ to $c_{\R(V')}$. For any $u' \in \R(V')$, we see from the proof of Theorem \ref{thm:duality} that $\mathbf{1}_{\R(V')'}(1)(u')=c_{\R(V')}(u')=u'(\vac)=(\vac^{**})_{| \R(V')}(u')$. It follows that $\mathbf{1}_{\R(V')'}$ sends $1$ to $(\vac^{**})_{| R(V)'}=(\mathbf{1}_{R(V)}(1))^{**}$. As $R(V)$ is Harish-Chandra, we have the identification $(\mathbf{1}_{R(V)}(1))^{**}=\mathbf{1}_{R(V)}(1)$, and so the unit on $\R(V')'$ is identified with the one of $R(V)$.
\end{proof}

We now assume that $\beta$ satisfies Relation \eqref{relation_2}. We know from Theorem \ref{thm:C_2_co_Poisson} that $\R(V')$ has a co-Poisson structure given by $\Delta_0$. We define a bilinear bracket on $\R(V')'$ by
\[
\begin{array}{cccc}
\{\cdot, \cdot\}: & \R(V')' \otimes \R(V')' & \longrightarrow & \R(V')' \\[5pt]
& f \otimes g & \longmapsto & \begin{array}[t]{cccc}
\{f, g\}: & \R(V')  & \longrightarrow & \cc \\[5pt]
& u' & \longmapsto & \langle f \otimes g,  \Delta_0(u') \rangle.
\end{array} 
\end{array}
\]

\begin{proposition}\label{prop:Lie_algebra}
Let $V$ be a Harish-Chandra $(G_\Gamma, \beta, \gamma_0)$-vertex algebra and assume that $\beta$ satisfies Relations \eqref{relation_1} and \eqref{relation_2}. Then $(\R(V')', \{\cdot, \cdot \})$ is a $(G_\Gamma, \beta)$-Lie algebra of degree $0$.
\end{proposition}

\begin{proof}
We know that $\Delta_0+T^\beta   \Delta_0=0$ by Equation \eqref{thm:C_2_co_Poisson_eq:1}. For $u' \in \R(V')$ we write $\Delta_0(u')=\sum_i u'_{0;1, i} \otimes u'_{0;2, i}$. Therefore
\[
\begin{array}{rcl}
\left(\{\cdot , \cdot\}+\{\cdot, \cdot\}   T^\beta \right)(f \otimes g)(u') & = & \langle f \otimes g+\beta(|f|, |g|)g \otimes f, \Delta_0(u') \rangle\\[5pt]
 & = & \sum_i f(u'_{0;2, i} )g(u'_{0;1, i} )+ \beta(|f|, |g|)\sum_i g(u'_{0;2, i} )f(u'_{0;1, i} ) \\[5pt]
 & = & \sum_i u'_{0;2, i}(f)u'_{0;1, i}(g)+ \beta(|f|, |g|)\sum_i u'_{0;2, i}(g) u'_{0;1, i}(f) \\[5pt]
 &&   \text{ by the identification }V \cong V''. 
\end{array}
\]
If $u'_{0;2, i}(g) u'_{0;1, i}(f) \neq 0$, then $|u'_{0;2, i}|=-|g|$ and $|u'_{0;1, i}|=-|f|$, and so $\beta(|f|, |g|)=\beta(-|u'_{0;1, i}|, -|u'_{0;2, i}|)=\beta(|u'_{0;1, i}|, |u'_{0;2, i}|)$ by the assumption on $\beta$. Hence
\[
\begin{array}{rcl}
\left(\{\cdot , \cdot\}+\{\cdot, \cdot\}   T^\beta \right)(f \otimes g)(u') & = & \sum_i u'_{0;2, i}(f)u'_{0;1, i}(g)+\sum_i \beta(|u'_{0;1, i}|, |u'_{0;2, i}|) u'_{0;2, i}(g) u'_{0;1, i}(f) \\[5pt]
 & = & \Delta_0(f \otimes g)+\sum_i \beta(|u'_{0;1, i}|, |u'_{0;2, i}|) (u'_{0;2, i} \otimes u'_{0;1, i})(f \otimes g) \\[5pt]
&= & (\Delta_0(u')+T^\beta   \Delta_0 )(u')(f \otimes g) \\[5pt]
 & = & 0.
\end{array}
\]
Thus the bracket is skew-symmetric on $\R(V')'$. Based on Definition \ref{def:Lie_algebra}, we need to verify the Jacobi identity. We have
\begin{align*}
&\left(\{\cdot, \cdot \}   (\{\cdot, \cdot\} \otimes \on{id}_{\R(V')'})   (\on{id}_{\R(V')'}+\xi^\beta+(\xi^\beta)^2)\right)( f \otimes g \otimes h)(u') \\[5pt]
&=\{\{f, g\}, h \}(u') +\beta(|f|, |g|)\beta(|f|, |h|)\{\{g, h\}, f \}(u') \\[5pt]
&\quad +\beta(|f|, |g|)\beta(|f|, |h|)\beta(|g|, |h|)\beta(|g|, |f|)\{\{h, f\}, g \} (u')\\[5pt]
&=\displaystyle \sum_i \{f, g\}(u'_{0;2, i})h(u'_{0;1, i})+\beta(|f|, |g|)\beta(|f|, |h|)\sum_i \{g, h\}(u'_{0;2, i})f(u'_{0;1, i})\\[5pt]
&\displaystyle \quad +\beta(|f|, |g|)\beta(|f|, |h|)\beta(|g|, |h|)\beta(|g|, |f|)\sum_i \{h, f\}(u'_{0;2, i})g(u'_{0;1, i}) \\[5pt]
&=\displaystyle \sum_i \sum_j f( (u'_{0;2, i})_{0;2, j}) g((u'_{0;2, i})_{0;1, j})h(u'_{0;1, i}) \\[5pt]
&\displaystyle\quad +\beta(|f|, |g|)\beta(|f|, |h|)\sum_i \sum_j  g( (u'_{0;2, i})_{0;2, j}) h((u'_{0;2, i})_{0;1, j})f(u'_{0;1, i}) \\[5pt]
&\displaystyle \quad +\beta(|f|, |g|)\beta(|f|, |h|)\beta(|g|, |h|)\beta(|g|, |f|)\sum_i \sum_j h( (u'_{0;2, i})_{0;2, j}) f((u'_{0;2, i})_{0;1, j})g(u'_{0;1, i})\\[5pt]
&=\displaystyle \sum_i \sum_j f( (u'_{0;2, i})_{0;2, j}) g((u'_{0;2, i})_{0;1, j})h(u'_{0;1, i}) \\[5pt]
&\displaystyle\quad +\sum_i \sum_j \beta(|u'_{0;1, i}|, |(u'_{0;2, i})_{0;2, j}|)\beta(|u'_{0;1, i}|, |(u'_{0;2, i})_{0;1, j}|) g( (u'_{0;2, i})_{0;2, j}) h((u'_{0;2, i})_{0;1, j})f(u'_{0;1, i}) \\[5pt]
&\displaystyle \quad +\sum_i \sum_j\beta(|(u'_{0;2, i})_{0;1, j}|, |u'_{0;1, i}|)\beta(|(u'_{0;2, i})_{0;1, j}|, |(u'_{0;2, i})_{0;2, j}|)\\[5pt]
&\quad \quad \quad \quad \quad \times \beta(|u'_{0;1, i}|, |(u'_{0;2, i})_{0;2, j}|)\beta(|u'_{0;1, i}|, |(u'_{0;2, i})_{0;1, j}|) h( (u'_{0;2, i})_{0;2, j}) f((u'_{0;2, i})_{0;1, j})g(u'_{0;1, i})\\[5pt]
&=\displaystyle \sum_i \sum_j ( u'_{0;1, i} \otimes (u'_{0;2, i})_{0;1, j} \otimes  (u'_{0;2, i})_{0;2, j}) ) (f \otimes g \otimes h) \\[5pt]
&\displaystyle\quad +\sum_i \sum_j \beta(|u'_{0;1, i}|, |(u'_{0;2, i})_{0;2, j}|)\beta(|u'_{0;1, i}|, |(u'_{0;2, i})_{0;1, j}|) ( (u'_{0;2, i})_{0;1, j}) \otimes  (u'_{0;2, i})_{0;2, j} \otimes u'_{0;1, i}) (f \otimes g \otimes h)\\[5pt]
&\displaystyle \quad +\sum_i \sum_j\beta(|(u'_{0;2, i})_{0;1, j}|, |u'_{0;1, i}|)\beta(|(u'_{0;2, i})_{0;1, j}|, |(u'_{0;2, i})_{0;2, j}|)\beta(|u'_{0;1, i}|, |(u'_{0;2, i})_{0;2, j}|)\beta(|u'_{0;1, i}|, |(u'_{0;2, i})_{0;1, j}|) \\[5pt]
&\quad \quad \quad \quad \quad \times ( (u'_{0;2, i})_{0;2, j}) \otimes u'_{0;1, i} \otimes  (u'_{0;2, i})_{0;1, j}) (f \otimes g \otimes h) \text{ using the identification with the bidual} \\[5pt]
&=\displaystyle \sum_i \sum_j ( u'_{0;1, i} \otimes (u'_{0;2, i})_{0;1, j} \otimes  (u'_{0;2, i})_{0;2, j}) ) (f \otimes g \otimes h) \\[5pt]
&\displaystyle\quad +\left(\xi^\beta\left(\sum_i \sum_j ( u'_{0;1, i} \otimes (u'_{0;2, i})_{0;1, j} \otimes  (u'_{0;2, i})_{0;2, j}) \right)\right) (f \otimes g \otimes h)\\[5pt]
&\displaystyle \quad +\left((\xi^\beta)^2 \left(\sum_i \sum_j ( u'_{0;1, i} \otimes (u'_{0;2, i})_{0;1, j} \otimes  (u'_{0;2, i})_{0;2, j}) )\right)\right) (f \otimes g \otimes h)\\[5pt]
&=\left((\on{id}_{\R(V')^{\otimes 3}} +\xi^\beta+(\xi^\beta)^2)   ( \on{id}_{\R(V')} \otimes \Delta_0)   \Delta_0(u')\right)(f \otimes g \otimes h). \\[5pt]
&=0 \text{ by Equation \eqref{thm:C_2_co_Poisson_eq:6}}.
\end{align*}
It follows that the bracket satisfies the Jacobi identity.
\end{proof}

\begin{proposition}\label{prop:dualPoisson}
Let $V$ be a Harish-Chandra $(G_\Gamma, \beta, \gamma_0)$-vertex algebra and assume that $\beta$ satisfies Relations \eqref{relation_1} and \eqref{relation_2}. Then $(\R(V')', \ast,  \{\cdot, \cdot \})$ is a $\beta$-commutative $(G_\Gamma, \beta)$-Poisson algebra of degree $0$.
\end{proposition}

\begin{proof}
We fix the notation $\Delta_0(u')=\sum_i u'_{0; 1, i} \otimes u'_{0; 2, i}$ and $\Delta(u')=\sum_i u'_{1, i} \otimes u'_{2, i}$ for any $u' \in \R(V')$. Moreover, we write $\pi_\ast$ for the map $\pi_\ast(f \otimes g)=f \ast g$. Then we have
\begin{align*}
&\left(\{\cdot, \cdot\}   ( \on{id}_{\R(V')'} \otimes \pi_\ast) - \pi_\ast   (\{\cdot, \cdot\} \otimes \on{id}_{\R(V')'})+\pi_\ast   (\on{id}_{\R(V')'} \otimes \{\cdot, \cdot\})   \xi^\beta\right)(f \otimes g \otimes h)(u')  \\[5pt]
&=\left(\{f, g \ast h\}-\{f, g\} \ast h+\beta(|f|, |g|)\beta(|f|, |h|)g \ast \{h, f\}\right)(u') \\[5pt]
&=\langle f \otimes (g \ast h), \sum_i u'_{0; 1, i} \otimes u'_{0; 2, i} \rangle - \langle \{f ,g\}\otimes h ,\sum_i u'_{ 1, i} \otimes u'_{ 2, i}\rangle \\[5pt]
& \quad +\beta(|f|, |g|)\beta(|f|, |h|)\langle g \otimes \{h, f\}, \sum_i u'_{ 1, i} \otimes u'_{ 2, i}\rangle, \\[5pt]
&=\sum_i  f(u'_{0; 2, i}) (g \ast h)(u'_{0; 1, i}) -\sum_i  \{f ,g\}(u'_{ 2, i})h(u'_{ 1, i} )+\sum_i  \beta(|f|, |g|)\beta(|f|, |h|)g(u'_{ 2, i})\{h, f\}(u'_{ 1, i}) \\[5pt]
&=\sum_i \sum_j f(u'_{0; 2, i}) g((u'_{0; 1, i})_{2, j}) h((u'_{0; 1, i})_{1, j}) -\sum_i \sum_j f((u'_{ 2, i})_{0; 2, j})g((u'_{ 2, i})_{0; 1, j})h(u'_{ 1, i} ) \\[5pt]
& \quad +\sum_i \sum_j \beta(|f|, |g|)\beta(|f|, |h|)g(u'_{ 2, i})h((u'_{ 1, i})_{0; 2, j}) f((u'_{ 1, i})_{0; 1, j}) \\[5pt]
&=\sum_i \sum_j f(u'_{0; 2, i}) g((u'_{0; 1, i})_{2, j}) h((u'_{0; 1, i})_{1, j}) -\sum_i \sum_j f((u'_{ 2, i})_{0; 2, j})g((u'_{ 2, i})_{0; 1, j})h(u'_{ 1, i} ) \\[5pt]
& \quad +\sum_i \sum_j \beta(|(u'_{ 1, i})_{0; 1, j}|, |u'_{ 2, i}|)\beta(|(u'_{ 1, i})_{0; 1, j}|, |(u'_{ 1, i})_{0; 2, j}|)g(u'_{ 2, i})h((u'_{ 1, i})_{0; 2, j}) f((u'_{ 1, i})_{0; 1, j}) \\[5pt]
&=\left((\Delta \otimes \on{id}_{\R(V')})   \Delta_0\right)(u')(f \otimes g \otimes h)-\left((\on{id}_{\R(V')} \otimes \Delta_0)   \Delta \right)(u')(f \otimes g \otimes h) \\[5pt]
&\quad + \left(\xi^\beta   (\Delta_0 \otimes \on{id}_{\R(V')})   \Delta \right)(u')(f \otimes g \otimes h) \\[5pt]
&=\big[(\Delta \otimes \on{id}_{\R(V')})   \Delta_0-(\on{id}_{\R(V')} \otimes \Delta_0)   \Delta + \xi^\beta   (\Delta_0 \otimes \on{id}_{\R(V')})   \Delta \big](u')(f \otimes g \otimes h), \\[5pt]
&=0 \text{ by Theorem }\ref{thm:C_2_co_Poisson}.
\end{align*}

It follows that
\[
\{\cdot, \cdot\}   ( \on{id}_{\R(V')'} \otimes \pi_\ast) =\pi_\ast   (\{\cdot, \cdot\} \otimes \on{id}_{\R(V')'})-\pi_\ast   (\on{id}_{\R(V')'} \otimes \{\cdot, \cdot\})   \xi^\beta.
\]
With Theorem \ref{thm:algebra_iso}, Proposition \ref{prop:Lie_algebra} and the above equation, we conclude the proof.
\end{proof}

Assume that $\beta$ satisfies Relation \eqref{relation_2}. By Theorem \ref{thm:C_2_Poisson}, $R(V)$ is a $(G_\Gamma, \beta)$-Poisson algebra of degree $0$, so we can compare the Poisson structure $\{\cdot, \cdot\}$ of $\R(V')'$ and the Poisson structure $\pi_0$ of $R(V)$. Using the bidual identification, we have
\begin{longtable}{RCL}
\{\overline{u}, \overline{v}\}(u')& = & \langle \overline{u} \otimes \overline{v}, \Delta_0(u') \rangle\\[5pt]
 & = &\displaystyle  \langle \overline{u} \otimes \overline{v}, \sum_i u'_{0;1, i} \otimes u'_{0; 2, i} \rangle \\[5pt]
 & = & \displaystyle \sum_i u'_{0;1, i}(\overline{v})u'_{0; 2, i}(\overline{u}), \text{ which is well-defined because for } \epsilon=1, 2, \\[5pt]
 &  &  u_{0;\epsilon, i}' \in \on{Ker}\Y_{-2} \text{ so }u_{0;\epsilon,  i}' (C_2(V))=0  \\[5pt]
 & = & \langle \Delta_0(u'), (u+C_2(V)) \otimes (v+C_2(V)) \rangle \\[5pt]
 & = & \langle u', Y_0\left((u+C_2(V) \otimes (v+C_2(V))\right) \rangle  \text{ by Theorem \ref{thm:duality}} \\[5pt] 
 & = & \langle u', Y_0(u \otimes v)+C_2(V) \rangle \text{ by Equations \eqref{eq:Y_0_C_2:1} and \eqref{eq:Y_0_C_2:2}} \\[5pt]
 & = & \langle u', \pi_0(\overline{u} \otimes \overline{v}) \rangle \text{ which is well-defined because }u' \in \on{Ker}\Y_{-2}, \\[5pt]
 & = & \pi_0(\overline{u} \otimes \overline{v})(u').
\end{longtable}
It follows that, through the bidual identification, the Poisson structures of $\R(V')'$ and $R(V)$ are identical. With this, Theorem \ref{thm:algebra_iso}, and Proposition \ref{prop:dualPoisson}, we have therefore proved:

\begin{theorem}\label{thm:Poisson_iso}
Let $V$ be a Harish-Chandra $(G_\Gamma, \beta, \gamma_0)$-vertex algebra and assume that $\beta$ satisfies Relations \eqref{relation_1} and \eqref{relation_2}. There is an isomorphism of $(G_\Gamma, \beta)$-Poisson algebras of degree $0$
\[
\R(V')' \cong R(V).
\]
\end{theorem}

We could have done the same thing but starting from a Harish-Chandra $(G_\Gamma, \beta, \gamma_0)$-vertex coalgebra. This would lead to 

\begin{proposition}\label{prop:dual_co_Poisson}
Let $V$ be a Harish-Chandra $(G_\Gamma, \beta, \gamma_0)$-vertex coalgebra and assume that $\beta$ satisfies Relations \eqref{relation_1} and \eqref{relation_2}. Then $R(V')'$ is a $\beta$-cocommutative $(G_\Gamma, \beta)$-co-Poisson coalgebra of degree $0$.
\end{proposition}

\begin{theorem}\label{thm:co_Poisson_iso}
Let $V$ be a Harish-Chandra $(G_\Gamma, \beta, \gamma_0)$-vertex coalgebra and assume that $\beta$ satisfies Relations \eqref{relation_1} and \eqref{relation_2}. There is an isomorphism of $(G_\Gamma, \beta)$-co-Poisson coalgebras of degree $0$
\[
R(V')' \cong \R(V).
\]
\end{theorem}

\subsection{The module-comodule correspondence}
Let $V$ be a Harish-Chandra $(G_\Gamma, \beta, \gamma_0)$-vertex algebra and $M$ a Harish-Chandra $(G_\Gamma, \beta, \gamma_0)$-left $V$-module. Assume that $\beta$ satisfies Relation \eqref{relation_1}. We set $C_2(M)^\perp=\{w' \in w' \ | \ w'(C_2(M))=0  \}$. But by the proof of  Theorem \ref{thm:duality_mod}, we know that for any $w' \in M'$ and $v \in V$, and $m \in M$, we have $\langle w', Y^M_{-2}(v \otimes m) \rangle=\langle \Y^{M'}_{-2}(w'), v \otimes m \rangle$. It follows that
\[
C_2(M)^\perp=\on{Ker}\Y^{M'}_{-2}
\]
as $G_\Gamma$-modules. We know that $R(M)=M/C_2(M)=\bigoplus_{\gamma \in \Gamma} M_\gamma/C_2(M)_\gamma$, and so $R(M)'=\bigoplus_{\gamma \in \Gamma} (M_\gamma/C_2(M)_\gamma)^*=\bigoplus_{\gamma \in \Gamma} C_2(M)_\gamma^\perp$ where $C_2(M)_\gamma^\perp=\{w' \in (M_\gamma)^* \ | \ w'(C_2(M)_\gamma)=0  \}$. We see that $C_2(M)^\perp=\bigoplus_{\gamma \in \Gamma} C_2(M)_\gamma^\perp$ because the $(M_\gamma)^*$ are in direct sum, and thus
\begin{align}\label{eq:duality_(co)_algebras_vector_spaces_mod}
R(M)'=C_2(M)^\perp=\on{Ker}\Y^{M'}_{-2}=\R(M')
\end{align}
as $G_\Gamma$-modules.

We have seen in Theorem \ref{thm:C_2_coalgebra_mod} that $\R(M')$ is a $G_\Gamma$-right comodule for $\R(V')$ and the coaction is given by $\Delta^{M'}$. We can make $\R(M')'$ into a left module for the algebra $\R(V')'$ as follows:
\begin{align}\label{eq:product_from_coproduct_mod}
\langle f_V \ast_M g_M, m' \rangle=\langle f_V \otimes g_M , \Delta^{M'}(m') \rangle
\end{align}
for all $f_V \in R(V')'$, $g_M \in \R(M')'$ and $m' \in \R(M')$. 

\begin{theorem}\label{thm:algebra_iso_mod}
Let $V$ be a Harish-Chandra $(G_\Gamma, \beta, \gamma_0)$-vertex algebra, $M$ a Harish-Chandra $(G_\Gamma, \beta, \gamma_0)$-left $V$-module, and assume that $\beta$ satisfies Relation \eqref{relation_1}. There is an isomorphism of algebra modules
\begin{align*}
\R(M')' \cong R(M).
\end{align*}
\end{theorem}

\begin{proof}
From Equation \eqref{eq:duality_(co)_algebras_vector_spaces_mod} we see that $\R(M')'=R(M)''=\bigoplus_{\gamma \in \Gamma} (M_n/C_2(M)_n)^{**} \cong \bigoplus_{\gamma \in \Gamma} M_n/C_2(M)_n=R(M)$ as $R(M)$ is Harish-Chandra. So $\ast_M$ defines an action of $R(V)$ on $R(M)$ using the isomorphism in Theorem \ref{thm:Poisson_iso}. For $m' \in \R(M')$, we write $\Delta^{M'}(m')=\sum_i m'_{i} \otimes v'_{i}$. For any $\overline{v} \in R(V)$, $\overline{w} \in R(M)$, we have
\begin{align*}
\begin{array}{rcl}
\langle \overline{v} \ast_M \overline{w}, m' \rangle & = & \langle \overline{v} \otimes \overline{w},  \Delta^{M'}(m') \rangle\\[5pt]
 & = & \displaystyle \sum_i \overline{v}(v'_{i})\overline{w}(m'_{i}) \\[5pt]
 & = & \displaystyle \sum_i v_{i}'(\overline{v})m_{i}'(\overline{w}) \text{ well-defined because } m_{i}' \in C_2(M)^\perp, v_{i}' \in C_2(V)^\perp \\[5pt]
 & = &\langle  \Y_{-1}^{M'}(m'), (v+C_2(V)) \otimes (w+C_2(M)) \rangle \\[5pt]
 & = & \langle m', Y_{-1}^{M}\big((v+C_2(V)) \otimes (w+C_2(M))\big) \rangle \\[5pt]
 & = & \langle m', Y_{-1}^{M}(v \otimes w)+C_2(M)) \text{ because } m'(C_2(M))=0 \\[5pt]
  && \quad  \text{ and } Y_{-1}^{M'} \text{ preserves } C_2(M) \\[5pt]
 & = & \langle m', \overline{Y_{-1}^{M}(v \otimes w}) \rangle \\[5pt]
 & = & \langle \overline{Y_{-1}^{M}(v \otimes w}), m' \rangle \\[5pt]
 & = &\langle \pi^{M}(\overline{v} \otimes \overline{w}), m'\rangle.
\end{array}
\end{align*}
It follows that the isomorphism of $G_\Gamma$-modules $\R(M')' \cong R(M)$ preserves the module structures $\ast_M$ and $\pi^M$.
\end{proof}

We now assume that $\beta$ satisfies Relation \eqref{relation_2}. We know from Theorem \ref{thm:C_2_co_Poisson_mod} that $\R(M')$ is a co-Poisson comodule for the co-Poisson coalgebra $\R(V')$ and the Lie coaction is given by $\Delta^{M'}_0$. We the following bilinear bracket 
\[
\begin{array}{cccc}
\{\cdot, \cdot\}_M: & \R(V')' \otimes \R(M')' & \longrightarrow & \R(M')' \\[5pt]
& f_V \otimes g_M & \longmapsto & \begin{array}[t]{cccc}
\{f_V, g_M\}_M: & \R(M')  & \longrightarrow & \cc \\[5pt]
& m' & \longmapsto & \langle f_V \otimes g_M,  \Delta_0^{M'}(m') \rangle.
\end{array} 
\end{array}
\]

\begin{proposition}\label{prop:Lie_algebra_mod}
Let $V$ be a Harish-Chandra $(G_\Gamma, \beta, \gamma_0)$-vertex algebra, $M$ a $(G_\Gamma, \beta, \gamma_0)$-left $V$-module, and assume that $\beta$ satisfies Relations \eqref{relation_1} and \eqref{relation_2}. Then the tuple $(\R(M')', \{\cdot, \cdot \}_M)$ is a $(G_\Gamma, \beta)$-left Lie module for the $(G_\Gamma, \beta)$-Lie algebra $(\R(V'), \{\cdot, \cdot\})$ of degree $0$.
\end{proposition}

\begin{proof}
We know from Theorem \ref{prop:Lie_algebra} that $(\R(V'), \{\cdot, \cdot\})$ is a $(G_\Gamma, \beta)$-Lie algebra of degree $0$. We need to verify the equation given in Definition \ref{def:Lie_algebra_mod}. Set $\Delta_0(v')=\sum_i v'_{0;i,1} \otimes v'_{0;i,2}$ and $\Delta_0^{M'}(m')=\sum_i m'_{0;i, 1} \otimes m'_{0;i, 2}$ (note that $m'_{0;i, 2} \in \R(V')$). We have
\begin{align*}
&\left[\{\cdot, \cdot \}_M   (\{\cdot, \cdot\} \otimes \on{id}_{\R(M')'}) - \{\cdot, \cdot \}_M   (\on{id}_{\R(V')'} \otimes \{\cdot, \cdot\}_M)\right. \\[5pt]
& \quad +\left.(\on{id}_{\R(V')'} \otimes \{\cdot, \cdot\}_M)   (T^\beta \otimes \on{id}_{\R(M')'})\right]( f_V \otimes g_V \otimes h_M)(m') \\[5pt]
&=\{\{f_V, g_V\}, h_M \}_M(m')-\{f_V, \{g_V, h_M\}_M\}_M(m')+\beta(|f_V|, |g_V|)\{g_V, \{f_V, h_M\}_M\}_M \\[5pt]
&=\langle \{f_V, g_V\} \otimes h_M, \Delta_0^{M'}(m') \rangle-\langle f_V \otimes  \{g_V,  h_M\}_M, \Delta_0^{M'}(m') \rangle+\beta(|f_V|, |g_V|)\langle g_V \otimes  \{f_V,  h_M\}_M, \Delta_0^{M'}(m') \rangle \\[5pt]
&=\sum_i \{f_V, g_V\}(m'_{0;i, 2}) h_M(m'_{0;i, 1})-\sum_i f_V(m'_{0;i, 2}) \{g_V, h_M\}_M(m'_{0;i, 1}) \\[5pt]
&\quad +\beta(|f_V|, |g_V|)\sum_i g_V(m'_{0;i, 2}) \{f_V, h_M\}_M(m'_{0;i, 1}) \\[5pt]
&=\sum_i \sum_j f_V((m'_{0;i, 2})_{0; j, 2} g_V((m'_{0;i, 2})_{0;j, 1}) h_M(m'_{0;i, 1})-\sum_i \sum_j f_V(m'_{0;i, 2}) g_V((m'_{0;i, 1})_{0; j, 2}), h_M((m'_{0;i, 1})_{0; j, 1})\\[5pt]
&\quad +\beta(|f_V|, |g_V|)\sum_i \sum_j g_V(m'_{0;i, 2}) f_V((m'_{0;i, 1})_{0; j, 2}), h_M((m'_{0;i, 1})_{0; j, 1}) \\[5pt]
&=\sum_i \sum_j \left( m'_{0;i, 1} \otimes (m'_{0;i, 2})_{0;j, 1} \otimes (m'_{0;i, 2})_{0; j, 2}\right)(f_V \otimes g_V \otimes h_M) \\[5pt]
& \quad -\sum_i \sum_j \left((m'_{0;i, 1})_{0; j, 1} \otimes (m'_{0;i, 1})_{0; j, 2} \otimes m'_{0;i, 2}\right) (f_V \otimes g_V \otimes h_M)\\[5pt]
&\quad +\sum_i \sum_j \beta(|(m'_{0;i,1})_{0;j,2}|, |m'_{0;i,2}|)\left((m'_{0;i, 1})_{0; j, 1} \otimes  m'_{0;i, 2} \otimes (m'_{0;i, 1})_{0; j, 2}\right) (f_V \otimes g_V \otimes h_M)\\[5pt]
&=(\on{id}_{\R(M')} \otimes \Delta_0 )   \Delta_0^{M'}(m')(f_V \otimes g_V \otimes h_M) \\[5pt]
& \quad - (\Delta_0^{M'} \otimes \on{id}_{\R(V')})   \Delta_0^{M'}(m') (f_V \otimes g_V \otimes h_M)\\[5pt]
&\quad +\left[\sum_i \sum_j (\on{id}_{\R(M')} \otimes T^\beta) \left((m'_{0;i, 1})_{0; j, 1} \otimes  (m'_{0;i, 1})_{0; j, 2} \otimes m'_{0;i, 2} \right)\right] (f_V \otimes g_V \otimes h_M)\\[5pt]
&=(\on{id}_{\R(M')} \otimes \Delta_0 )   \Delta_0^{M'}(m')(f_V \otimes g_V \otimes h_M) - (\Delta_0^{M'} \otimes \on{id}_{\R(V')})   \Delta_0^{M'}(m') (f_V \otimes g_V \otimes h_M)\\[5pt]
&\quad +(\on{id}_{\R(M')} \otimes T^\beta)   (\Delta_0^{M'} \otimes \on{id}_{\R(V')})   \Delta_0^{M'}(m') (f_V \otimes g_V \otimes h_M)\\[5pt]
&=0 \text{ by Equation \eqref{eq:coalgebra_eq:2}}.
\end{align*}
\end{proof}

\begin{proposition}\label{prop:dualPoisson_mod}
Let $V$ be a Harish-Chandra $(G_\Gamma, \beta, \gamma_0)$-vertex algebra, $M$ a Harish-Chandra $(G_\Gamma, \beta, \gamma_0)$-left $V$-module, and assume that $\beta$ satisfies Relations \eqref{relation_1} and \eqref{relation_2}. Then $(\R(M')', \ast_M,  \{\cdot, \cdot \}_M)$ is a $(G_\Gamma, \beta)$-left Poisson module for the $\beta$-commutative $(G_\Gamma, \beta)$-Poisson algebra  $(\R(V')', \ast,  \{\cdot, \cdot \})$ of degree $0$.
\end{proposition}

\begin{proof}
We fix the notation $\Delta_0^{M'}(m')=\sum_i m'_{0,M; i, 1} \otimes m'_{0,M; i, 2}$, $\Delta^{M'}(m')=\sum_i m'_{M; i, 1} \otimes m'_{M; i, 2}$, $\Delta_0(v')=\sum_i v'_{0,V; i, 1} \otimes v'_{0,V; i, 2}$, $\Delta(v')=\sum_i v'_{V; i, 1} \otimes v'_{V; i, 2}$ for any $v' \in \R(V'), m' \in \R(M')$. Moreover, we write $\pi_{\ast_M}$ for the map $\pi_{\ast_M}(f_V \otimes g_M)=f_V \ast_M g_M$. Then we have
\begin{align*}
&\left[\{\cdot, \cdot\}_M   ( \on{id}_{\R(V')'} \otimes \pi_{\ast_M}) - \pi_{\ast_M}   (\{\cdot, \cdot\} \otimes \on{id}_{\R(M')'})\right. \\[5pt]
&-\left.\pi_{\ast_M}   (\on{id}_{\R(V')'} \otimes \{\cdot, \cdot\}_M)   (T^\beta \otimes \on{id}_{\R(M')'})\right](f_V \otimes g_V \otimes h_M)(m')  \\[5pt]
&=\{f_V, g_V \ast_M h_M\}_M (m') - (\{f_V, g_V\} \ast_M h_M) (m') -\beta(|f_V|, |g_V|) (g_V \ast_M  \{f_V, h_M\}_M) (m')  \\[5pt]
&=\sum_i f_V(m'_{0, M;2, i}) (g_V \ast_M h_M)(m'_{0, M;1, i}) \\[5pt]
& \quad - \sum_i (\{f_V, g_V\}(m'_{M; i, 2}) h_M(m'_{M; i, 1})  \\[5pt]
& \quad -\beta(|f_V|, |g_V|) \sum_i g_V(m'_{M; i, 2})   \{f_V, h_M\}_M(m'_{M; i, 1})  \\[5pt]
&=\sum_i \sum_j f_V(m'_{0, M;2, i}) g_V((m'_{0, M;1, i})_{M; j, 2})  h_M)((m'_{0, M;1, i})_{M; j, 1}) \\[5pt]
& \quad - \sum_i \sum_j f_V((m'_{M; i, 2})_{0,V; j, 2}) g_V((m'_{M; i, 2})_{0, V; j, 1}) h_M(m'_{M; i, 1}) \\[5pt]
& \quad -\beta(|f_V|, |g_V|) \sum_i \sum_j g_V(m'_{M; i, 2})   f_V((m'_{M; i, 1})_{0, M; j, 2}) h_M((m'_{M; i, 1})_{0, M; j, 1})  \\[5pt]
&=\sum_i \sum_j ((m'_{0, M;1, i})_{M; j, 1} \otimes (m'_{0, M;1, i})_{M; j, 2} \otimes m'_{0, M;2, i})(f_V \otimes g_V \otimes h_M) \\[5pt]
& \quad - \sum_i \sum_j (m'_{M; i, 1} \otimes (m'_{M; i, 2})_{0, V; j, 1} \otimes (m'_{M; i, 2})_{0,V; j, 2}) (f_V \otimes g_V \otimes h_M) \\[5pt]
& \quad - \sum_i \sum_j \beta(|(m'_{M; i, 1})_{0, M; j, 2}|, |m'_{M; i, 2}|) ((m'_{M; i, 1})_{0, M; j, 1} \otimes m'_{M; i, 2} \otimes (m'_{M; i, 1})_{0, M; j, 2})(f_V \otimes g_V \otimes h_M)  \\[5pt]
&= (\Delta^{M'} \otimes \on{id}_{\R(V')})  \Delta_0^{M'} (m')(f_V \otimes g_V \otimes h_M) \\[5pt]
& \quad - (\on{id}_{\R(M')} \otimes \Delta_0)   \Delta^{M'} (m')(f_V \otimes g_V \otimes h_M) \\[5pt]
&  \quad - \left[(\on{id}_{\R(M')} \otimes T^\beta)   (\Delta_0^{M'} \otimes \on{id}_{\R(V')})   \Delta^{M'}(m')\right](f_V \otimes g_V \otimes h_M)  \\[5pt]
&=0 \text{ by the proof of Theorem }\ref{thm:C_2_co_Poisson_mod}.
\end{align*}
It follows that
\[
\{\cdot, \cdot\}_M   ( \on{id}_{\R(V')'} \otimes \pi_{\ast_M}) - \pi_{\ast_M}   (\{\cdot, \cdot\} \otimes \on{id}_{\R(M')'}) -\pi_{\ast_M}   (\on{id}_{\R(V')'} \otimes \{\cdot, \cdot\}_M)   (T^\beta \otimes \on{id}_{\R(M')'})=0,
\]
and so Equation \eqref{def:Poissonmod3} in Definition \ref{def:Poisson_mod} is satisfied. For the last relation, we see that
\begin{align*}
&\left[\{\cdot, \cdot\}_M   (\pi_\ast \otimes \on{id}_{\R(M')'}) - \pi_{\ast_M}   (\on{id}_{\R(V')'} \otimes \{\cdot, \cdot\}_M)\right. \\[5pt]
&-\left.\pi_{\ast_M}   (\on{id}_{\R(V')'} \otimes \{\cdot, \cdot\}_M)   (T^\beta \otimes \on{id}_{\R(M')'})\right](f_V \otimes g_V \otimes h_M)(m')  \\[5pt]
&=\{f_V \ast g_V, h_M\}_M(m')  - (f_V \ast_M \{g_V, h_M\}_M)(m')- \beta(|f_V|, |g_V|)(g_V \ast_M \{f_V, h_M\}_M)(m') \\[5pt]
&=\sum_i (f_V \ast g_V)(m'_{0, M ;i, 2}) h_M(m'_{0, M ;i, 1})   \\[5pt]
& \quad - \sum_i f_V(m'_{M; i, 2})  \{g_V, h_M\}_M(m'_{M; i, 1}) \\[5pt]
& \quad- \beta(|f_V|, |g_V|) \sum_i  g_V(m'_{M; i, 2})  \{f_V, h_M\}_M(m'_{M; i, 1}) \\[5pt]
&=\sum_i \sum_j f_V((m'_{0, M ;i, 2})_{V; j, 2})g_V((m'_{0, M ;i, 2})_{V; j, 1}) h_M(m'_{0, M ;i, 1})   \\[5pt]
& \quad - \sum_i \sum_j f_V(m'_{M; i, 2})  g_V((m'_{M; i, 1})_{0, M; j, 2}) h_M((m'_{M; i, 1})_{0, M; j, 1}) \\[5pt]
& \quad- \beta(|f_V|, |g_V|) \sum_i \sum_j  g_V(m'_{M; i, 2})  f_V((m'_{M; i, 1})_{0, M; j, 2}) h_M((m'_{M; i, 1})_{0, M; j, 1}) \\[5pt]
&=\sum_i \sum_j (m'_{0, M; i, 1} \otimes (m'_{0, M; i, 2})_{V; j, 1} \otimes   (m'_{0, M; i, 2})_{V; j, 2})(f_V \otimes g_V \otimes h_M) \\[5pt]
& \quad - \sum_i \sum_j ( (m'_{M; i, 1})_{0, M; j, 1} \otimes (m'_{M; i, 1})_{0, M; j, 2} \otimes m'_{M; i, 2})(f_V \otimes g_V \otimes h_M)\\[5pt]
& \quad-  \sum_i \sum_j  \beta(|(m'_{M; i, 1})_{0, M; j, 2}|, |m'_{M; i, 2} |) ((m'_{M; i, 1})_{0, M; j, 1} \otimes m'_{M; i, 2} \otimes (m'_{M; i, 1})_{0, M; j, 2}) (f_V \otimes g_V \otimes h_M) \\[5pt]
&=(\on{id}_{\R(M')} \otimes \Delta)   \Delta^{M'}_0(m')(f_V \otimes g_V \otimes h_M) \\[5pt]
& \quad - (\Delta_0^{M'} \otimes \on{id}_{\R(V')})   \Delta^{M'}(m')(f_V \otimes g_V \otimes h_M) \\[5pt]
& \quad-    [(\on{id}_{\R(M')} \otimes T^\beta) (\Delta_0^{M'} \otimes \on{id}_{\R(V')})   \Delta^{M'}(m')] (f_V \otimes g_V \otimes h_M) \\[5pt]
&=0 \text{ by the proof of Theorem }\ref{thm:C_2_co_Poisson_mod}.
\end{align*}
It follows that
\[
\{\cdot, \cdot\}_M   (\pi_\ast \otimes \on{id}_{\R(M')'}) - \pi_{\ast_M}   (\on{id}_{\R(V')'} \otimes \{\cdot, \cdot\}_M)-\pi_{\ast_M}   (\on{id}_{\R(V')'} \otimes \{\cdot, \cdot\}_M)   (T^\beta \otimes \on{id}_{\R(M')'})=0,
\]
and so Equation \eqref{def:Poissonmod4} in Definition \ref{def:Poisson_mod} is satisfied. With Theorem \ref{thm:algebra_iso_mod}, we conclude the proof.
\end{proof}

Assume that $\beta$ satisfies Relation \eqref{relation_2}. By Theorem \ref{thm:C_2_Poisson_mod}, $R(M)$ is a $(G_\Gamma, \beta)$-left Poisson module for $R(V)$, so using the isomorphism in Theorem \ref{thm:Poisson_iso}, we can compare the Poisson module structure $\{\cdot, \cdot\}_M$ of $\R(M')'$ and the Poisson module structure $\pi^M_0$ of $R(M)$. Using the bidual identification, we have
\begin{longtable}{RCL}
\{\overline{v}, \overline{w}\}_M(m')& = & \langle \overline{v} \otimes \overline{w}, \Delta^{M'}_0(m') \rangle\\[5pt]
 & = &\displaystyle  \langle \overline{v} \otimes \overline{w}, \sum_i m'_{0, M; i, 1} \otimes m'_{0, M; i, 2} \rangle \\[5pt]
 & = & \displaystyle \sum_i m'_{0, M; i, 1}(\overline{w})m'_{0, M; i, 2}(\overline{v})\\[5pt]
 & = & \langle \Delta_0^{M'}(m'), (v+C_2(V)) \otimes (w+C_2(M)) \rangle \\[5pt]
 & = & \langle m', Y_0^{M}\left((v+C_2(V) \otimes (w+C_2(V))\right) \rangle  \text{ by Theorem \ref{thm:duality_mod}} \\[5pt] 
 & = & \langle m', Y_0^M(v \otimes w)+C_2(M) \rangle \text{ by Equations \eqref{eq:Y_0M_C_2:1} and \eqref{eq:Y_0M_C_2:2}} \\[5pt]
 & = & \langle m', \pi_0^M(\overline{v} \otimes \overline{w}) \rangle \text{ which is well-defined because }u' \in \on{Ker}\Y_{-2}, \\[5pt]
 & = & \pi_0^M(\overline{v} \otimes \overline{w})(m').
\end{longtable}
It follows that, through the bidual identification and Theorem \ref{thm:Poisson_iso}, the Poisson module structures of $\R(M')'$ and $R(M)$ are identical. With this, Theorem \ref{thm:algebra_iso_mod}, and Proposition \ref{prop:dualPoisson_mod}, we have therefore proved:
\

\begin{theorem}\label{thm:Poisson_iso_mod}
Let $V$ be a Harish-Chandra $(G_\Gamma, \beta, \gamma_0)$-vertex algebra, $M$ a Harish-Chandra $(G_\Gamma, \beta, \gamma_0)$-left $V$-module, and assume that $\beta$ satisfies Relations \eqref{relation_1} and \eqref{relation_2}. There is an isomorphism of $(G_\Gamma, \beta)$-left Poisson modules
\[
\R(M')' \cong R(M)
\]
where $R(M')'$ is seen as an $R(V)$-module through the isomorphism in Theorem \ref{thm:Poisson_iso}.
\end{theorem}

We could have done the same thing but starting from a Harish-Chandra $(G_\Gamma, \beta, \gamma_0)$-right comodule for a Harish-Chandra $(G_\Gamma, \beta, \gamma_0)$-vertex coalgebra. This would lead to 

\begin{proposition}\label{prop:dual_co_Poisson_mod}
Let $V$ be a Harish-Chandra $(G_\Gamma, \beta, \gamma_0)$-vertex coalgebra, $M$ a Harish-Chandra $(G_\Gamma, \beta, \gamma_0)$-right $V$-comodule, and assume that $\beta$ satisfies Relations \eqref{relation_1} and \eqref{relation_2}. Then $R(M')'$ is a $(G_\Gamma, \beta)$-right co-Poisson comodule for the $\beta$-commutative $(G_\Gamma, \beta)$-Poisson algebra  $\R(V')'$ of degree $0$.
\end{proposition}

\begin{theorem}\label{thm:co_Poisson_iso_mod}
Let $V$ be a Harish-Chandra $(G_\Gamma, \beta, \gamma_0)$-vertex coalgebra, $M$ a Harish-Chandra $(G_\Gamma, \beta, \gamma_0)$-right $V$-comodule,  and assume that $\beta$ satisfies Relations \eqref{relation_1} and \eqref{relation_2}. There is an isomorphism of $(G_\Gamma, \beta)$-co-Poisson comodules
\[
R(M')' \cong \R(M)
\]
where $R(M')'$ is seen as an $\R(V)$-comodule through the isomorphism in Theorem \ref{thm:co_Poisson_iso}.
\end{theorem}

\delete{
{\color{red} 
\section{Final remarks}

$\bullet$ Other functor for the dual where the space does not change.

$\bullet$ Contra-module.

$\bullet$ Zhu algebra (needs a filtration, not a gradation). There should be a Zhu coalgebra.

$\bullet$ dg is an extra version of $\Gamma'=\Gamma \times \mathbb{Z}$ with extra dg condition in the $\mathbb{Z}$-part. Not done here, but later on we will work in a more general context of a monoidal category (\cite{Caradot-Lin-7}).

}
}
\delete{
\section*{Acknowledgments}
This work was initiated while the first author was visiting Kansas State University during the fall of 2023 with the support of the NFSC Grant No. 12250410252. 
}

\end{document}